\documentclass[11pt, twoside]{article}
\usepackage{enumerate}
\usepackage{graphics}
\usepackage{graphicx}
\usepackage{amssymb}
\usepackage{amsmath}
\usepackage{amsthm}
\usepackage{color}
\usepackage{mathrsfs}
\usepackage{anysize}

\usepackage{indentfirst}

\usepackage{txfonts}

\allowdisplaybreaks

\def\rr{{\mathbb R}}
\def\rn{{{\rr}^n}}

\def\nn{{\mathbb N}}

\def\cc{{\mathbb C}}
\def\cs{{\mathcal S}}
\def\cx{{\mathcal X}}

\def\cf{{\mathcal F}}

\def\cl{{\mathcal L}}

\def\cq{{\mathcal Q}}

\def\cm{{\mathcal M}}

\def\fz{\infty}

\def\ext{{\mathop\mathrm{\,ext\,}}}
\def\dist{{\mathop\mathrm{\,dist\,}}}

\def\lz{\lambda}

\def\hs{\hspace{0.3cm}}

\def\ls{\lesssim}
\def\gs{\gtrsim}

\def\gfz{\genfrac{}{}{0pt}{}}

\def\rn{{{\mathbb R}^n}}
\def\rr{{\mathbb R}}
\def\cc{{\mathbb C}}

\def\nn{{\mathbb N}}
\def\cm{{\mathcal M}}

\def\hs{\hspace{0.3cm}}

\def\fz{\infty}

\def\supp{{\mathop\mathrm{\,supp\,}}}
\def\dist{{\mathop\mathrm{\,dist\,}}}
\def\diam{{\mathop\mathrm{\,diam\,}}}
\def\lz{\lambda}

\def\ls{\lesssim}
\def\gs{\gtrsim}

\def\diam{{\mathop\mathrm{\,diam\,}}}

\def\dint{\displaystyle\int}

\def\dsup{\displaystyle\sup}

\def\lf{\left}

\newcommand{\sbt}{{b}_{p,q}^{s,\tau}(\mathbb{R}^n)}
\newcommand{\sft}{{f}_{p,q}^{s,\tau}(\mathbb{R}^n)}

\newcommand{\per}{{\rm per\,}}

\newcommand{\cfi}{{\mathcal{F}}^{-1}}

\def\hs{\hspace{0.3cm}}

\usepackage[T1]{fontenc}
\newtheorem{theorem}{Theorem}[section]
\newtheorem{lemma}[theorem]{Lemma}
\newtheorem{corollary}[theorem]{Corollary}

\newtheorem{proposition}[theorem]{Proposition}

\theoremstyle{definition}
\newtheorem{remark}[theorem]{Remark}
\newtheorem{definition}[theorem]{Definition}
\renewcommand{\appendix}{\par
	\setcounter{section}{0}%
	\setcounter{subsection}{0}%
	\setcounter{subsubsection}{0}%
	\gdef\thesection{\@Alph\c@section}%
	\gdef\thesubsection{\@Alph\c@section.\@arabic\c@subsection}%
	\gdef\theHsection{\@Alph\c@section.}%
	\gdef\theHsubsection{\@Alph\c@section.\@arabic\c@subsection}%
	\csname appendixmore\endcsname
}

\numberwithin{equation}{section}

\begin{document}
	\title{\bf\Large Regularity of Characteristic Functions
in Besov-Type and Triebel--Lizorkin-Type Spaces
	\footnotetext{\hspace{-0.35cm} 2020 {\it
			Mathematics Subject Classification}. Primary 46E35;
		Secondary   42C15.
		\endgraf {\it Key words and phrases.}   Besov(-type) space,
  Triebel--Lizorkin(-type) space,
	characteristic function, thick domain, weakly exterior thick domain,
orthonormal Haar-wavelet system.
		\endgraf  This project is partially supported by the National
Natural Science Foundation of China (Grant Nos. 12122102, 12431006 and 12371093)
and the Fundamental Research Funds
for the Central Universities (Grant No. 2233300008).}}
\author{Wen Yuan, Winfried Sickel\footnote{Corresponding author,
E-mail: \texttt{winfried.sickel@uni-jena.de}/{\color{red} November 21,
2024}/Revised version.}\ \ and Dachun Yang}
\date{}
\maketitle

\vspace{-0.8cm}

\begin{center}
\begin{minipage}{13cm}
{\small {\bf Abstract}\quad
In this article, the authors determine the optimal regularity of characteristic functions in  Besov-type and Triebel--Lizorkin-type spaces under
restrictions on the measure of the $\delta$-neighborhoods of the boundary. In particular,
the  necessary and sufficient
conditions for the membership in
these spaces of characteristic functions of  the snowflake domain and also
some  spiral type domains are obtained.
}
\end{minipage}
\end{center}

%\vspace{0.2cm}
%
%\tableofcontents
%
%\vspace{0.2cm}

%&&&&&&&&&&&&&&&&&&&&&&&&&&&&&&&&&&&&&&&&&&&&&&&&&&&&&&&&&&
%&&&&&&&&&&&&&&&&&&&&&&&&&&&&&&&&&&&&&&&&&&&&&&&&&&&&&&&&&&

\section{Introduction}

%&&&&&&&&&&&&&&&&&&&&&&&&&&&&&&&&&&&&&&&&&&&&&&&&&&&&&&&&&&
%&&&&&&&&&&&&&&&&&&&&&&&&&&&&&&&&&&&&&&&&&&&&&&&&&&&&&&&&&&

The study on the regularity of  characteristic functions
is known to be an important part of the pointwise multiplier
problem for function spaces (see, for instance, \cite{Gu1,Gu2,MS09,RS,t06,LSYY24})
and  has also  found
several applications in partial differential equations such as
the Calder\'on inverse problem (see \cite{FR}).
Note that, for any $p\in[1,\infty]$,
the characteristic  function ${\bf 1}_E$ of any measurable set $E\subset \mathbb{R}^n$, $0 < |E|<\infty$,
belongs to  $L^p(\mathbb{R}^n)$ but does not belong to the Sobolev space $W^{1,p}(\mathbb{R}^n)$. Thus, compared with Lebesgue spaces and Sobolev spaces,
Besov and Triebel--Lizorkin spaces are more natural
spaces to measure the regularity of  characteristic functions
because they have more elaborate structure.
It is well known that the characteristic functions ${\bf 1}_Q, {\bf 1}_B$ of a cube $Q$
or a ball $B$ in $\mathbb{R}^n$ belong to all Besov spaces $B^{1/p}_{p,\infty}(\mathbb{R}^n)$, $p\in[1,\infty]$.
Additionally they belong to $BV (\mathbb{R}^n) \cap L^1 (\mathbb{R}^n)$, which is in case $n=1$ an improvement and otherwise an independent statement.
These results are sharp in the following sense:
there is no other Besov space $B^s_{p,q} (\mathbb{R}^n)$ or Triebel--Lizorkin space $F^s_{p,q} (\mathbb{R}^n)$ such that
\[
 {\bf 1}_Q, {\bf 1}_B \in A^s_{p,q} (\mathbb{R}^n), \quad A \in \{B,F\}\quad\mbox{and}\quad
 A^s_{p,q} (\mathbb{R}^n) \hookrightarrow B^{1/p}_{p, \infty}(\mathbb{R}^n).
\]
In addition we know that
$
 {\bf 1}_Q, {\bf 1}_B \not\in \mathring{B}^{1/p}_{p,\infty}(\mathbb{R}^n),
$
where $\mathring{B}^{1/p}_{p,\infty} (\mathbb{R}^n)$ denotes the closure of Schwartz functions in ${B}^{1/p}_{p,\infty} (\mathbb{R}^n)$.
Now we switch to the Morrey generalizations of these spaces denoted by
$A^{s,\tau}_{p,q} (\mathbb{R}^n)$, $A\in \{B,F\} $.
Recall that $A^{s,0}_{p,q} (\mathbb{R}^n) = A^s_{p,q}(\mathbb{R}^n)$ with $A\in \{B,F\} $.
As general references for $B^{s,\tau}_{p,q}(\mathbb{R}^n)$ and $F^{s,\tau}_{p,q}(\mathbb{R}^n)$ the booklet by the authors \cite{ysy},
the books by Triebel \cite{t13b,t14}, and the survey by Haroske and Triebel \cite{HT23} can be used;
see also \cite{s011,s011a,yy1,yy2,YHSY,HMS16}.
Locally these so-called   Besov-type and Triebel--Lizorkin-type spaces are monotone with respect to
the Morrey parameter $\tau$, that is,
$f \in A^{s,\tau_0}_{p,q} (\mathbb{R}^n)$ with $\supp f$ compact  implies
$f \in A^{s,\tau_1}_{p,q} (\mathbb{R}^n)$ for all $\tau_0,\tau_1\in [0,\infty)$ satisfying $\tau_1 < \tau_0$.
It has been a surprise for us that
\[
 {\bf 1}_Q,   {\bf 1}_B \in B^{1/p,\tau}_{p,\infty} (\mathbb{R}^n), \qquad
 p\in[1, \infty), \quad \tau\in \left[0, \frac{n-1}{np}\right],
\]
see Propositions \ref{charact1} and \ref{charact2}.
In view of the just mentioned monotonicity property this is a real improvement. Next we go one step further.
For $\tau=0$ we know that   all the characteristic functions
having the same smoothness as ${\bf 1}_Q$ or ${\bf 1}_B$ are those  belonging to $BV (\mathbb{R}^n) \cap L^1 (\mathbb{R}^n)$.
For any $\tau\in(0,\infty)$ this is no longer true.
There exists a compactly supported characteristic function  ${\bf 1}_E$
such that
\[
 {\bf 1}_E \in B^{1/p}_{p,\infty}(\mathbb{R}^n) = B^{1/p,0}_{p,\infty} (\mathbb{R}^n)  \quad \mbox{for\ all}\  p\in[1, \infty],
\]
and
\[
 {\bf 1}_E \not\in B^{1/p, \tau}_{p,\infty}(\mathbb{R}^n) \quad
 \mbox{for all}\   p \in [1,\infty)\
 \mbox{and all}\  \tau\in(0,\infty),
\]
see Theorems \ref{spirale0} and \ref{spirale1} below.
To make this a bit more transparent we introduce the notion of
the maximal regularity for characteristic functions, see Definition \ref{MAX}.
We have found two sufficient conditions for sets $E \subset \mathbb{R}^n$
such that the associated characteristic function has the maximal regularity.
One is using the wavelet coefficients from ${\bf 1}_E$, the other is in terms of estimates of the first order differences
\[
 \Delta_h^1 {\bf 1}_E (x):= {\bf 1}_E (x+h)-{\bf 1}_E (x), \quad x,h \in \mathbb{R}^n,
 \: |h|\le 1.
\]
It would be desirable to derive a characterization as in case of $\tau=0$, see Proposition \ref{prima0}.

As a second main topic we have studied also sets $E$ with ${\bf 1}_E$
not being of the maximal regularity.
By restricting to bounded weakly thick domains $E$
we have been able to derive necessary and sufficient conditions
in terms of the Lebesgue measure of the sets
\[
(\partial E)^\delta :=  \left\{x \in \mathbb{R}^n:~ {\mathrm{dist}}\,(x,\partial E)< \delta\right\},   \quad   \delta \in(0, 1),
\]
and guaranteeing  ${\bf 1}_E \in F^{s,\tau}_{p,q}(\mathbb{R}^n)$.
Our main example here is the snowflake domain. Also of certain interest in this context is a family of spiral-type domains.
In both cases we have been able to describe the necessary and sufficient conditions for the
membership in $F^{s,\tau}_{p,q}(\mathbb{R}^n)$.

The above described results make clear that
characteristic functions of open sets might have quite different regularity in the framework of type spaces but being of an identical regularity in the framework of classical Besov and Triebel--Lizorkin spaces.
For $\tau =0$ most of the results are known.
We refer to
Strichartz \cite{Str}, Jaffard and Meyer \cite{JM}, Runst and  Sickel \cite{RS},
Triebel \cite{t03,t06,t20}, Faraco and Rogers \cite{FR},
Schneider and Vybiral \cite{SV},  Lebedev \cite{Leb2013},
Ko and Lee \cite{KL2017},
Haroske and Triebel  \cite{HT23}, and Sickel \cite{s99,s99b,Si21,Si23}.

The use of the refined smoothness classes (Besov-type and Triebel--Lizorkin-type spaces) leads to some new insights. E.g., maximal regularity for Besov spaces will have a different meaning than maximal regularity for Besov-type spaces and this will have applications with respect to the pointwise multiplier problem for Besov spaces, see Theorem 4.14 below.
But also the other way around makes sense.
In general, the study of simple functions and their regularity properties within
the classes $B^{s,\tau}_{p,q}(\mathbb{R}^n)$ and  $F^{s,\tau}_{p,q}(\mathbb{R}^n)$ leads to a better understanding of the role of the four parameters $s,p,q$, and $\tau$.
Of particular interest here is the role of $\tau$, the Morrey parameter.
Also twenty years after the introduction of those spaces its influence is not well understood.
Here and later we shall concentrate   on
 the case of Banach spaces, i.e., to values $p,q \in [1, \infty]$ ($p<\infty$ in the $F$-case). Most often we restrict $\tau$ to the interval $[0,1/p]$.
This will lead to some simplifications of the results and proofs.
In addition we shall restrict $s$ to  positive values.

Recall that for $s>0$ the Besov spaces $ B^{s}_{\infty,\infty} (\mathbb{R}^n)$ belong to the scale of H\"older--Zygmund spaces.
 The following chain of embeddings is known since  a few years:
\begin{equation}\label{new1}
A^{s +n\tau, {\rm loc}}_{p,q} (\mathbb{R}^n) \subset
A^{s,\tau}_{p,q} (\mathbb{R}^n) \hookrightarrow B^{s+ n\tau- n/p}_{\infty,\infty} (\mathbb{R}^n),
\end{equation}
see Proposition \ref{embc} below. Here $A^{s + n\tau, {\rm loc}}_{p,q} (\mathbb{R}^n)$ is the collection of all tempered distributions $f$ such that the products $f\cdot \varphi$ belong to $A^{s + n\tau}_{p,q} (\mathbb{R}^n)$ for any $\varphi \in C_{\rm c}^\infty (\mathbb{R}^n)$.
Hence, the classical H\"older--Zygmund regularity of a function increases with $\tau$ or more exactly, $s$ and $n\tau$
have the same influence on this type of smoothness.
Because of
$
A^{s +n\tau}_{p,q} (\mathbb{R}^n) \hookrightarrow B^{s+ n\tau- n/p}_{\infty,\infty} (\mathbb{R}^n)
$
one could conjecture that $n\tau$, at least locally, contributes to the
smoothness as $s$ itself also in a wider sense
(now including spaces with a $p$-dependence).
One may even ask whether it is possible to replace $A^{s +n\tau}_{p,q} (\mathbb{R}^n)$ locally by
the larger space $A^{s,\tau}_{p,q} (\mathbb{R}^n) $ in certain problems.
We know at least one positive  example for such a statement.
For $p\in[1, \infty)$ and $n=1$ the Besov space $B^{s+\tau}_{p,p} (\mathbb{R})$ allows a characterization in terms of  Haar wavelets like in  Proposition \ref{haarprop} if and only if
the locally larger space $B^{s,\tau}_{p,p} (\mathbb{R})$ allows a characterization in terms of  Haar wavelets, see \cite{syyz}.

No such general relation will be true in case of the regularity of characteristic functions.
It will be enough to consider  the regularity of ${\bf 1}_Q$ and ${\bf 1}_B$. In Propositions \ref{charact1} and \ref{charact2}
we show the following. Let  $p\in[1, \infty)$.
\begin{itemize}
\item If $\tau =0$, then ${\bf 1}_Q, {\bf 1}_B \in B^{1/p,0}_{p,\infty} (\mathbb{R}^n)$.
\item If $\tau \in(0, \frac{n-1}{np}]$, then ${\bf 1}_Q, {\bf 1}_B \in
B^{1/p,\tau}_{p,\infty} (\mathbb{R}^n)$.
\item  If $\tau\in(\frac{n-1}{np},\infty)$, then
${\bf 1}_Q, {\bf 1}_B  \in B^{s,\tau}_{p,1} (\mathbb{R}^n)$ with  $s=n   (\frac 1p - \tau )$.
\end{itemize}
Because of $ {\bf 1}_Q, {\bf 1}_B \not\in B^{\frac 1p,0}_{p,q} (\mathbb{R}^n)$, $q< \infty$ (see Proposition \ref{gul}),
the fact that ${\bf 1}_Q, {\bf 1}_B \in
B^{\frac1p,\frac{n-1}{np}}_{p,\infty} (\mathbb{R}^n)$ will not imply   ${\bf 1}_Q, {\bf 1}_B \in
B^{\frac1p + \frac{n-1}{p}}_{p,\infty} (\mathbb{R}^n)$, even
 ${\bf 1}_Q, {\bf 1}_B \in
B^{\frac1p + \varepsilon}_{p,\infty} (\mathbb{R}^n)$ is wrong for any $\varepsilon >0$.
In case of these two characteristic functions it is therefore obvious that the Morrey parameter $\tau$ will not improve the smoothness in the scale of classical Besov spaces itself.

In Theorems \ref{spirale0} and \ref{spirale1} we
construct  a bounded spiral-type domain $E \subset \mathbb{R}^2$ such that ${\bf 1}_E \in
B^{s,\tau}_{p, \infty} (\mathbb{R}^2)$ if and only
if $\tau + s \le 1/p$. Hence, the parameter $s$ is
dominated by the difference $\frac 1p - \tau$. In case of
${\bf 1}_Q, {\bf 1}_B$ with $n=2$ and $\tau < 1/p$ the parameter $s$ is dominated by
$2(\frac 1p - \tau)$. But in case $\tau =0$ we know that all these domains belong to $B^{\frac1p}_{p,\infty} (\mathbb{R}^2)$.
So in view of the Besov-type spaces with $0 < \tau < 1/p$
the regularity of ${\bf 1}_E$ is much less than that of
${\bf 1}_Q, {\bf 1}_B$, but in the framework of Besov spaces
they have identical regularity.

Finally, in  Theorems \ref{spiral4} and \ref{spiral5} we discuss  a family $\{E_\alpha\}_{\alpha\in(0,\infty)}$ of bounded spiral-type domains. Concerning the regularity we will show that  the following three
cases for $  \tau \in[0,1/p]$ will occur:
\begin{itemize}
 \item  Let $\alpha \in(0, 1)$. Then
	${\bf 1}_{E_\alpha} \in B^{s,\tau}_{p, \infty} (\mathbb{R}^2)$  if and only if
		$s\le \frac{2\alpha}{\alpha+1} \Big(\frac 1p - \tau\Big)$.
\item Let $\alpha =1$. Then ${\bf 1}_{E_\alpha} \in B^{s,\tau}_{p, \infty} (\mathbb{R}^2)$ if and only if $s \le \frac1p -\tau$ and $s<\frac1p$.
\item Let $\alpha \in(1,\infty)$. Then
	${\bf 1}_{E_\alpha} \in B^{s,\tau}_{p, \infty} (\mathbb{R}^2)$  if and only if
		$s\le \frac{2\alpha}{\alpha+1} \Big(\frac 1p - \tau\Big)$ and $s\le 1/p$.
\end{itemize}
As consequences we obtain
\begin{itemize}
	\item The smoothness of  ${\bf 1}_{E_\alpha}$ increases with $\alpha$.
	\item For any $\varepsilon \in(0,\infty)$ and any $\tau \in (0,1/p)$ there exists  some $\alpha \in (0,\infty)$
	such that ${\bf 1}_{E_\alpha} \not\in B^{s,\tau}_{p, \infty} (\mathbb{R}^2).$
	\item Let us fix $p \in [1,\infty)$. We assume that $\tau \in (0,\frac 1p)$ and $s\in (0, 2(\frac 1p -\tau))$.
Then there exists some $\alpha \in (0,\infty)$ such that
\[
{\bf 1}_{E_\alpha} \in B^{s,\tau}_{p, \infty} (\mathbb{R}^2) \qquad \mbox{and}\qquad  {\bf 1}_{E_\alpha} \not\in
\bigcup_{\gfz{\varepsilon,\delta \ge 0}{\varepsilon + \delta >0}}\, B^{s+\varepsilon,\tau+\delta}_{p, \infty} (\mathbb{R}^2)\, .
\]
\item
There is only a rather vague summary. In all examples we have seen so far, there is a relation of the form
$
s+ \beta \tau = f(p)
$
to determine the best possible smoothness within type spaces.
Here $\beta \in [0,n]$ is a factor and $f:~[1,\infty) \mapsto (0,\infty) $ a continuous function and both  may depend on the concrete domain $E$.
If $\beta$ is positive, then increasing $s$ means decreasing $\tau$ and vice versa. If we choose $s=1/p$ and $\tau \in (0,\infty)$, then
a nontrivial characteristic function belonging to $A^{s,\tau}_{p,\infty}(\rn)$ will never be an element of $A^{s+n\tau}_{p,\infty} (\rn)$, see \eqref{new1}.
\end{itemize}

All these assertions will be supplemented by results involving the spaces $F^{s,\tau}_{p, q} (\mathbb{R}^n)$.
As said before, we shall also treat the regularity of the characteristic function of the snowflake domain, probably the most famous fractal.

We would like to   mention that our main motivation,
to study the regularity of characteristic functions
in Besov-type and Triebel--Lizorkin-type spaces, originates from some pointwise multiplier problems for Besov spaces, where Besov-type spaces
play an essential role.
This will be explained in detail in Section \ref{multi} below.

There is another very simple application.
Haar wavelet characterizations of Besov and Triebel-Lizorkin  spaces represent an interesting  and well studied subject (see, for instance, \cite{t10,Os81,Os18,Os21,GSU18,GSU21,GSU23}).
To find the counterpart of those assertions for type spaces
remains a challenging mainly open question. A partial answer can be found in \cite[Theorem 2.1]{ysy20}, see also \cite{HT23}.
There we have proved that,   for any $s\in\mathbb{R}$ and $p,q\in(0,\infty]$,
an arbitrary element of the  Haar wavelet system  belongs to $B^{s,\tau}_{p,q}(\mathbb{R}^n)$ if and only if
${\bf 1}_{[0,1)^n}\in B^{s,\tau}_{p,q}(\mathbb{R}^n)$ which immediately carries over in an upper bound in $s$ (depending on  $\tau, p$ and $q$) for
the validity of those characterizations.
For positive results with respect to Haar wavelet characterizations of Besov-type spaces
we refer to Triebel \cite[Theorem 3.4.1]{t14}, see Proposition
\ref{haarprop} below, and to Sickel et al \cite{syyz}.

To recall the definitions  of Besov, Besov-type, Triebel--Lizorkin, and
Triebel--Lizorkin-type spaces,  let $\varphi_0$, $\varphi\in\mathcal{S}(\mathbb{R}^n)$ be such that
$ \mathcal{F}{\varphi}_0,  \mathcal{F}{\varphi}$ are real-valued, nonnegative, satisfying
\begin{align}\label{e1.0}
	\supp \mathcal{F}{\varphi}_0\subset \{\xi\in\mathbb{R}^n:\,|\xi|\le2\}\quad\mathrm{and}\quad
	|\mathcal{F}{\varphi}_0(\xi)|\ge C>0\hs\mathrm{if}\hs |\xi|\le \frac53
\end{align}
and
\begin{align}\label{e1.1}
	\supp
	\mathcal{F} {\varphi}\subset \left\{\xi\in\mathbb{R}^n:\,\frac12\le|\xi|\le2\right\} \ \ \mathrm{and}\ \
	|\mathcal{F} {\varphi}(\xi)|\ge C>0\ \ \mathrm{if}\ \ \frac35\le|\xi|\le \frac 53,
\end{align}
where $C$ is a positive constant independent of $\varphi_0$ and $\varphi$.
Observe that there exist positive constants $A$ and $B$ such that
\[
A \le
\mathcal{F} {\varphi_0}(\xi) +
\sum_{j=1}^\infty \mathcal{F} {\varphi} (2^{-j}\xi) \le B \quad \mbox{for any}\quad \xi \in \mathbb{R}^n.
\]
In what follows, for any $j\in\mathbb{N}$, we let $\varphi_j(\cdot):=2^{jn}\varphi(2^j\cdot)$.

For any given $j\in\mathbb{Z}$ and $k\in\mathbb{Z}^n$,
denote by $Q_{j,k}$ the \emph{dyadic cube} $2^{-j}([0,\,1)^n+k)$
and $\ell(Q_{j,k})$
its \emph{edge length}. Let
$$\mathcal{Q}:=\left\{Q_{j,k}:\,j\in\mathbb{Z},
\,k\in\mathbb{Z}^n\right\},\quad \mathcal{Q}^\ast:=\left\{Q\in\mathcal{Q}:\ell(Q)\le1\right\},$$
and $j_Q:=-\log_2\ell(Q)$ for any $Q\in\mathcal{Q}$. For any $j\in\mathbb{Z}$, let
$\mathcal{Q}_{j}$ be the collection of all dyadic cubes with edge length $2^{-j}$.

\begin{definition}\label{d1}
	Let $s\in\mathbb{R}$, $\tau\in[0,\infty)$, $p,$ $q \in(0,\infty]$, and $\varphi_0$, $\varphi\in\mathcal{S}(\mathbb{R}^n)$ be
	the same as in \eqref{e1.0} and \eqref{e1.1}, respectively.
\\
{\rm (i)}	
	The \emph{Besov space} $B^{s}_{p,q}(\mathbb{R}^n)$ is defined as the space of all $f\in \mathcal{S}'(\mathbb{R}^n)$ such that
$$\|f\|_{B^{s}_{p,q}(\mathbb{R}^n)}:= \left\{\sum_{j=0}^\infty
2^{js q}\left[\int_{\mathbb{R}^n}
|\varphi_j\ast f(x)|^p\,dx\right]^{\frac qp}\right\}^{\frac1q}<\infty$$
with the usual modifications made if $p=\infty$ and/or $q=\infty$.
\\
{\rm (ii)}	
The \emph{Besov-type space} $B^{s,\tau}_{p,q}(\mathbb{R}^n)$ is defined as the space of all $f\in \mathcal{S}'(\mathbb{R}^n)$ such that
	$$\|f\|_{B^{s,\tau}_{p,q}(\mathbb{R}^n)}:=
\sup_{P\in\mathcal{Q}}\frac1{|P|^{\tau}}\left\{\sum_{j=\max\{j_P,0\}}^\infty
2^{js q}\left[\int_P
|\varphi_j\ast f(x)|^p\,dx\right]^{\frac qp}\right\}^{\frac1q}<\infty$$
with the usual modifications made if $p=\infty$ and/or $q=\infty$.
\end{definition}

\begin{definition}\label{d1b}
	Let $s\in\mathbb{R}$, $\tau\in[0,\infty)$, $p \in(0,\infty)$, $q \in(0,\infty]$, and $\varphi_0$, $\varphi\in\mathcal{S}(\mathbb{R}^n)$ be
	the same as in \eqref{e1.0} and \eqref{e1.1}, respectively.
\\
{\rm (i)}	
	The \emph{Triebel--Lizorkin space} $F^{s}_{p,q}(\mathbb{R}^n)$ is defined as the space of all $f\in \mathcal{S}'(\mathbb{R}^n)$ such that
	$$\|f\|_{F^{s}_{p,q}(\mathbb{R}^n)}:=
\left\{\int_{\mathbb{R}^n)} \left[\sum_{j=0}^\infty
	2^{js q} 	|\varphi_j\ast f(x)|^q \right]^{\frac pq}dx \right\}^{\frac1p}<\infty$$
	with the usual modification made if $q=\infty$.\\
{\rm (ii)}	
	The \emph{Triebel--Lizorkin-type space} $F^{s,\tau}_{p,q}(\mathbb{R}^n)$ is defined as the space of all $f\in \mathcal{S}'(\mathbb{R}^n)$ such that
	$$\|f\|_{F^{s,\tau}_{p,q}(\mathbb{R}^n)}:=
	\sup_{P\in\mathcal{Q}}\frac1{|P|^{\tau}}\left\{\int_P \left[\sum_{j=\max\{j_P,0\}}^\infty
	2^{js q} 	|\varphi_j\ast f(x)|^q \right]^{\frac pq}dx \right\}^{\frac1p}<\infty$$
	with the usual modification made if $q=\infty$.
\end{definition}

It is easily seen that we have $B^{s,0}_{p,q}(\mathbb{R}^n)  = B^{s}_{p,q}(\mathbb{R}^n)$ and
$F^{s,0}_{p,q}(\mathbb{R}^n)  = F^{s}_{p,q}(\mathbb{R}^n)$
with coincidence of the quasi-norms.

If there is no need to distinguish $B^{s,\tau}_{p,q}(\mathbb{R}^n)$ and $F^{s,\tau}_{p,q}(\mathbb{R}^n)$ we often    write $A^{s,\tau}_{p,q}(\mathbb{R}^n)$ instead.
For the reader the following characterization of $A^{s,\tau}_{p,q}(\mathbb{R}^n)$, due to Triebel   \cite[Theorem 3.64]{t14}, may be helpful for the understanding of the philosophy of these spaces. Under some extra conditions on $s,\tau,p$, and $q$, it holds that $f\in \mathcal{S}'(\mathbb{R}^n)$ belongs to $A^{s,\tau}_{p,q}(\mathbb{R}^n)$ if and only if
\begin{equation}\label{new4}
 \|   f  \|^*_{A^{s,\tau}_{p,q} (\mathbb{R}^n)} := \sup_{P \in \mathcal{Q}} \, |P|^{-\tau} \, \|   f  \|_{A^s_{p,q} (P)}<\infty\, .
\end{equation}
I.e., we control the quasi-norm $\|  f  \|_{A^s_{p,q} (P)}$ on all dyadic cubes, hence,
locally and globally. This relates the ``new'' spaces $A^{s,\tau}_{p,q}(\mathbb{R}^n)$ to the ``old'' spaces $A^s_{p,q}(P)$ in a rather transparent way. A slightly weaker  interpretation can be derived from Proposition \ref{wav-type2},
now valid for the full range of the parameters  $s,\tau,p$, and $q$.

\begin{remark}\label{grund}
	\begin{enumerate}
		\item[(i)] It is known that the  spaces $B^{s,\tau}_{p,q}(\mathbb{R}^n)$ and $F^{s,\tau}_{p,q}(\mathbb{R}^n)$ are  quasi-Banach spaces (see \cite[Lemma~2.1]{ysy}).
		
\item[(ii)] We have the monotonicity with respect to $s$ and with respect to $q$, that is,
\[
A^{s_0,\tau}_{p,q_0}(\mathbb{R}^n)\hookrightarrow A^{s_1,\tau}_{p,q_1}(\mathbb{R}^n)\quad \mbox{if}\quad s_0 >s_1
\quad \mbox{and} \quad q_0,\,q_1\in(0,\infty],
\]
as well as
\[
A^{s,\tau}_{p,q_0} (\mathbb{R}^n)\hookrightarrow A^{s,\tau}_{p,q_1}(\mathbb{R}^n)  \quad \mbox{if}\quad q_0 \le q_1\,.
\]
\item[(iii)] Elementary embeddings:
It holds
\[
B^{s,\tau}_{p,\min(p,q)} (\mathbb{R}^n)\hookrightarrow F^{s,\tau}_{p,q}(\mathbb{R}^n)
\hookrightarrow B^{s,\tau}_{p,\max(p,q)}(\mathbb{R}^n),
\]
see, for instance, \cite[p.~39, Proposition~2.1]{ysy}.
\item[(v)] Let  $s\in\mathbb{R}$ and $p\in (0,\infty]$. Then it holds that
\[
F^{s,\tau}_{p,q}(\mathbb{R}^n)= B^{s,\tau}_{p,q}(\mathbb{R}^n) = B^{s+n(\tau-1/p)}_{\infty,\infty}(\mathbb{R}^n)
\]
if either $q\in(0,\infty)$ and $\tau\in(1/p,\infty)$ or
$q=\infty$ and $\tau\in[1/p,\infty)$ (see \cite{yy4}).
In case $s+n(\tau-1/p)>0$, the space $B^{s+n(\tau-1/p)}_{\infty,\infty}(\mathbb{R}^n)$ is a H\"older--Zygmund space
with a transparent description in terms of differences (see, for instance, \cite[Section 2.5.7]{t83}).
\item[(v)] Since $B^{s,\tau}_{p,q}(\mathbb{R}^n)$ and $F^{s,\tau}_{p,q}(\mathbb{R}^n)$ with $\tau>1/p$ are both classical Besov spaces,
we will mainly consider  the situation 		 $\tau\in[0,1/p]$ in this article.
\end{enumerate}
\end{remark}

The article is organized as follows. First, in Section \ref{Main0},
we recall  the characterizations of
$A^{s,\tau}_{p,q}(\mathbb{R}^n)$, respectively,  in terms of
Daubechies wavelets, Haar wavelets,  and differences.
In Section \ref{Main}, we determine the smallest
class $B^{s,\tau}_{p,q}(\mathbb{R}^n)$ containing non-trivial characteristic functions.
Afterwards we derive sufficient conditions on sets $E$ such that ${\bf 1}_E$
has the maximal regularity. The next section, Section \ref{Main4},  is devoted to the study of characteristic functions with low positive regularity.
For weakly exterior thick domains $E$ we are able to derive necessary and sufficient conditions for
the membership of ${\bf 1}_E$ in $F^{s,\tau}_{p,q}(\mathbb{R}^n)$. Here we also investigate the regularity of the characteristic function of the snowflake.

Finally, we make some convention on the notation used in this article.
As usual, $\mathbb{N}$ denotes the \emph{natural numbers}, $\mathbb{N}_0$
the \emph{natural numbers including $0$},
$\mathbb{Z}$ the \emph{integers}, and
$\mathbb{R}$ the \emph{real numbers}.
For any $a\in\mathbb{R}$, $\lfloor a\rfloor$ denotes the largest integer not greater than $a$ and  $ \lceil a \rceil$
the smallest integer not less than $a$.
We also use $\cc$ to denote the \emph{complex numbers} and $\mathbb{R}^n$ the
\emph{$n$-dimensional  Euclidean space}.

All functions are assumed to be complex-valued, that is,
we consider functions
$f: \mathbb{R}^n \to \cc$.
Let $\mathcal{S}(\mathbb{R}^n)$ be the collection of all \emph{Schwartz functions} on $\mathbb{R}^n$ equipped
with the well-known topology determined by a countable family of  norms
and denote by $\mathcal{S}'(\mathbb{R}^n)$ its \emph{topological dual}, namely,
the space of all bounded linear functionals on $\mathcal{S}(\mathbb{R}^n)$
equipped with the weak-$\ast$ topology.
The symbol $\mathcal{F}$ refers to  the \emph{Fourier transform},
$\mathcal{F}^{-1}$ to its \emph{inverse transform},
both defined on $\mathcal{S}'(\mathbb{R}^n)$. Recall that, for any $\varphi\in \mathcal{S}(\mathbb{R}^n)$ and $\xi\in\mathbb{R}^n$,
$$\mathcal{F}\varphi(\xi):=(2\pi)^{-\frac n2}\int_{\mathbb{R}^n} e^{-\iota x\xi}\varphi(x)\,dx\quad
\mbox{and}\quad \mathcal{F}^{-1}\varphi(\xi):=\mathcal{F}\varphi(-\xi),$$ where, for any $x:=(x_1,\ldots,x_n),
\xi:=(\xi_1,\ldots,\xi_n)\in \mathbb{R}^n$, $x\xi:=\sum_{i=1}^nx_i\xi_i$ and $\iota:=\sqrt{-1}$.

All function spaces, which  we consider in this article, are subspaces of $\mathcal{S}'(\mathbb{R}^n)$,
namely  spaces of equivalence classes with respect to
almost everywhere equality.
However, if such an equivalence class  contains a continuous representative,
then usually we work with this representative and call also the equivalence class a continuous function.

Let  $0<p\le u\le\infty$.
Then the \emph{Morrey space $\mathcal{M}^u_{p}(\mathbb{R}^n)$}
is defined as  the collection of all $p$-locally
Lebesgue-integrable functions $f$ on $\mathbb{R}^n$ such that
\begin{equation}\label{morrey7}
	\|f\|_{\mathcal{M}^u_{p}(\mathbb{R}^n)} :=  \sup_{B}
	|B|^{\frac1u-\frac1p}\left[\int_B |f(x)|^p\,dx\right]^{\frac1p}<\infty,
\end{equation}
where the supremum is taken over all balls $B$ in $\mathbb{R}^n$.
Obviously, $\mathcal{M}^p_{p}(\mathbb{R}^n)=L^p(\mathbb{R}^n)$.

The \emph{symbol}  $C $ denotes   a positive constant
which depends
only on the fixed parameters $n$, $s$, $\tau$, $p$, $q$, and probably on auxiliary functions,
unless otherwise stated; its value  may vary from line to line.
The \emph{meaning of $A \ls B$} is
given by that there exists a positive constant $C\in(0,\infty)$ such that
$A \le C \,B$.
The symbol $A \asymp B$ will be used as an abbreviation of
$A \ls B \ls A$.

For any $x\in\mathbb{R}^n$ and $t\in(0,\infty)$, let
$B(x,t) := \{y\in \mathbb{R}^n: \  |x-y|< t\}$. Let ${\bf 0}$ denote the origin of $\mathbb{R}^n$.

Let $E \subset \mathbb{R}^n$. Then $E^\complement$ denotes the complement of $E$.
If  $E$ is measurable, then $|E|$ denotes its Lebesgue measure.
In case $\mathcal{W}$ is a discrete set then $|\mathcal{W}|$ denotes its \emph{cardinality}.

The characteristic function of a set $E \subset \mathbb{R}^n$
is denoted by ${\bf 1}_E$. By $\partial E$ we denote the boundary of $E$.
We  use the convention $F:= \mathbb{R}^n \setminus \overline{E}$.
A domain $\Omega$ is a non-empty  open set in $\mathbb{R}^n$, and a non-trivial domain
is a domain different from $\emptyset$ and $\mathbb{R}^n$.
For a set $E$ we denote by $\mathring{E}$ the set of all inner points of
$E$.

Sometimes  we will work with the $\delta$-neighborhood $\Gamma^\delta$ of a set  $\Gamma$, where $\delta\in(0,\infty)$, namely
\[
\Gamma^\delta :=   \left\{x\in \mathbb{R}^n:   \dist (x, \Gamma)< \delta \right\}.
\]
Mainly we are interested in the $\delta$-neighborhood  $(\partial E)^\delta$ of the boundary of a set $E$.
In case of $2$-dimensional examples we sometimes use $\ell (\partial E)$ to denote the length of the boundary.

The symbols $\dim_H$, $\dim_M$, and $\dim_P$ refer to
the Hausdorff dimension, the Minkowski dimension, and the packing dimension, respectively.
Their definitions can be found, for instance, in Falconer \cite{Fa90,Fa97}.

%&&&&&&&&&&&&&&&&&&&&&&&&&&&&&&&&&&&&&&&&&&&&&&&&&&&&&&&&&&&&&&&&&&
%&&&&&&&&&&&&&&&&&&&&&&&&&&&&&&&&&&&&&&&&&&&&&&&&&&&&&&&&&&&&&&&&&

\section{Besov-type and Triebel--Lizorkin-type spaces\label{s3}}
\label{Main0}

%&&&&&&&&&&&&&&&&&&&&&&&&&&&&&&&&&&&&&&&&&&&&&&&&&&&&&&&&&&&&&&&&&&&
%&&&&&&&&&&&&&&&&&&&&&&&&&&&&&&&&&&&&&&&&&&&&&&&&&&&&&&&&&&&&&&&&&&&&
In this section, we recall  the characterizations of
$B^{s,\tau}_{p,q}(\mathbb{R}^n)$ and $F^{s,\tau}_{p,q}(\mathbb{R}^n)$, respectively, in terms of
Daubechies wavelets, Haar wavelets,  and differences.

%&&&&&&&&&&&&&&&&&&&&&&&&&&&&&&&&&&&&&&&&&&&&&&&&&&&&&&&&&&&&&&&&&&&&&&&&&&&&&&&&&&&&&
%&&&&&&&&&&&&&&&&&&&&&&&&&&&&&&&&&&&&&&&&&&&&&&&&&&&&&&&&&&&&&&&&&&&&&&&&&&&&&&&&&&&&&

\subsection{Characterization by wavelets}\label{smoothwavelets}

%&&&&&&&&&&&&&&&&&&&&&&&&&&&&&&&&&&&&&&&&&&&&&&&&&&&&&&&&&&&&&&&&&&&&&&&&&&&&&&&&&&&&&
%&&&&&&&&&&&&&&&&&&&&&&&&&&&&&&&&&&&&&&&&&&&&&&&&&&&&&&&&&&&&&&&&&&&&&&&&&&&&&&&&&&&&&

Wavelet bases in  Besov and Triebel--Lizorkin spaces are a well-developed concept (see, for instance, Meyer \cite{me},
Wojtasczyk \cite{woj}, and Triebel \cite{t06,t08}). Let $\widetilde{\phi}$ be
an orthonormal scaling function on $\mathbb{R}$ with compact support
and of sufficiently high regularity. Let $\widetilde{\psi}$ be one
{\it corresponding orthonormal wavelet}.
Then the tensor product ansatz yields a scaling function $\phi$ and
associated wavelets
$\{\psi_1,\ldots,\psi_{2^n-1}\}$, all defined now on $\mathbb{R}^n$ (see, for instance,
\cite[Proposition 5.2]{woj}).
We suppose
\begin{align*}
\phi\in C^{N_1}(\mathbb{R}^n)\quad\mathrm{and}\quad
\supp\phi\subset[-N_2,\,N_2]^n
\end{align*}
for certain natural numbers $N_1$ and $N_2$. This further implies
\begin{align*}
\psi_i\in C^{N_1}(\mathbb{R}^n)\hs\mathrm{and}\quad
\supp\psi_i\subset[-N_3,\,N_3]^n,\quad\forall\, i\in\{1,\ldots,2^n-1\}
\end{align*}
for some $N_3 \in \mathbb{N}$.
For any $k\in\mathbb{Z}^n$, $j\in\mathbb{N}_0$, and $i\in\{1,\ldots,2^n-1\}$, we shall use the
standard abbreviations in this article: for any $x\in\mathbb{R}^n$,
\begin{equation*}
\phi_{j,k}(x):= 2^{jn/2}\phi(2^jx-k) \quad \mathrm{and} \quad
\psi_{i,j,k}(x):= 2^{jn/2}\psi_i(2^jx-k).
\end{equation*}
Furthermore, it is well
known that
\begin{align}\label{moment}
\int_{\mathbb{R}^n} \psi_{i,j,k}(x)\, x^\gamma\,dx = 0 \quad  \mbox{if}
\quad  |\gamma|\le N_1
\end{align}
(see \cite[Proposition 3.1]{woj}) and
\begin{align}\label{4.21}
\left\{\phi_{0,k}: \ k\in\mathbb{Z}^n\right\}  \cup   \left\{\psi_{i,j,k}:\ k\in\mathbb{Z}^n,\
j\in\mathbb{N}_0,\ i\in\{1,\ldots,2^n-1\}\right\}
\end{align}
forms an {\it orthonormal basis} of $L^2(\mathbb{R}^n)$ (see \cite[Section
3.9]{me} or \cite[Section 3.1]{t06}).
Thus, for any $f\in L^2(\mathbb{R}^n)$,
\begin{align}\label{wavelet}
f=\sum_{k\in\mathbb{Z}^n}\, \lambda_k \, \phi_{0,k}+\sum_{i=1}^{2^n-1} \sum_{j=0}^\infty
\sum_{k\in\mathbb{Z}^n}\, \lambda_{i,j,k}\, \psi_{i,j,k}
\end{align}
converges in $L^2(\mathbb{R}^n)$,
where  $\lambda_k (f) := \langle f,\,\phi_{0,k}\rangle$ and
$\lambda_{i,j,k} (f):= \langle f,\,\psi_{i,j,k}\rangle$ with $\langle\cdot,\cdot\rangle$
denoting the inner product of $L^2(\mathbb{R}^n)$.
For brevity we put
\begin{align*}
\lambda (f) :=  \left\{\lambda_k (f)\right\}_{k\in\mathbb{Z}^n} \cup \left\{\lambda_{i,j,k}(f)\right\}_{i\in\{1,\ldots,2^n-1\}, j\in\mathbb{N}_0, k\in\mathbb{Z}^n} .
\end{align*}
By means of such a wavelet system one can discretize the quasi-norm $\|\cdot\|_{A^{s,\tau}_{p,q}(\mathbb{R}^n)}$.
Therefore we need some sequence spaces (see \cite[Definition 2.2]{ysy}). Let $\cx$ denote the characteristic function of the cube $[0,1)^n$. Recall, for a dyadic cube $P$ we put $j_P:= -\log \ell (P) $. Here and thereafter, we write $\log:=\log_2$ for simplicity.

\begin{definition}\label{dts}
Let $s\in \mathbb{R}$, $\tau\in[0,\infty)$, and $p,$ $q\in(0,\infty]$.
The sequence space
$\sbt$ is defined to be the space of all
 sequences $\lambda= \{\lambda_k \}_{k\in\mathbb{Z}^n} \cup \{\lambda_{i,j,k}\}_{i\in\{1,\ldots,2^n-1\}, j\in\mathbb{N}_0, k\in\mathbb{Z}^n} \subset \cc$  such that $\|\lambda\|_{\sbt}<\infty$, where
\begin{align*}\|\lambda\|_{\sbt}&:=   \sup_{\gfz{P\in\mathcal{Q}}{|P|\ge 1}}\frac1{|P|^{\tau}}
	\left\{\sum_{\{m\in\mathbb{Z}^n:\ Q_{0,m}\subset P\}} 	| \lambda_m |^p\right\}^{\frac 1p}\\
	&\quad +
\sup_{P\in\mathcal{Q}}\frac1{|P|^{\tau}}\left\{\sum_{j=\max\{j_P,0\}}^\infty
2^{j(s+\frac n2-\frac np)q} \sum_{i=1}^{2^n-1}
\left[\sum_{\{m\in\mathbb{Z}^n:\ Q_{j,m}\subset P\}}
|\lambda_{i,j,m}|^p\right]^{\frac qp}\right\}^{\frac 1q}.
\end{align*}
\end{definition}

\begin{definition}\label{dts2}
	Let $s\in \mathbb{R}$, $\tau\in[0,\infty)$,  $p\in(0,\infty)$, and $q\in(0,\infty]$.
	The sequence space
	$\sft$ is defined to be the space of all
	sequences $$\lambda= \left\{\lambda_k \right\}_{k\in\mathbb{Z}^n} \cup \left\{\lambda_{i,j,k}\right\}_{i\in\{1,\ldots,2^n-1\}, j\in\mathbb{N}_0, k\in\mathbb{Z}^n} \subset \cc$$ such that $\|\lambda\|_{\sft}<\infty$, where
	\begin{align*}\|\lambda\|_{\sft} &:=   	\sup_{\gfz{P\in\mathcal{Q}}{|P|\ge 1}}\frac1{|P|^{\tau}}
		\left\{\sum_{\{m\in\mathbb{Z}^n:\ Q_{0,m}\subset P\}} 	| \lambda_m |^p\right\}^{\frac 1p}+
		\sup_{P\in\mathcal{Q}}\frac1{|P|^{\tau}}\left\| \left[
		 \sum_{j=\max\{j_P,0\}}^\infty \sum_{i=1}^{2^n-1}
		2^{j(s+\frac n2)q}\right.\right.\\
		&\qquad\left.\left.\times \sum_{\{m\in\mathbb{Z}^n:\ Q_{j,m}\subset P\}}
		|\lambda_{i,j,m}  \cx (2^j\cdot -m)|^q \right]^{\frac 1q}\right\|_{L^p(P)}
	\end{align*}
with $\mathcal{X}:={\bf 1}_{[0,1)^n}$.
\end{definition}

As a special case of  \cite[Theorem 4.12]{lsuyy} (see also \cite{lsuyy1} and \cite{t14,t20,HT23}), we have the following wavelet characterization.

\begin{proposition}\label{wav-type2}
Let $s\in\mathbb{R}$, $\tau\in[0,\infty)$, and $p,\,q\in(0,\infty]$ ($p<\infty$ for $F$ spaces).
Let $N_1 = N_1 (s,\tau,p,q)\in\mathbb{N}_0$ be large enough.
Let $f\in\mathcal{S}'(\mathbb{R}^n)$. Then
$f\in A^{s,\tau}_{p,q}(\mathbb{R}^n)$ if and only if $f$ can be represented in $\mathcal{S}'(\mathbb{R}^n)$ as in \eqref{wavelet}
such that $\|\lz(f)\|_{{a}^{s,\tau}_{p,q}(\mathbb{R}^n)}<\infty $.
Moreover, $\|f\|_{A^{s,\tau}_{p,q}(\mathbb{R}^n)}$ is equivalent to
$\|\lz(f)\|_{{a}^{s,\tau}_{p,q}(\mathbb{R}^n)}$ with the positive equivalence constants independent of $f$.
\end{proposition}

\begin{remark}
On the interpretation of $\lambda (f)$, we observe that in general the element  $f$ of $B^{s,\tau}_{p,q}(\mathbb{R}^n)$ is not an element of $L^2 (\mathbb{R}^n)$,
it might be a singular distribution. Thus,
$\langle f,\,\phi_{0,m}\rangle$ and $\langle f,\,\psi_{i,j,m}\rangle$ require an interpretation,
which has been done in the proof
of Proposition \ref{wav-type2} (see
\cite[Theorem 4.12]{lsuyy} for all the details).
\end{remark}

Finally we recall the following  embeddings.

\begin{proposition}\label{embc}
	Let $s\in\mathbb{R}$, $\tau\in[0,\infty)$, $p\in(0,\infty)$, and $q \in (0,\infty]$.

\begin{itemize}
\item[{\rm (i)}] One has
	\[
	A^{s +n\tau, {\rm loc}}_{p,q} (\mathbb{R}^n) \subset
	A^{s,\tau}_{p,q} (\mathbb{R}^n) \hookrightarrow B^{s+n\tau- n/p}_{\infty,\infty} (\mathbb{R}^n) \, .
	\]	
\item[{\rm (ii)}]	Let $C(\mathbb{R}^n)$ denote the space of all complex-valued continuous and bounded  functions   on $\mathbb{R}^n$
equipped with the supremum norm.
Then
\[
A^{s,\tau}_{p,q} (\mathbb{R}^n) \hookrightarrow C (\mathbb{R}^n) \quad \Longleftrightarrow \quad
s> n\left( \frac 1p - \tau \right) .
\]	
\end{itemize}
\end{proposition}	

\begin{proof}
Part (ii) has been proved in \cite{YHMSY}.  The second embedding in (i) is a direct consequence of the wavelet characterization of
$	A^{s,\tau}_{p,q} (\mathbb{R}^n) $ and  $B^{s+\tau- n/p}_{\infty,\infty} (\mathbb{R}^n)$, respectively,
because
\[
\max_{i=1, \ldots, 2^n-1}\, \sup_{j \in \mathbb{N}_0} \, \sup_{k\in \mathbb{Z}}\, 2^{j n \tau} \,  2^{j (s+ n(\frac 12 - \frac 1p))}
\,| \langle f , \psi_{i,j,k} \rangle|\ls \| f \|_{B^{s,\tau}_{p,\infty}(\mathbb{R}^n)}
\]
holds for any $ f \in B^{s,\tau}_{p,\infty}(\mathbb{R}^n) $.
For the first assertion in (i) one only needs to compare the wavelet characterization  of
$	A^{s + n\tau, {\rm loc}}_{p,q} (\mathbb{R}^n) $ and  $A^{s, \tau}_{p,q} (\mathbb{R}^n)$.
\end{proof}

%&&&&&&&&&&&&&&&&&&&&&&&&&&&&&&&&&&&&&&&&&&&&&&&&&&&&&&&&&&&&&&&&&&&&&&&&&&&&&&&&&&&&&
%&&&&&&&&&&&&&&&&&&&&&&&&&&&&&&&&&&&&&&&&&&&&&&&&&&&&&&&&&&&&&&&&&&&&&&&&&&&&&&&&&&&&&

\subsection{Characterization by  Haar wavelets}\label{haarwavelets}

%&&&&&&&&&&&&&&&&&&&&&&&&&&&&&&&&&&&&&&&&&&&&&&&&&&&&&&&&&&&&&&&&&&&&&&&&&&&&&&&&&&&&&
%&&&&&&&&&&&&&&&&&&&&&&&&&&&&&&&&&&&&&&&&&&&&&&&&&&&&&&&&&&&&&&&&&&&&&&&&&&&&&&&&&&&&&

First, we recall the definition of the Haar system.
Let $\widetilde{\mathcal{X}}:={\bf 1}_{[0,1)}$.
The generator of the Haar system in dimension $1$ is denoted by $\widetilde{h}$, namely
\[
\widetilde{h}(t):=\left\{
\begin{array}{lll}
	1 & \quad &    t\in[0,1/2),
	\\
	-1 & \quad &    t\in [1/2,1),\\
	0  &\quad &\mbox{otherwise}.
\end{array}
\right.
\]
The functions, we are interested in, are just tensor products of these two  functions.
For any given $\varepsilon:= (\varepsilon_1, \, \ldots , \varepsilon_n)$ with
$\varepsilon_i \in \{0,1\}$, define
\begin{align*}
	h_\varepsilon (x) := \left[\prod_{\{i: \, \varepsilon_i=0\}} \widetilde{\mathcal{X}}(x_i)\right]\,
	\left[\prod_{\{i: \, \varepsilon_i=1\}} \widetilde{h}(x_i)\right],
	\ \  \forall\, x=(x_1, \, \ldots , x_n) \in \mathbb{R}^n .
\end{align*}
This results in $2^n$ different functions.
In case $\varepsilon = (0, \, \ldots , 0)$  we
\emph{always  write $\mathcal{X}$ instead of
	$h_{(0,\ldots , 0)}$}, which is just ${\bf 1}_{[0,1)^n}$.
The other $2^n-1$ functions will be enumerated in an appropriate way
and denoted by $\{h_1, \ldots , h_{2^n-1}\}$. These functions
are the generators of the inhomogeneous Haar system in $\mathbb{R}^n$.
In what follows, for any $i\in\{1,\ldots,2^n-1\}$, $j\in\mathbb{N}_0$, and $m\in\mathbb{Z}^n$,
we let
\begin{align*}
	\cx_{j,m}:=2^{jn/2}\cx(2^j\cdot-m),\  \  h_{i,j,m}:=2^{jn/2}h_i(2^j\cdot-m).
\end{align*}
Then the set
\begin{align*}
H:=\left\{\cx_{0,m},\,h_{i,j,m}:\ i\in\left\{1,\ldots,2^n-1\right\},\
j\in\mathbb{N}_0,\ m\in\mathbb{Z}^n\right\}
\end{align*}
forms the well-known orthonormal Haar wavelet system in $\mathbb{R}^n$ (we shall call it
just the \emph{Haar system}).

For a function $f\in L^1_{\rm loc}(\mathbb{R}^n)$, the Haar wavelet expansion is given by
\begin{equation*}
	f=\sum_{m\in\mathbb{Z}^n}\, \langle f , \mathcal{X}_{0,m}  \rangle \, \mathcal{X}_{0,m} + \sum_{i=1}^{2^n-1} \sum_{j=0}^\infty
	\sum_{m\in\mathbb{Z}^n}\, \langle f , h_{i,j,m} \rangle\, h_{i,j,m}.
\end{equation*}
For given $f \in \mathcal{S}'(\mathbb{R}^n)$ we denote by $\delta (f)$ the sequence of all Fourier--Haar coefficients. More exactly we put
\begin{align}\label{koeffh}
	\delta (f) := \{\delta_k (f)\}_{k} \cup \{\delta_{i,j,k}(f)\}_{i,j,k} ,
\end{align}
where  $\delta_k (f) := \langle f,\,\mathcal{X}_{0,k}\rangle$ and
$\delta_{i,j,k} (f):= \langle f,\,h_{i,j,k}\rangle$ with $\langle\cdot,\cdot\rangle$
denoting the inner product of $L^2(\mathbb{R}^n)$.
Triebel \cite[Theorem 3.41]{t14} (see also \cite{HT23}) has proved the following.

\begin{proposition}\label{haarprop}
\begin{itemize}
\item[{\rm (i)}] 	Let $s\in \mathbb{R}$, $p,q\in(0,\infty]$, and $\tau\in[0,1/p)$. 	Then
	the mapping $I:~ f \mapsto \delta (f)$
	is an isomorphism of $B^{s,\tau}_{p,q}(\mathbb{R}^n) $ onto $b^{s,\tau}_{p,q}(\mathbb{R}^n)$	if
	\begin{equation}\label{w001b}
		\max \left\{\frac 1p -1,~ n\left(\frac 1p -1\right)\right\} < s < \min \left\{\frac 1p, 1, n \left(\frac 1p -\tau\right)\right\}.	
	\end{equation}
\item[{\rm (ii)}] Let either
$p,q\in(0,\infty)$ and
\begin{equation*}
	\max \left\{ 0, ~ n\left(\frac 1p -1\right), n\left(\frac 1q -1\right)\right\}
< s < \min\left\{ \frac 1p,\frac 1q, 1, n\left(\frac 1p -\tau\right)\right\}
\end{equation*}	
or $p,q \in(1,\infty)$ and $s=0$ or $p\in(1,\infty)$, $q\in(1,\infty]$, and
\begin{equation*}
	\max \left\{ \frac 1p -1,~ \frac 1q -1 \right\} < s < \min \left\{ \frac 1p, \frac 1q , n\left(\frac 1p
-\tau\right)\right\}.	
\end{equation*}	
	Then
the mapping $I:~ f \mapsto \delta (f)$
is an isomorphism of $F^{s,\tau}_{p,q}(\mathbb{R}^n) $ onto $f^{s,\tau}_{p,q}(\mathbb{R}^n)$.
\end{itemize}
\end{proposition}

\begin{remark}
\begin{itemize}
\item[{\rm (i)}]  Clearly, for singular distributions it may happen that $\delta (f)$ is not well defined.
	The statement of the proposition  implies that, if  \eqref{w001b}  is satisfied, then
	$\delta (f)$ is  well defined.
\item[{\rm (ii)}]  There are many articles dealing with the basis problem for the Haar system in
	Besov as well as  Triebel--Lizorkin spaces.
	Let us mention here at least Ropela \cite{r76}, Triebel \cite{t78,t10}, Oswald \cite{Os81,Os18,Os21},
	 Seeger and Ullrich \cite{SU1,SU2}, and Garrigos, Seeger, and Ullrich \cite{GSU,GSU18,GSU21,GSU23}.
	Observe that for $\tau>0$ this problem for the type spaces $A^{s,\tau}_{p,q}(\mathbb{R}^n)$ is meaningless because these spaces are not separable.
\item[{\rm (iii)}]	In the particular case that $n=1$ and $p=q$,  Sickel et al \cite{syyz} have found necessary and sufficient conditions on $s,\tau$, and $p$ such that Proposition \ref{haarprop} remains true. It turns out that the lower bound in
\eqref{w001b} is too restrictive. But this has no influence on our problem here.
\end{itemize}
\end{remark}

%&&&&&&&&&&&&&&&&&&&&&&&&&&&&&&&&&&&&&&&&&&&&&&&&&&&&&&&&&&&&&&&&&&&&&&&&&&&&&&&&&
%&&&&&&&&&&&&&&&&&&&&&&&&&&&&&&&&&&&&&&&&&&&&&&&&&&&&&&&&&&&&&&&&&&&&&&&&&&&&&&&&

\subsection{Characterization by differences}
\label{differ}

%&&&&&&&&&&&&&&&&&&&&&&&&&&&&&&&&&&&&&&&&&&&&&&&&&&&&&&&&&&&&&&&&&&&&&&&&&&&&&&&&
%&&&&&&&&&&&&&&&&&&&&&&&&&&&&&&&&&&&&&&&&&&&&&&&&&&&&&&&&&&&&&&&&&&&&&&&&&&&&&&&

For any $M\in\mathbb{N}$, any function $f: \mathbb{R}^n \to \cc$, and any $h, x\in\mathbb{R}^n$, let
\[
\Delta_h^Mf(x):= \sum_{j=0}^M \, (-1)^j\left(\gfz{M}{j}\right)f(x+[M-j]h),
\]
where $ (\gfz{M}{j})$ for any $j\in\{0,\ldots, M\}$ denotes the \emph{binomial coefficients}.
For any $p\in(0,\infty]$, let $L^p(\mathbb{R}^n)$ denote the \emph{Lebesgue space} which consists of all measurable
functions $f$ on $\mathbb{R}^n$ such that
$$\|f\|_{L^p(\mathbb{R}^n)}:=\lf[\int_{\mathbb{R}^n} |f(x)|^p\,dx\right]^{\frac1p}<\infty$$
with the usual modification made when $p=\infty$, and $L_{\rm loc}^p(\mathbb{R}^n)$ the space of all measurable functions
which belong locally to $L^p(\mathbb{R}^n)$.
For any $\tau\in[0, \infty)$, $p\in (0,\infty]$, and $f\in L_{\rm loc}^p(\mathbb{R}^n)$, let
\begin{align*}%\label{2.17}
\|f\|_{L^p_\tau(\mathbb{R}^n)} := \sup_{\{P\in\mathcal{Q}: |P|\ge1\}}
\frac1{|P|^\tau}
\lf[\dint_P|f(x)|^p\,dx\right]^{\frac1p}
\end{align*}
with the usual modification made when $p=\infty$.
Denote by $L^p_\tau(\mathbb{R}^n)$ the set of all
functions $f$ on $\mathbb{R}^n$ satisfying $\|f\|_{L^p_\tau(\mathbb{R}^n)}<\infty$.
Obviously, for any $p\in(0,\infty]$, $L^p_0(\mathbb{R}^n)=L^p(\mathbb{R}^n)$.
We set
\begin{eqnarray*}
\|f\|^\spadesuit_{B^{s,\tau}_{p,q}(\mathbb{R}^n)} :=\sup_{P\in\mathcal{
Q}}\frac1{|P|^\tau}\left\{\int_0^{2\min\{\ell(P),1\}}
t^{-sq}\sup_{\frac t2\le|h|<t}\left[\int_P
|\Delta_h^M
f(x)|^p\,dx\right]^{\frac qp}\,\frac{dt}{t}\right\}^{\frac 1q}.
\end{eqnarray*}
The following difference characterization was proved  in
\cite[Theorem 4.7]{ysy}. Here we focus on the case   $\tau\in[0,1/p]$.

\begin{proposition}\label{t4.7}
Let  $p,q\in [1, \infty]$ and $M\in\mathbb{N}$.
Let $s\in(0,M)$ and $\tau \in[0,1/p].$
Then $f\in B^{s,\tau}_{p,q}(\mathbb{R}^n)$ if and only
if $f\in L^p_\tau(\mathbb{R}^n)$ and $\|f\|^\spadesuit_{B^{s,\tau}_{p,q}(\mathbb{R}^n)}<\infty$. Furthermore,
$\|f\|_{L^p_\tau(\mathbb{R}^n)}+\|f\|^\spadesuit_{B^{s,\tau}_{p,q}(\mathbb{R}^n)}$ and $\|f\|_{B^{s,\tau}_{p,q}(\mathbb{R}^n)}$ are
equivalent with the positive equivalence constants independent of $f$.
\end{proposition}

In the next result Morrey spaces are used, see \eqref{morrey7}.

\begin{proposition}\label{H1}
Let   $ 0 < p \le u <\infty$ and $ M \in \mathbb{N} $.
Suppose that
\begin{equation*}
		n \max \left \{0, \frac{1}{p} - 1 \right\} < s < M.
\end{equation*}
	Then a function $ f~ \in ~ L^{\min\{p,q\}}_{\rm loc}(\mathbb{R}^n)$ belongs to $ F^{s,\frac 1p - \frac 1u}_{p,\infty}(\mathbb{R}^n) $
	if and only if  $ f \in L^{1}_{\rm loc}(\mathbb{R}^n)$ and
	\[
 \| f \|_{\mathcal{M}^u_p (\mathbb{R}^n)}
 +   \left\|   \sup_{t\in(0, 1)} t^{-s} \,  \left[ t^{-n} \int_{B({\bf 0},t)} |\Delta^{M}_{h}f(\cdot) |^p dh \right]^{\frac 1p}   \right\|_{\mathcal{M}^u_p (\mathbb{R}^n)}< \infty.
\]
\end{proposition}

We refer to Hovemann \cite{Diss,Ho}, see also   Hedberg and Netrusov \cite{HN}.
The main difference  between these characterizations consists in the used domain of integration.

\begin{remark}
 In the literature there exist a couple of further characterizations by differences, we refer to
 Drihem \cite{Dr},
Hovemann \cite{Diss,Ho}, Hovemann and Sickel \cite{HS},	and Triebel \cite{t14}.
\end{remark}

%&&&&&&&&&&&&&&&&&&&&&&&&&&&&&&&&&&&&&&&&&&&&&&&&&&&&&&&&&&&&&&&&&&&&&&&
%&&&&&&&&&&&&&&&&&&&&&&&&&&&&&&&&&&&&&&&&&&&&&&&&&&&&&&&&&&&&&&&&&&&&&&&

\section{Besov-type spaces and the pointwise multiplier problem for Besov spaces}
\label{multi}

%&&&&&&&&&&&&&&&&&&&&&&&&&&&&&&&&&&&&&&&&&&&&&&&&&&&&&&&&&&&&&&&&&&&&&&&
%&&&&&&&&&&&&&&&&&&&&&&&&&&&&&&&&&&&&&&&&&&&&&&&&&&&&&&&&&&&&&&&&&&&&&&&

As we stated above, this section will motivate  our investigations in this article.
Therefore we need to deal with pointwise multiplication of distributions.
Let
$\phi_0 \in C_{{\rm c}}^\infty (\mathbb{R}^n)$ be a radial and real-valued
function satisfying $0 \le \phi_0 \le 1$ for all $x\in \mathbb{R}^n$ and $\phi_0 (x)= 1$ in a neighborhood of the origin.
Then, for any $f \in \mathcal{S}'(\mathbb{R}^n)$,
\[
S^j f (x) := (2\pi)^{-\frac n2} \, 2^{jn} \int_{\mathbb{R}^n} \cfi \phi\left(2^j(x-y)\right)\, f(y)\, dy \, , \quad x\in \mathbb{R}^n\, , \ j \in \nn_0\, ,
\]
defines a weakly convergent sequence of smooth functions with limit $f$ in $\mathcal{S}'(\mathbb{R}^n)$.
In what follows we are making use of the following definition.

\begin{definition}
For any $f,g\in\mathcal{S}'(\mathbb{R}^n)$, define
	\begin{equation*}
		f\, \cdot \,  g:=\lim_{j \rightarrow\infty}\, S^jf\, \cdot \,  S^jg
	\end{equation*}
	if the limit on the right-hand side
	exists in $\mathcal{S}'(\mathbb{R}^n)$ and is independent of $\phi_0$.
\end{definition}

We further need the following.

\begin{definition}\label{df-mul}
	Let $p,q\in(0,\infty]$ and $s \in\mathbb{R}$.
\begin{itemize}
\item[{\rm (i)}]	
A tempered distribution $f$ is called a {\em pointwise multiplier} for the space $B^s_{p,q} (\mathbb{R}^n)$	if
the product 	$f\, \cdot \, g$ is not only  well defined for all $g \in B^s_{p,q} (\mathbb{R}^n)$ but also satisfies
$	f\, \cdot \, g  \in  B^s_{p,q} (\mathbb{R}^n) $.
\item[{\rm (ii)}]
	The \emph{pointwise multiplier space}
	$M(B^{s}_{p,q}(\mathbb{R}^n))$ of $B^{s}_{p,q}(\mathbb{R}^n)$ is the collection of all pointwise multipliers of
	$ B^s_{p,q} (\mathbb{R}^n)$ equipped with
	the following quasi-norm:
	\begin{equation*}
		\left\|f\right\|_{M(B^{s}_{p,q}(\mathbb{R}^n))}:=\sup_{\{g\in B^{s}_{p,q}(\mathbb{R}^n),\,g\neq\theta\}}
		\frac{\left\|fg\right\|_{B^{s}_{p,q}(\mathbb{R}^n)}}{\left\|g\right
\|_{B^{s}_{p,q}(\mathbb{R}^n)}},
	\end{equation*}
	where $\theta$ denotes the zero element of $B^{s}_{p,q}(\mathbb{R}^n)$.
\end{itemize}
\end{definition}

There are two results in the theory of pointwise multipliers which are of interest for us.  The first one can be found in \cite{LSYY24},
the second on in \cite{Gu2}. We need  further notation.
In what follows, for any $p\in(0,\infty]$, let $\sigma_p := n \, \max (0, \frac 1p -1 )$.

Let $\psi$ be a non-negative smooth function on  $\mathbb{R}^n$ with compact support such that
\begin{equation*}
	\sum_{\ell\in\mathbb{Z}^n}\, \psi(x-\ell) = 1,\, \quad x\in \mathbb{R}^n.
\end{equation*}
The \emph{uniform Besov space} $B^{s}_{p,q,{\rm unif}}(\mathbb{R}^n)$
is defined to be the set of all $f\in\mathcal{S}'(\mathbb{R}^n)$ such that
\[
\|f\|_{B^{s}_{p,q,{\rm unif}}(\mathbb{R}^n)}:=\sup_{\ell\in\mathbb{Z}^n}\left\|\psi(\cdot-\ell) f(\cdot)\right\|_{B^{s}_{p,q}(\mathbb{R}^n)}<\infty.
\]
The \emph{uniform Besov-type space} $B^{s,\tau}_{p,q,{\rm unif}}(\mathbb{R}^n)$
is defined to be the set of all $f\in\mathcal{S}'(\mathbb{R}^n)$ such that
\[
\|f\|_{B^{s,\tau}_{p,q,{\rm unif}}(\mathbb{R}^n)}:=\sup_{\ell\in\mathbb{Z}^n}\left\|\psi(\cdot-\ell) f(\cdot)\right\|_{B^{s,\tau}_{p,q}(\mathbb{R}^n)}<\infty.
\]
Let $b\in \mathbb{R}$. The
	\emph{Besov space with only logarithmic smoothness},
	$\mathcal{B}^{0,b}_{\infty,\infty}(\mathbb{R}^n)$,
	is defined to be the set of all $f\in\mathcal{S}'(\mathbb{R}^n)$ such that
	\[
	\|f\|_{\mathcal{B}^{0,b}_{\infty,\infty}(\mathbb{R}^n)}:=
	\sup_{k\in\mathbb{N}_0} \left\{(1+k)^b
	\left\|S_kf\right\|_{L^\infty(\mathbb{R}^n)}\right\}<\infty.
	\]
In the following proposition we collect characterizations of $M(B^s_{p,p}(\rn))$ in case $0 < p \le 1$, see  \cite{LSYY24}.

\begin{proposition}\label{FAchar}
	Let $p\in(0,1]$ and  $s\in(-\infty,\frac np)$.
	\begin{enumerate}
		\item[{\rm(i)}] If $s\in(\sigma_p,\frac np)$,
		then
		\begin{align*}
			M\left(B^s_{p,p}(\mathbb{R}^n)\right)
			=L^\infty(\mathbb{R}^n)\cap B^{s,\frac1p-\frac sn}_{p,p,{\rm unif}}(\mathbb{R}^n).
		\end{align*}
		\item[{\rm(ii)}] If $s=\sigma_p$, then
		\begin{align*}
			M\left(B^{s}_{p,p}(\mathbb{R}^n)\right)
			=L^\infty(\mathbb{R}^n)\cap B^{s,\frac1p-\frac sn}_{p,p,{\rm unif}}(\mathbb{R}^n)
			\cap \mathcal{B}^{0,\frac{1}{p}}_{\infty,\infty}(\mathbb{R}^n).
		\end{align*}
		\item[{\rm(iii)}] If $p=1$ and $s\in(-\infty,0)$, then
		\begin{align*}
			M\left(B^{s}_{1,1}(\mathbb{R}^n)\right)=B^{-s}_{\infty,\infty}(\mathbb{R}^n).
		\end{align*}
		\item[{\rm(iv)}] If $p\in(0,1)$ and $s\in[\frac12\sigma_p,\sigma_p)$,
		then
		\begin{align*}
			M\left(B^{s}_{p,p}(\mathbb{R}^n)\right)=
			B^{s,\frac1p-\frac sn}_{p,p,{\rm unif}}(\mathbb{R}^n)\cap
			B^{\sigma_p-s}_{\infty,\infty}(\mathbb{R}^n).
		\end{align*}
		\item[{\rm(v)}] If $p\in(0,1)$ and $s\in(-\infty,\frac12\sigma_p)$,
		then
		\begin{align*}
			M\left(B^{s}_{p,p}(\mathbb{R}^n)\right)=B^{\sigma_p-s}_{\infty,\infty}(\mathbb{R}^n).
		\end{align*}
	\end{enumerate}
	All the above equalities hold in the sense of equivalent quasi-norms.
\end{proposition}

In case $p=1$ there is a number of further descriptions of   $M(B^s_{1,1}(\mathbb{R}^n))$.
Maz'ya and  Shaposhnikova \cite{MS81} have been the first who have found a characterization of
 $M(B^s_{1,1}(\mathbb{R}^n))$ for $s>0$, see  also \cite[Chapter~5]{MS09}, but this differs from the given one in Proposition  \ref{FAchar}. Further characterizations have been found by Netrusov \cite{Ne92}, Sickel \cite{s99b},
 and Triebel \cite{t03}, all different from the given one in Proposition  \ref{FAchar}.

Now we turn to a first consequence of  Proposition \ref{FAchar}. Since the H\"older--Zygmund spaces $B^\lambda_{\infty, \infty} (\mathbb{R}^n)$, $\lambda >0$, contain only equivalence classes of locally integrable functions which contain a continuous function,   Proposition  \ref{FAchar} implies the following.

 \begin{corollary}
 Let either $p=1$ and $s \in (-\infty,0]$	or   $p\in(0,1)$ and $s\in(-\infty,\sigma_p)$.
 Let $E \subset \mathbb{R}^n$ be a measurable set such that $\cx_E \in 	M(B^{s}_{p,p}(\mathbb{R}^n))$.
 Then $|E|=0$.
 \end{corollary}

 \begin{proof}
 Only the case 	$p=1$ and $s=0$ can not be proved by using Proposition  \ref{FAchar} and the argument before this corollary.
 But because of $M(B^0_{1,1}(\mathbb{R}^n)) = M(B^0_{\infty,\infty}(\mathbb{R}^n))$ (see \cite[Theorem 4.1]{lsyy23}),  this has been proved in Koch and Sickel \cite[Proposition 18]{KS02}.
 \end{proof}

There is a second consequence.

 \begin{corollary}\label{Gulidual}
 	Let $s\in[0,n)$, $q\in(0, \infty]$, and $p\in(1,\infty)$.
If $f \in M(B^{s}_{1,1}(\mathbb{R}^n))$, then $f \in  M(B^{t}_{p,q}(\mathbb{R}^n))$
for any $t \in (0,s/p)$.
\end{corollary}

\begin{proof}
	Our assumption  $f \in M(B^{s}_{1,1}(\mathbb{R}^n))$ implies $f\in L^\infty (\mathbb{R}^n)$,
due to Proposition  \ref{FAchar}.	Hence,  for any $p$,
$f\in M(L^p (\mathbb{R}^n))$. Notice that the operator $g \mapsto f \, \cdot \, g$ is linear.
Let $ \theta, \theta'\in(0, 1)$,   $r\in(0,\infty]$, and $ p_1\in(1, \infty)$ be such that
\[
\frac 1p := 1-\theta  + \frac{\theta}{p_1} \qquad \mbox{and}\qquad \frac{1}{r} :=  1-\theta  + \frac{\theta}{2}.
\]	
Complex interpolation yields
\[
\left[B^{s}_{1,1}(\mathbb{R}^n), L^{p_1} (\mathbb{R}^n)\right]_\theta = F^{(1-\theta)s}_{p,r} (\mathbb{R}^n)\,
\]	
(see \cite{FJ90}), and therefore $f \in M(F^{(1-\theta)s}_{p,r} (\mathbb{R}^n))$.
Now we continue with real interpolation
\[
\left(F^{(1-\theta)s}_{p,r}(\mathbb{R}^n), L^{p} (\mathbb{R}^n)\right)_{\theta',q} = B^{(1-\theta')(1-\theta)s}_{p,q} (\mathbb{R}^n)\, ,
\]	
see, e.g., \cite[Section 2.4.2]{t83},
which implies $f \in M(B^{(1-\theta')(1-\theta)s}_{p,q} (\mathbb{R}^n))$. Letting $p_1 \to \infty$ and $\theta, \theta' \to 0$ the claim then follows.
\end{proof}

Further consequences of Proposition \ref{FAchar}  will be discussed in Sections \ref{Main} and \ref{Main4}.
The second result,  we would like to recall, is due to  Gulisashvili \cite{Gu2} and deals with the case $1 <p< \infty$.

\begin{definition}\label{df-mul2}
	Let $p_1, p_2, q_1, q_2 \in(0,\infty]$ and $s_1,s_2 \in\mathbb{R}$.
	\begin{itemize}
\item[{\rm (i)}]	
	A tempered distribution $f$ is called a {\em pointwise multiplier} for the pair
	$(B^{s_1}_{p_1,q_1} (\mathbb{R}^n), B^{s_2}_{p_2,q_2} (\mathbb{R}^n))$	if
	the product 	$f\, \cdot \, g$ is not only  well defined for all $g \in B^{s_1}_{p_1,q_1} (\mathbb{R}^n)$ but also satisfies
	$	f\, \cdot \, g  \in  B^{s_2}_{p_2,q_2} (\mathbb{R}^n) $.
	\item[{\rm (ii)}] 	
	The \emph{pointwise multiplier space}
	$M(B^{s_1}_{p_1,q_1} (\mathbb{R}^n), B^{s_2}_{p_2,q_2} (\mathbb{R}^n))$  is the collection of all pointwise multipliers with respect to the pair $\big(B^{s_1}_{p_1,q_1} (\mathbb{R}^n), B^{s_2}_{p_2,q_2} (\mathbb{R}^n)\big)$
	 equipped with
	the following quasi-norm:
	\begin{equation*}
		\left\|f\right\|_{M\big(B^{s_1}_{p_1,q_1} (\mathbb{R}^n), B^{s_2}_{p_2,q_2} (\mathbb{R}^n)\big)}:=\sup_{\{g\in B^{s_1}_{p_2,q_2}(\mathbb{R}^n),\,g\neq\theta\}}
		\frac{\left\|f\, \cdot \, g\right\|_{B^{s_2}_{p_2,q_2}(\mathbb{R}^n)}}{\left\|g\right\|_{
B^{s_1}_{p_1,q_2}(\mathbb{R}^n)}}\, ,
	\end{equation*}
		where $\theta$ denotes the zero element of $B^{s_1}_{p_1,q_1}(\mathbb{R}^n)$.
\end{itemize}
\end{definition}

\begin{proposition}\label{Guli}
	Let $p\in(1,\infty)$ and  $s\in(0,\frac 1p]$. Then $f \in M(B^{s}_{p,1} (\mathbb{R}^n), B^{s}_{p,\infty} (\mathbb{R}^n))$
	if and only if $f \in L^\infty (\mathbb{R}^n)$ and
\[
A (f):= \sup_{h \in \mathbb{R}^n \setminus \{ \bf{0}\}} \,|h|^{-sp} \sup_{r\in(0, 1)} r^{ps-n} \sup_{x\in \mathbb{R}^n} \, \int_{B(x,r)} \, |f(y+h)-f(y)|^p dy < \infty\, .
\]	
Moreover,
\[
	\left\|f\right\|_{M\big(B^{s}_{p,1} (\mathbb{R}^n), B^{s}_{p,\infty} (\mathbb{R}^n)\big)}~ \sim \, \|f \|_{L^\infty (\mathbb{R}^n)} + [A (f)]^{\frac 1p}\,,
\]	
where the   positive equivalent constants are independent of $f$.
\end{proposition}

Proposition \ref{Guli} is a little bit surprising in the sense that the description of the multiplier spaces
$M(B^{s}_{p,1} (\mathbb{R}^n), B^{s}_{p,\infty} (\mathbb{R}^n))$ is much more simple than the description of the
spaces $M(B^{s}_{p,p} (\mathbb{R}^n))$, for which we refer to \cite[Theorem~4.1.1]{MS09}.
For our purposes we need a reformulation of Proposition \ref{Guli}. This will be based on the next lemma.
Recall that $L^p_\tau (\rn)$  has been defined in Subsection \ref{differ}.

\begin{lemma}\label{Guli2}
	Let $p\in [1,\infty)$, $u \in [p,\infty)$, and  $s\in(0,1)$. We put $\tau := \frac 1p -\frac 1u$.
	Then $B^{s,\tau}_{p,\infty} (\mathbb{R}^n)$
is the collection of all $f \in \cm^u_p (\mathbb{R}^n)$ such that	
	\[
B(f):= \sup_{|h | < 1} \, |h|^{-s}\, \|f(\, \cdot \, +h)-f(\, \cdot\, )\|_{\cm^u_p (\mathbb{R}^n)} < \infty\, .
	\]	
	Moreover,
	\[
	\left\|f\right\|_{B^{s,\tau}_{p,\infty} (\mathbb{R}^n)} \sim \|f \|_{L^p_\tau  (\mathbb{R}^n)} + B (f)\, ,
	\]	
where the positive equivalent constants are independent of $f$.
\end{lemma}

\begin{proof}
There are various publications on characterizations of Besov-type spaces
 by differences, see Subsection \ref{differ} for some.
In all these references only ball means of differences are considered. Therefore we will sketch a proof.
In some sense we will follow Triebel's arguments in the proofs of Theorems 2.5.10 and 2.5.12 in \cite{t83}.
\\
{\em Step 1.} First we shall prove $B(f) \ls \|f\|_{B^{s,\tau}_{p,\infty}(\mathbb{R}^n)}$.
By $\{\phi_j\}_{j\in\mathbb{N}_0}$ we denote a smooth dyadic decomposition of unity.
For simplicity we may start the construction with a radial nonnegative function $\phi_0 \in C_{{\rm c}}^\infty (\mathbb{R}^n)$ such that
$\phi_0 (x) = 1$ if $|x|\le 1$ and define
\[
\phi_1 (x):= \phi_0 (x/2) - \phi_0 (x) \quad \mbox{and} \quad
\phi_j (x) := \phi_1 (2^{-j+1}x) \, , \quad x\in \mathbb{R}^n, ~ j \in \nn\, .
\]
Then, for any $j \ge 1$, $\phi_j$ has support in a dyadic annulus and $\sum_{j=0}^\infty \phi_j (x) =1 $ for any $x\in\mathbb{R}^n$.
Put $f_j := \cfi [\phi_j  \, \cf f  ]$ for any $j \in \nn_0$.
We need the associated Peetre maximal functions
defined as
\[
f_j^* (x):= \sup_{z \in \mathbb{R}^n} \,
\frac{ |f_j(x-z)|}{(1+2^j|z|)^a} \, , \quad x \in \mathbb{R}^n\, .
\]
Here $f \in \cs'(\mathbb{R}^n)$, $j \in \nn_0$, and $a>0$ will be chosen later on.
Then it follows that
\begin{equation}\label{novy1}
\sup_{|h| < 2^{-k}}\, \|\Delta^1_h 	f_j\|_{\cm^u_p (\mathbb{R}^n)} \ls \min(1,2^{j-k}) \, \|f_j^*\|_{\cm^u_p (\mathbb{R}^n)} \, , \quad x \in \mathbb{R}^n\, ,\ k\in\mathbb{N}_0,
\end{equation}
similarly as in the formulas (6)-(8) in the page 102 of \cite{t83}.
This is the main step in proving the claimed inequality. To finish this estimate one can continue as in Step 1 of the proof of
Theorem 2.5.12  in \cite{t83}.
The boundedness of the Peetre maximal function in Morrey spaces  (see \cite{yy3}), completes the proof of the claimed inequality.

It still needs to show $\|f \|_{L_\tau^p (\mathbb{R}^n)} \ls \|f\|_{B^{s,\tau}_{p,\infty}(\mathbb{R}^n)}$. Indeed,
the continuous embedding $B^{s,\tau}_{p,\infty}(\mathbb{R}^n) \hookrightarrow \mathcal{M}^u_p (\mathbb{R}^n)  \hookrightarrow L^p_\tau (\mathbb{R}^n)$, with $\tau = \frac 1p - \frac 1u$, is well known,
we refer to \cite[Section 2.2]{ysy} in case $p\in(1,\infty)$. For $p=1$ and $s\in(0,\infty)$ one may use
\[
B^{s,\tau}_{1,\infty}(\mathbb{R}^n) \hookrightarrow F^{0,\tau}_{1,1}(\mathbb{R}^n) \hookrightarrow \mathcal{M}^u_1 (\mathbb{R}^n)\hookrightarrow L^1_\tau (\mathbb{R}^n)\, , \quad \tau = 1 - \frac 1u \, .
\]
{\em Step 2.}
We shall prove  the inverse
inequality  $ \|f\|_{B^{s,\tau}_{p,\infty}(\mathbb{R}^n)}\ls \|f \|_{L^p_\tau (\mathbb{R}^n)} + B (f)$.
Employing Proposition \ref{t4.7} with $M=1$, we obtain
\begin{align*}
\|f\|_{B^{s,\tau}_{p,\infty}(\mathbb{R}^n)} & \sim 	\|f\|_{L^p_\tau(\mathbb{R}^n)}+ \|f\|^\spadesuit_{B^{s,\tau}_{p,\infty}(\mathbb{R}^n)}
\\
 & \ls 	\|f \|_{L^p_\tau (\mathbb{R}^n)} + \sup_{P\in\mathcal{Q}} \, \frac1{|P|^\tau} \, \sup_{0 < t< 2\min\{\ell(P),1\}}
 	t^{-s}\, \sup_{\frac t2\le|h|<t}\left[\int_P
 	|\Delta_h^1 f(x)|^p\,dx\right]^{\frac 1p}
 \\
 & \ls 	\|f \|_{L^p_\tau (\mathbb{R}^n)} + \sup_{P\in\mathcal{Q}} \, \frac1{|P|^\tau} \, \sup_{0 < |h|< 2}
|h|^{-s}\, \left[\int_P |\Delta_h^1 f(x)|^p\,dx\right]^{\frac 1p}
\\
&\ls
 \|f \|_{L^p_\tau (\mathbb{R}^n)} + B (f)
\end{align*}
as claimed.
Altogether we complete the proof of Lemma \ref{Guli2}.
\end{proof}

Now we are able to reformulate Proposition \ref{Guli}.
As a consequence of Lemma \ref{Guli2} and under the same restrictions as there  we obtain
\[
\left\|f\right\|_{B^{s,\tau}_{p,\infty,{\rm unif}} (\mathbb{R}^n)}  \sim   \sup_{Q\in \mathcal{Q}, |Q|= 1}
|Q|^{-\tau} \, \|f \|_{L^p (Q)} +
\sup_{0<|h | < 1} \, |h|^{-s}\, \sup_{Q\in \mathcal{Q}, |Q|\le 1} \, |Q|^{-\tau}\,
\| \Delta ^1_h f\|_{L^p (Q)} < \infty
\, .
\]	
Observing $ \|f \|_{L^p_\tau (\mathbb{R}^n)} \le  \|f \|_{L^\infty (\mathbb{R}^n)}$ and
\[
\sup_{|h|>1} \,|h|^{-s} \sup_{r\in(0,1)} r^{s-n/p} \sup_{x\in \mathbb{R}^n}\, \left\{\int_{B(x,r)} \, |f(y+h)-f(y)|^p dy \right\}^{1/p} \le 2 \, \|f\|_{L^\infty (\mathbb{R}^n)}\,,
\]
we finally arrive at the following conclusion.

\begin{theorem}\label{Guli3}
	Let $p\in(1,\infty)$ and  $s\in(0,\frac 1p]$. Let $\tau:= \frac 1p - \frac sn$ and $u:=n/s$.
	Then $f \in M(B^{s}_{p,1} (\mathbb{R}^n), B^{s}_{p,\infty} (\mathbb{R}^n))$
if and only if $f \in L^\infty (\mathbb{R}^n) \cap  B^{s,\tau}_{p,\infty,{\rm unif}} (\mathbb{R}^n)$.
Moreover,
\[
\left\|f\right\|_{M(B^{s}_{p,1} (\mathbb{R}^n), B^{s}_{p,\infty} (\mathbb{R}^n))} \sim \, \|f \|_{L^\infty (\mathbb{R}^n)} +
\sup_{0<|h | < 1} \, |h|^{-s}\, \|f(\, \cdot \, +h)-f(\, \cdot\, )\|_{\cm^u_p (\mathbb{R}^n)} \, ,
\]	
where the   positive equivalent constants are independent of $f$.
\end{theorem}

\begin{corollary}\label{Guli4}
	Let $p\in(1,\infty)$, $q \in (0,\infty]$, and  $s\in(0,\frac 1p]$. Let $\tilde{p} \in [p,\infty)$ and
	 $\tilde{s} \in (0, \frac {1}{\tilde{p}})$. If $f \in M(B^{s}_{p,1} (\mathbb{R}^n), B^{s}_{p,\infty} (\mathbb{R}^n))$, then
 $f \in M(B^{\tilde{s}}_{\tilde{p},q} (\mathbb{R}^n))$.
\end{corollary}

\begin{proof}
{\em Step 1.}	We consider the case $p=\tilde{p}$.
By Theorem \ref{Guli3}, we have $ M(B^{s}_{p,1} (\mathbb{R}^n), B^{s}_{p,\infty} (\mathbb{R}^n)) \hookrightarrow L^\infty (\mathbb{R}^n)$. Hence, if $ f\in  M(B^{s}_{p,1} (\mathbb{R}^n), B^{s}_{p,\infty} (\mathbb{R}^n))$, then
	 $f \in M( L^p (\mathbb{R}^n))$. This allows us to use the real interpolation. Let $\theta \in (0,1)$. It follows
\[
\left(B^{{s}}_{p,1}(\mathbb{R}^n), L^{p} (\mathbb{R}^n)\right)_{\theta,q} = B^{(1-\theta){s}}_{p,q} (\mathbb{R}^n)\quad \mbox{and}\quad
\left(B^{{s}}_{p,\infty}(\mathbb{R}^n), L^{p} (\mathbb{R}^n)\right)_{\theta,q} = B^{(1-\theta){s}}_{p,q} (\mathbb{R}^n)
\]	
see, e.g., \cite[Section 2.4.2, Theorem 2 and Remark 4]{t78b}. The claim in this case is now a consequence of standard properties of the real interpolation method.
\\
{\em Step 2.} Because   $f \in  M(B^{s}_{p,1} (\mathbb{R}^n), B^{s}_{p,\infty} (\mathbb{R}^n))$ implies $f\in  M(B^{(1-\theta){s}}_{p,p} (\mathbb{R}^n))$ and $f\in M(L^r(\mathbb{R}^n))$ for any $r\in(1,\infty)$, we can now  continue with the complex and the real interpolation as in the proof of Corollary \ref{Gulidual}.
\end{proof}

Let $E$ denote a domain in $\mathbb{R}^n$.
In case that the space $ A^s_{p,q} (E)$ is defined by restriction, as done, e.g.,  in Triebel's books,
there is a simple relation between ${\bf 1}_E \in  M(A^s_{p,q} (\mathbb{R}^n)) $ and  another problem, namely the extension by zero.
We put  $T_E: ~ g \mapsto {\bf 1}_E \, \cdot \, g$ with domain of definition given by a space $A^s_{p,q}(\mathbb{R}^n)$, $A \in \{B,F\}$.
Let $X$ be a quasi-Banach space. By $\cl (X)$ we denote the class of all continuous linear operators mapping $X$
into $X$. More generally, if also $Y$ is a quasi-Banach space,
we  denote by    $\cl (X,Y)$ the class of all continuous linear operators mapping $X$
into $Y$.
Of course, $T_E \in \cl (A^s_{p,q}(\mathbb{R}^n))$ if and only if $\cx_E \in M(A^s_{p,q}(\mathbb{R}^n))$ and the corresponding norms are equal.
By
\[
\ext_E f(x) := \left\{ \begin{array}{lll}
	f(x) &\quad & \mbox{if}\quad x\in E,\\
	0 && \mbox{otherwise}\, ,
\end{array}\right.
\]
 we denote the extension by zero.

\begin{proposition}
	Let $p\in[1, \infty)$, $ q\in[1,\infty]$ and $s\in(0, 1/p]$. Let $A \in \{B,F\}$.
	Then
	${\bf 1}_E \in  M(A^s_{p,q} (\mathbb{R}^n)) $ if and only if
	$T_E \in \cl (A^s_{p,q} (\mathbb{R}^n))$
	if and only if $\ext_E \in \cl (A^s_{p,q} (E),A^s_{p,q} (\mathbb{R}^n))$.
	Moreover, it holds
	\[
	\| \,  {\bf 1}_E\, \|_{M (A^s_{p,q} (\mathbb{R}^n))} =	\| \, T_E \, \|_{\cl (A^s_{p,q} (\mathbb{R}^n))}
= \|\,  \ext_E \, \|_{\cl(A^s_{p,q} (E), A^s_{p,q} (\mathbb{R}^n))}\, .
	\]
\end{proposition}

Extension by zero has been studied by Besov \cite{Be98} for domains satisfying a flexible horn condition.
But he was using a different definition for the spaces on domains. So the results in our situation might be different.
However, it would be of some interest if the
results of Besov \cite{Be98} carry over to the present situation.

\begin{remark}\label{remhalf}
The pointwise multiplier problem with respect to characteristic functions has some history.
In case of  the characteristic function of the half space the investigations started in 1961 at latest, see
Lions and Magenes \cite{LM1} ($B^s_{p,p}(\mathbb{R}^n)$). Later also Shamir \cite{Sh} ($H^s_p (\mathbb{R}^n)$),  Strichartz \cite{Str-67} ($H^s_p (\mathbb{R}^n)$),
Triebel \cite[Lemma 2.10.2]{t78b}, \cite[2.8.5]{t83}
($B^s_{p,q}(\mathbb{R}^n)$, $F^s_{p,q}(\mathbb{R}^n)$), Gulisashvili \cite{Gu1,Gu2}, Maz'ya and Shaposhnikova
\cite{MS09},
Franke \cite{Fra} ($F^s_{p,q}(\mathbb{R}^n)$), Frazier and Jawerth \cite{FJ90}, and Runst and Sickel \cite[4.6.3]{RS} have contributed.
More general characteristic functions have been considered by  Gulisashvili \cite{Gu1,Gu2},
Frazier, Jawerth \cite{FJ90},
 Runst, Sickel \cite[4.6.3]{RS}, Triebel \cite{t03,t06} ($d$-sets), Faraco and Rogers \cite{FR} (quasi-balls), Schneider, Vybiral \cite{SV} ($d$-sets) and Sickel \cite{s99b,Si21,Si23}.
\end{remark}

\noindent
{\bf Summary:} \
For  $p=1$	we have found a characterization of the
 pointwise multiplier space
 $M(B^s_{1,1}(\rn))$ which involves uniform
Besov-type spaces.
For any $p \in (1,\infty)$	we only have the chain of embeddings
\[
M\left(B^s_{p,p}(\rn)\right) \hookrightarrow M\left(B^s_{p,1}(\rn),B^s_{p,\infty}(\rn)\right) =
 L^\infty (\mathbb{R}^n) \cap  B^{s,\tau}_{p,\infty,{\rm unif}} (\mathbb{R}^n)
\hookrightarrow M\left(B^{\tilde{s}}_{p,p}(\rn)\right)\, ,
\]
where $0 < \tilde{s} < s \le  \frac 1p$.
Therefore, the investigation of the regularity of characteristic functions in uniform Besov-type spaces
is essentially equivalent   to an understanding of their pointwise multiplier properties.

%&&&&&&&&&&&&&&&&&&&&&&&&&&&&&&&&&&&&&&&&&&&&&&&&&&&&&&&&&&&&&&&&&&&&&&&
%&&&&&&&&&&&&&&&&&&&&&&&&&&&&&&&&&&&&&&&&&&&&&&&&&&&&&&&&&&&&&&&&&&&&&&&

\section{Characteristic functions with maximal regularity}
\label{Main}

%&&&&&&&&&&&&&&&&&&&&&&&&&&&&&&&&&&&&&&&&&&&&&&&&&&&&&&&&&&&&&&&&&&&&&&&&&
%&&&&&&&&&&&&&&&&&&&&&&&&&&&&&&&&&&&&&&&&&&&&&&&&&&&&&&&&&&&&&&&&&&&&&&&&&

In this section, we determine the smallest class $B^{s,\tau}_{p,q}(\mathbb{R}^n)$ containing non-trivial characteristic functions and then derive some
sufficient conditions on sets $E$ such that ${\bf 1}_E$
has the maximal regularity. To be precise, in Subsection \ref{Main1},
we recall
some related known results on  characteristic functions with maximal regularity
in the classical Besov and Triebel--Lizorkin spaces. In Subsection \ref{Main2},
we introduce the concept of characteristic functions of maximal regularity in Besov-type spaces with $\tau\in(0,\infty)$
and discuss some necessary conditions on such maximal regularity.
In Subsection \ref{Main3}, we prove that the
characteristic functions of elementary Lipschitz domains are of maximal regularity in Besov-type spaces. Finally, in
Subsection \ref{Main4}, we show that the characteristic functions of certain spiral type domains are not of maximal regularity.

%&&&&&&&&&&&&&&&&&&&&&&&&&&&&&&&&&&&&&&&&&&&&&&&&&&&&&&&&&&&&&&&&&&&&&&&
%&&&&&&&&&&&&&&&&&&&&&&&&&&&&&&&&&&&&&&&&&&&&&&&&&&&&&&&&&&&&&&&&&&&&&&&

\subsection{Characteristic functions with maximal regularity in the framework of the classical spaces}
\label{Main1}

%&&&&&&&&&&&&&&&&&&&&&&&&&&&&&&&&&&&&&&&&&&&&&&&&&&&&&&&&&&&&&&&&&&&&&&&&&
%&&&&&&&&&&&&&&&&&&&&&&&&&&&&&&&&&&&&&&&&&&&&&&&&&&&&&&&&&&&&&&&&&&&&&&&&&

First we recall some well-known results concerning the maximal regularity of those functions in case
when $\tau =0$.

\begin{proposition}\label{gul}
	Let $p,q\in[1,\infty)$.  Then there exists {\bf no} measurable subset
	$E \subset \mathbb{R}^n$ having positive Lebesgue measure  such that  $ {\bf 1}_E \in B^{1/p}_{p,q} (\mathbb{R}^n)$.
\end{proposition}

We refer to Gulisashvili \cite{Gu1,Gu2} and Sickel \cite{Si21} for the above proposition.
Hence, a positive result  within Besov spaces $B^{1/p}_{p,\infty}(\mathbb{R}^n)$ is best possible.
Because of the continuous embedding $F^{1/p}_{p,\infty}(\mathbb{R}^n)
\hookrightarrow B^{1/p}_{p,\infty}(\mathbb{R}^n)$, it makes sense to ask for membership of ${\bf 1}_E$ in
$F^{1/p}_{p,\infty}(\mathbb{R}^n)$. Here is the negative answer, see Sickel \cite{Si23}.

\begin{proposition}\label{max}
Let  $p\in[1,\infty)$.
Then there is {\bf no} set $E \subset \mathbb{R}^n$ having positive Lebesgue measure such that
${\bf 1}_E \in F^{1/p}_{p,\infty}(\mathbb{R}^n)$.
\end{proposition}

This means, considering both scales simultaneously, $B^{1/p}_{p,\infty}(\mathbb{R}^n)$ is
the best possible for a nontrivial characteristic function to belong to.
To formulate the next result we need more notation.

%&&&&&&&&&&&&&&&&&&&&&&&&&&&&&&&&&&&&&&&&&&&&&&&&&&&&&&&&&&&&&&&&&&&&&&&
%&&&&&&&&&&&&&&&&&&&&&&&&&&&&&&&&&&&&&&&&&&&&&&&&&&&&&&&&&&&&&&&&&&&&&&&

Recall that a locally integrable function $f:~\mathbb{R}^n\to \mathbb{R}$ is \emph{of bounded variation}, denoted by $f\in BV(\mathbb{R}^n)$,
if its first order partial derivatives (in the distributional sense) are bounded Borel measures. For any $f\in BV(\mathbb{R}^n)$, let $\|f\|_{BV(\mathbb{R}^n)}:=\sum_{j=1}^n  |\frac{\partial f}{\partial x_j}|$, where $|\frac{\partial f}{\partial x_j}|$ denotes the total variation of the measure $\frac{\partial f}{\partial x_j}$.
The space $BV (\mathbb{R}^n)\cap L^1 (\mathbb{R}^n)$ is endowed with the following norm
$$
 \| f \|_{BV (\mathbb{R}^n)\cap L^1 (\mathbb{R}^n)}:= \|f\|_{BV(\mathbb{R}^n)}  + \|f\|_{L^1 (\mathbb{R}^n)}.$$
Then the \emph{perimeter} of a Lebesgue measurable  set $E\subset \mathbb{R}^n$ is the quantity
\[
\per E := \|{\bf 1}_E\|_{BV(\mathbb{R}^n)}\, ,
\]
see, e.g., Evans and Gariepy \cite[Chapter~5]{EG}, Burago and Zalgaller \cite[Chapter 14]{BZ}, or Ziemer \cite[Chapter~5]{Zie}.

Next we recall the definition of the space
${\rm Lip}\,(1,1)(\mathbb{R}^n)$. For any $t\in(0,\infty)$, let $\omega_1 (f,t)$   denote the
first-order modulus
of smoothness of $f$ with respect to $L^1 (\mathbb{R}^n)$, namely
\[
\omega_1 (f,t):= \sup_{\{h\in\mathbb{R}^n:\, |h|<t\}}\, \int_{\mathbb{R}^n} \, |f(x+h)-f(x)|\, dx.
\]
The \emph{space   ${\rm Lip}\,(1,1)(\mathbb{R}^n)$} is defined to be the collection of all $f \in L^1 (\mathbb{R}^n)$ such that $$\sup_{t\in(0,\infty)} t^{-1}\, \omega_1 (f,t) < \infty.$$
The norm is given by
\[
 \|  f  \|_{{\rm Lip}\,(1,1)(\mathbb{R}^n)}:= \| f \|_{L^1 (\mathbb{R}^n)}+ \sup_{t\in(0,\infty)} t^{-1}\, \omega_1 (f,t).
\]
A well-known result by Hardy and Littlewood is $BV(\mathbb{R}) \cap L^1 (\mathbb{R})={\rm Lip}\,(1,1)(\mathbb{R})$.
The generalization to the case $n>1$ has been proved by Gulisashvili \cite{Gu1}.
For a set $E \subset \mathbb{R}^n$ and the shift $h \in \mathbb{R}^n$, we put
\begin{equation}\label{w040}
 E(h)=\left\{x\in E:~ ~ x+h \in E^\complement \right\}
 \end{equation}
 and
 \begin{equation} \label{w041}
 F(h):=\left\{x\in E^\complement:~ ~ x+h \in E \right\}.
\end{equation}

The proof of the following proposition, containing if and only if characterizations,  can be found in
Gulisashvili \cite{Gu1} and Sickel \cite{Si21}.

\begin{proposition}\label{prima0}
	Let $E \subset \mathbb{R}^n$ and $0< |E|<\infty$.
	Then  the following assertions are mutually equivalent:
	\begin{itemize}
		\item[{\rm(i)}]  $\per E < \infty $;
		\item[{\rm(ii)}]   $\sup_{\{h\in\mathbb{R}^n: |h|<1\}}\,  |h|^{-1}\, \{|E(h)| + |F(h)|\} < \infty$;
		\item[{\rm(iii)}]   ${\bf 1}_E \in BV (\mathbb{R}^n)$;
		\item[{\rm(iv)}]   ${\bf 1}_E \in {\rm Lip} (1,1) (\mathbb{R}^n)$;
		\item[{\rm(v)}]   ${\bf 1}_E \in B^{1/p_0}_{p_0,\infty} (\mathbb{R}^n)$ for some  $p_0\in[1,\infty)$;
		\item[{\rm(vi)}]
		${\bf 1}_E \in B^{1/p}_{p,\infty} (\mathbb{R}^n)$ for any  $p\in[1,\infty)$.
	\end{itemize}
\end{proposition}

It is interesting to notice that characteristic functions of maximal regularity cannot distinguish between Besov spaces  $B^{1/p}_{p,\infty}(\mathbb{R}^n)$ with different $p$.
Except for $n=1$ there is no embedding relation of these Besov spaces $B^{1/p}_{p,\infty}(\mathbb{R}^n)$, $p\in[1,\infty)$. Hence, there is no optimal space in a certain sense.
For $n=1$ the smallest space is given by $BV (\mathbb{R}) \cap L^1 (\mathbb{R})$.
In \cite{Si21,Si23} one can find a large number of examples of sets with
associated characteristic functions of maximal regularity. Let us only mention
here that  the characteristic functions of cubes and balls belong to the
collection of characteristic functions of maximal regularity.

%&&&&&&&&&&&&&&&&&&&&&&&&&&&&&&&&&&&&&&&&&&&&&&&&&&&&&&&&&&&&&&&&&&&&&&&
%&&&&&&&&&&&&&&&&&&&&&&&&&&&&&&&&&&&&&&&&&&&&&&&&&&&&&&&&&&&&&&&&&&&&&&&

\subsection{A first example}
\label{Main2}

%&&&&&&&&&&&&&&&&&&&&&&&&&&&&&&&&&&&&&&&&&&&&&&&&&&&&&&&&&&&&&&&&&&&&&&&&&
%&&&&&&&&&&&&&&&&&&&&&&&&&&&&&&&&&&&&&&&&&&&&&&&&&&&&&&&&&&&&&&&&&&&&&&&&&

Now we turn to the regularity within type spaces.
The example of the characteristic function $\cx$ of the cube $[0,1)^n$ has been treated in  \cite{ysy20},
which can be extended to the following general case.

\begin{proposition}\label{charact1}
	Let $Q$ be a cube in $\mathbb{R}^n$ and let ${\bf 1}_Q$ denote its characteristic function.
	Let $s \in \mathbb{R}$ and $p$, $q\in (0,\infty]$.
	\begin{enumerate}
		\item[{\rm(i)}] Let $\tau \in(1/p,\infty)$. Then
		${\bf 1}_Q \in B^{s,\tau}_{p,q}(\mathbb{R}^n)$ if and only if
		$s \le n (1/p - \tau )$.
		
		\item[{\rm(ii)}] Let $\tau \in [0, 1/p]$. Then
		${\bf 1}_Q \in B^{s,\tau}_{p,q}(\mathbb{R}^n)$ if and only if
		either
		\begin{align*}
			s = \frac 1p,\,  \qquad q=\infty, \quad \mbox{and}\quad  s \le n \left(\frac 1p - \tau \right)
		\end{align*}
		or
		\begin{align*}
			s < \frac 1p,\,  \qquad q\in(0, \infty], \quad  \mbox{and}
			\quad s \le n \left(\frac 1p - \tau \right).
		\end{align*}
		\end{enumerate}
\end{proposition}

The cube is not a $C^\infty$ domain. So it is not clear whether the associated characteristic function
has the maximal regularity. But, as mentioned above, under the restriction $\tau=0$, it has.
For that reason, as an alternative example, we look for the characteristic function of a ball $B$  which  obviously is  the characteristic function of a $C^\infty$ domain.

\begin{proposition}\label{charact2}
Let $B$ be a ball  in $\mathbb{R}^n$ and let ${\bf 1}_B$ denote its characteristic function.
	Let $s \in \mathbb{R}$, $p$, $q\in (0,\infty]$, and $\tau \in [0,\infty)$.
	Then
	\[
	{\bf 1}_B \in B^{s,\tau}_{p,q}(\mathbb{R}^n) \quad \Longleftrightarrow \quad 	{\bf 1}_Q \in B^{s,\tau}_{p,q}(\mathbb{R}^n),
	\]
where $Q$ is any cube in $\mathbb{R}^n$.
\end{proposition}

\begin{proof}
	The claim follows by  obvious modifications of the arguments used in
the proof of  \cite[Theorem 2.1]{ysy20}; we omit the details.
\end{proof}

\begin{remark}
Haroske and   Triebel \cite{HT23} have used a different method to
describe the regularity of $\mathcal{X}:={\mathbf 1}_{[0,1)^n}$ in the framework of Morrey smoothness spaces.
\end{remark}

For a moment we  turn to general necessary conditions.
Here we concentrate on $\tau \in(0,\infty)$.

\begin{theorem}\label{wichtig10}
Let $E \subset \mathbb{R}^n$ be a bounded nontrivial  measurable set.
 Let  $s\in \mathbb{R}$, $p \in [1,\infty]$ ($p<\infty$ in the case of $F^{s,\tau}_{p,q}(\mathbb{R}^n)$), and $\tau \in (0,\infty)$.
\begin{itemize}
		\item[{\rm(i)}]  Then  ${\bf 1}_E \not\in B^{1/p, \tau}_{p,q} (\mathbb{R}^n)$ for all $q\in(0,\infty)$.
\item[{\rm (ii)}]	Then  ${\bf 1}_E \not\in F^{1/p, \tau}_{p,q} (\mathbb{R}^n)$ for all $q\in [1,\infty]$.
\item[{\rm (iii)}] Assume that $s> n(\frac 1p - \tau)$.
Then  ${\bf 1}_E \not\in B^{s, \tau}_{p,q} (\mathbb{R}^n)$ for all $q\in [1,\infty]$.
\item[{\rm (iv)}]  Assume that $s> n(\frac 1p - \tau)$.
Then  ${\bf 1}_E \not\in F^{s, \tau}_{p,q} (\mathbb{R}^n)$ for all $q\in [1,\infty]$.
\end{itemize}
\end{theorem}

The proof of Theorem \ref{wichtig10} will be given at the end of this subsection.
Based on this theorem and Propositions \ref{charact1} and \ref{charact2},
we introduce the following notion.

\begin{definition}\label{MAX}
Let  $s\in \mathbb{R}$, $p \in [1,\infty]$, and $\tau \in (0,\infty)$.
Let $E \subset \mathbb{R}^n$ be a nontrivial bounded and measurable set. The function ${\bf 1}_E$ is called a \emph{characteristic function of maximal regularity} if it satisfies the following   conditions:
\begin{itemize}
\item when $\tau\in(0,  \frac{n-1}{np}]$, ${\bf 1}_E \in B^{1/p,\tau}_{p,\infty} (\mathbb{R}^n)$;
\item when $\tau\in(\frac{n-1}{np}, \infty)$,   ${\bf 1}_E \in B^{s,\tau}_{p,q} (\mathbb{R}^n)$ with  $s=n   (\frac 1p - \tau )$ and all $q \in [1,\infty]$;
\end{itemize}
In such a case, for brevity, we shall also  say that $E$ has the maximal regularity.
\end{definition}

For short, ${\bf 1}_E$ is of maximal regularity if
\[
{\bf 1}_E \in B^{s,\tau}_{p,q}(\mathbb{R}^n) \quad \Longleftrightarrow \quad 	{\bf 1}_Q \in B^{s,\tau}_{p,q}(\mathbb{R}^n)
\]
for all admissible $s,\tau, p,q$.
Observe that, by Remark \ref{grund}(iii), ${\bf 1}_E \in B^{s,\tau}_{p,q} (\mathbb{R}^n)$ for all $q \in [1,\infty]$
implies ${\bf 1}_E \in F^{s,\tau}_{p,q} (\mathbb{R}^n)$ for all $q \in [1,\infty]$.
Below we will give an example of a domain $E$ such that
${\bf 1}_E  \in B^{1/p,0}_{p,\infty}(\mathbb{R}^n)$ for all $p \in [1,\infty)$,
but
${\bf 1}_E  \not\in B^{1/p,\tau}_{p,\infty}(\mathbb{R}^n)$ for all   $p \in [1,\infty)$ and all $\tau \in(0,\infty)$.
This shows that in this case the maximal regularity has a much stronger version
than  the classical case  where $
B^{1/p,0}_{p,\infty}(\mathbb{R}^n)$ for all $p  \in [1,\infty)$ is the best possible.

In view of Definition \ref{MAX}, it is of interest to understand counterparts
of the equivalence of assertions (iv) and (v) in Proposition \ref{prima0} in the more
general $\tau$-spaces.
First we have the following result, which is a simple consequence of Proposition  \ref{t4.7}.

\begin{corollary}\label{t4.8}
	Let $p\in [1, \infty]$,  $q=\infty$, $s\in(0,1/p ]$, and $\tau \in[0,1/p].$
	Let $E$ be a bounded measurable set in $\mathbb{R}^n$. Then
	${\bf 1}_E \in B^{s,\tau}_{p,\infty}(\mathbb{R}^n)$ holds if and only if
	${\bf 1}_E \in B^{sp,\tau p}_{1,\infty}(\mathbb{R}^n)$.
\end{corollary}

\begin{proof}	First we prove
${\bf 1}_E \in B^{1/p,\tau}_{p,\infty}(\mathbb{R}^n)$ implies
	${\bf 1}_E \in B^{1,\tau p}_{1,\infty}(\mathbb{R}^n)$ by two steps.\\
	{\em	Step 1.} Let $s\in(0,1)$ and $M=1$. Then it is enough to use the identity
	\begin{align*}
		& \dsup_{P\in\mathcal{Q}}\frac1{|P|^\tau} \sup_{0 < t < 2\min\{\ell(P),1\}}
		\, t^{-s} \, \sup_{\frac t2\le|h|<t}\left[\int_P|\Delta_h^1 {\bf 1}_E(x)|^p\,dx\right]^{\frac 1p}
		\\
		& \quad =   \left[\sup_{P\in\mathcal{Q}}\frac1{|P|^{\tau p}} \, \sup_{0 < t < 2\min\{\ell(P),1\}}
		\, t^{-sp} \, \sup_{\frac t2\le|h|<t}\int_P
		|\Delta_h^1 {\bf 1}_E(x)|\,dx\right]^{\frac 1p},
	\end{align*}	
	where we have applied $|\Delta_h^1 {\bf 1}_E(x)| \in \{0,1\}$ for any $x\in\mathbb{R}^n$ and therefore
	\begin{equation}\label{diff-eq}
	\int_P|\Delta_h^1 {\bf 1}_E(x)|^p\,dx = \int_P|\Delta_h^1 {\bf 1}_E(x)|\,dx.
	\end{equation}	
	{\em	Step 2.} Let  $s=1/p$ for some $p \in (1,\infty)$. To characterize this space
	$B^{1,\tau p}_{1,\infty}(\mathbb{R}^n)$,  we need second order differences (see Proposition \ref{t4.7}).
	The above argument works as well but the interpretation as parts of the corresponding norms is no longer true.
	However, if
	\[
	\sup_{P\in\mathcal{Q}}\frac1{|P|^\tau}  \, \sup_{0 < t < 2\min\{\ell(P),1\}}
	\, t^{-\frac1p} \, \sup_{\frac t2\le|h|<t}\left[\int_P|\Delta_h^1 {\bf 1}_E(x)|^p\,dx\right]^{\frac 1p} <\infty,
	\]
	then, by \eqref{diff-eq} again, we find that
	\[
	\left[	\sup_{P\in\mathcal{Q}}\frac1{|P|^{\tau p}} \, \sup_{0 < t < 2\min\{\ell(P),1\}}
	\, t^{-1} \, \sup_{\frac t2\le|h|<t}\dint_P
	|\Delta_h^1 {\bf 1}_E(x)|\,dx\right]^{\frac 1p} < \infty
	\]
	and therefore
	\[
	\left[	\sup_{P\in\mathcal{Q}}\frac1{|P|^{\tau p}} \, \sup_{0 < t < 2\min\{\ell(P),1\}}
	\, t^{-1} \, \sup_{\frac t2\le|h|<t}\int_P
	|\Delta_h^2 {\bf 1}_E(x)|\,dx\right]^{\frac1p} < \infty.
	\]
	This proves that	${\bf 1}_E \in B^{1/p,\tau}_{p,\infty}(\mathbb{R}^n)$ implies
	${\bf 1}_E \in B^{1,\tau p}_{1,\infty}(\mathbb{R}^n)$.
	
	Now we assume
	${\bf 1}_E \in B^{1,\tau p}_{1,\infty}(\mathbb{R}^n)$. Then
	\[
	\left[	\sup_{P\in\mathcal{Q}}\frac1{|P|^{\tau p}} \, \sup_{0 < t < 2\min\{\ell(P),1\}}
	\, t^{-1} \, \sup_{\frac t2\le|h|<t}\int_P
	|\Delta_h^2 {\bf 1}_E(x)|\,dx\right]^{\frac 1p} < \infty
	\]
	follows. Obviously $|\Delta_h^2 {\bf 1}_E(x)|\in \{0,1,2\}$ for any $x\in\mathbb{R}^n$. This implies
	\[
	|\Delta_h^2 {\bf 1}_E(x)| \le |\Delta_h^2 {\bf 1}_E(x)|^p \le 2^p\, |\Delta_h^2 {\bf 1}_E(x)|
	\]
	and therefore
	\[
	\left[	\sup_{P\in\mathcal{Q}}\frac1{|P|^{\tau p}} \, \sup_{0 < t < 2\min\{\ell(P),1\}}
	\, t^{-1} \, \sup_{\frac t2\le|h|<t}\int_P
	|\Delta_h^2 {\bf 1}_E(x)|^p \,dx\right]^{\frac 1p} < \infty.
	\]
	For this situation we can apply Proposition \ref{t4.7} with $M=2$ and conclude that
	${\bf 1}_E \in B^{1/p,\tau}_{p,\infty}(\mathbb{R}^n)$. This finishes the proof of Corollary \ref{t4.8}.
\end{proof}

\begin{remark}
 Let us discuss consequences for characteristic functions.
If $ p\in[1, \infty)$ and $\tau\in(0,\frac{n-1}{n}]$,
then ${\bf 1}_E \in B^{1,\tau}_{1,\infty}(\mathbb{R}^n)$ implies
 ${\bf 1}_E \in B^{1/p,\tau/p}_{p,\infty}(\mathbb{R}^n)$.
If $\tau\in(\frac{n-1}{n}, 1)$ and $s= n(1-\tau)$, then
${\bf 1}_E \in B^{s,\tau}_{1,\infty}(\mathbb{R}^n)$ implies
 ${\bf 1}_E \in B^{s/p,\tau/p}_{p,\infty}(\mathbb{R}^n)$.
 Hence, in both  cases it is enough to check the maximal regularity for $p=1$.
\end{remark}

The proof  of Theorem \ref{wichtig10}(i)
 will be a simple consequence of the following two observations.
 We shall use the notation from Proposition \ref{wav-type2}.

\begin{lemma}\label{support}
Let $f \in L^\infty(\mathbb{R}^n)$	be compactly supported.
\begin{itemize}
\item[{\rm (i)}] Let $s\in \mathbb{R}$, $p,q \in [1,\infty]$, and $\tau \in (0,\infty)$.
Then
$\|\lz(f)\|_{{b}^{s,\tau}_{p,q}(\mathbb{R}^n)}<\infty $ if and only if
\[
\sup_{\{P\in\mathcal{Q}: |P|\le 1\}}\frac1{|P|^{\tau}}\left\{\sum_{j=\max\{j_P,0\}}^\infty
2^{j(s+\frac n2-\frac np)q} \sum_{i=1}^{2^n-1}
\left[\sum_{\{m\in\mathbb{Z}^n:\ Q_{j,m}\subset P\}}
|\lambda_{i,j,m}  (f)|^p\right]^{\frac qp}\right\}^{\frac 1q}<\infty.
\]
\item[{\rm (ii)}]  Let $s\in \mathbb{R}$, $p \in [1,\infty)$, $q \in [1,\infty]$, and $\tau \in (0,\infty)$.
Then
$\|\lz(f)\|_{{f}^{s,\tau}_{p,q}(\mathbb{R}^n)}<\infty $ if and only if
\begin{align*}
&\sup_{\{P\in\mathcal{Q}: |P|\le 1\}}\frac1{|P|^{\tau}}\left\| \, \left[
\sum_{j=\max\{j_P,0\}}^\infty \sum_{i=1}^{2^n-1}
2^{j(s+\frac n2)q} \right.\right.\\
&\quad\left.\left.\times \sum_{\{m\in\mathbb{Z}^n:\ Q_{j,m}\subset P\}}
|\lambda_{i,j,m}(f) \, \cx (2^j\, \cdot\, -m)|^q \right]^{\frac 1q}\right\|_{L^p(P)} <\infty .
\end{align*}
\end{itemize}
\end{lemma} 	
	
\begin{proof}
	We fix a  wavelet system 	which is admissible for $A^{s,\tau}_{p,q}(\mathbb{R}^n)$ in the sense of
	Proposition \ref{wav-type2}. By
the compactness of the supports of the generators $\phi$ and $\psi_{i}$,
	 we find   that $|\lambda_{i,j,m}  (f)|>0$ implies that
\[
Q_{j,m} \subset \bigcup_{\ell = 1}^N Q_{0,k_{\ell}} =: \Omega
\]
for an appropriate number $N$ and a  fixed set $\{k_{1}, \ldots , k_{N}\} \subset \mathbb{Z}^n$.
If $P$ is a dyadic cube with $|P|>1$, then
\begin{align*}
& \left\| \left[
\sum_{j=0}^\infty \sum_{i=1}^{2^n-1}
2^{j(s+\frac n2)q} \sum_{\{m\in\mathbb{Z}^n:\ Q_{j,m}\subset P\}}
|\lambda_{i,j,m}(f) \, \mathcal{X} (2^j\cdot -m)|^q \right]^{\frac 1q}\right\|_{L^p(P)}
\\
&\quad\le   \sum_{\{\ell\in\mathbb{Z}^n:\ Q_{0,\ell}\subset P \cap \Omega\}} \left\|\left[
\sum_{j=0}^\infty  \sum_{i=1}^{2^n-1}
2^{j(s+\frac n2)q}\right.\right.\\
&\quad\quad\left.\left.\times \sum_{\{m\in\mathbb{Z}^n:\ Q_{j,m}\subset Q_{0,\ell}\}}
|\lambda_{i,j,m}(f) \, \mathcal{X}(2^j \cdot -m)|^q \right]^{\frac 1q}\right\|_{L^p(Q_{0,\ell})}.
\end{align*}
On the other hand, since  $f \in L^\infty(\mathbb{R}^n)$ and $\supp f$ is compact, we conclude that
\[
	\sup_{\{P\in\mathcal{Q}:~ |P|\ge 1\}}\frac1{|P|^{\tau}}
\left\{\sum_{\{m\in\mathbb{Z}^n:\ Q_{0,m}\subset P\}} 	| \lambda_m (f) |^p\right\}^{\frac 1p} < \infty.
\]
From these, the claim in the $F$-case follows. Similarly one can deal with the $B$-case.
\end{proof}

\begin{lemma}\label{support2}
Let $f \in L^\infty(\mathbb{R}^n)$	be compactly supported.
 Let $s\in \mathbb{R}$, $q \in [1,\infty]$, and $\tau_0,\tau_1 \in [0,\infty)$ such that   	
$\tau_0 < \tau_1$.

\begin{itemize}
\item[{\rm(i)}] 	 Let  $p \in [1,\infty]$.
 If $f \in  B^{s,\tau_1}_{p,q}(\mathbb{R}^n)$,
 then $f \in B^{s,\tau_0}_{p,q}(\mathbb{R}^n)$ follows, in particular one has $f \in B^s_{p,q}(\mathbb{R}^n)$.

\item[{\rm(ii)}]   Let  $p \in [1,\infty)$.
 If $f \in  F^{s,\tau_1}_{p,q}(\mathbb{R}^n)$,
 then $f \in F^{s,\tau_0}_{p,q}(\mathbb{R}^n)$ follows, in particular one has $f \in F^s_{p,q}(\mathbb{R}^n)$.
\end{itemize}
 \end{lemma}

\begin{proof}
This follows immediately from the wavelet characterization
of $B^{s,\tau}_{p,q}(\mathbb{R}^n)$ and $ B^s_{p,q}(\mathbb{R}^n)$ and similarly of $F^{s,\tau}_{p,q}(\mathbb{R}^n)$ and $ F^s_{p,q}(\mathbb{R}^n)$, see Definitions  \ref{dts} and \ref{dts2} with $\tau =0$, and Lemma \ref{support}.
\end{proof}

\begin{proof}[Proof of Theorem \ref{wichtig10}]
{\em Step 1.} Proofs of (i) and (ii). These assertions are immediate consequences of Lemma \ref{support2} 	and Propositions \ref{gul} and \ref{max}.

{\em Step 2.}	Proofs of (iii) and (iv).
These assertions are direct consequences of Proposition \ref{embc} because ${\bf 1}_E$ is not continuous
(more exactly   the equivalence class of ${\bf 1}_E$ has no continuous representative).
\end{proof}

\begin{remark}
By means of Lemma \ref{support2}  we are now able to compare the results for
$\tau =0$, the classical case, and those with $\tau >0$. It follows that
Propositions \ref{charact1} and  \ref{charact2} are improvements of the assertions that
${\bf 1}_Q, {\bf 1}_B \in B^{1/p}_{p,\infty} (\mathbb{R}^n)$ for any $p\in[1,\infty]$.
\end{remark}

Now we turn to consequences for pointwise multipliers. As usual, we define $q'$ by $\frac 1q + \frac 1{q'}=1$ for any $q\in[1,\infty].$

\begin{theorem}\label{multi1}
Let $E \subset \mathbb{R}^n$, $|E|>0$,  be a set of maximal regularity.
\begin{itemize}
\item[{\rm (i)}] Let $p\in(1,\infty)$.
Then ${\bf 1}_E   \in M(B^{1/p}_{p,1} (\mathbb{R}^n), B^{1/p}_{p,\infty} (\mathbb{R}^n))$.
\item[{\rm (ii)}] Let $p=1$, $q\in(0, \infty]$, and $s \in (0,1)$.	Then ${\bf 1}_E  \in M(B^{s}_{1,q} (\mathbb{R}^n))$.
\item[{\rm (iii)}] Let $p\in(1,\infty)$, $q\in(0, \infty]$, and $s \in (-1/p',1/p)$.
Then ${\bf 1}_E \in   M(B^{s}_{p,q} (\mathbb{R}^n))$.
\item[{\rm (iv)}] Let $p=1$.	Then ${\bf 1}_E  \not\in M(B^{1}_{1,1} (\mathbb{R}^n))$.
\end{itemize}
\end{theorem}

\begin{proof}
	Part (i) is a consequence of Theorem \ref{Guli3} and of ${\bf 1}_E \in B^{1/p, \tau}_{p,\infty} (\mathbb{R}^n)$, $\tau = \frac{n-1}{np}$.
Part (ii) follows from $ {\bf 1}_E  \in B^{s, \tau}_{1,1} (\mathbb{R}^n)$, $s= n ( 1- \tau)$, $\tau>1 -\frac 1n$ and
Proposition \ref{FAchar}(i).	
Concerning (iii) we first argue by using (ii) and Corollary \ref{Gulidual}. This covers all cases with $s>0$.
Next we shall use the formula
\begin{equation}\label{dual1}
M\left(B^s_{p,q}(\mathbb{R}^n)\right)
=M\left(B^{-s}_{p',q'}(\mathbb{R}^n)\right),
\end{equation}
where $p,q\in[1,\infty]$, $s\in\mathbb{R}$, see \cite[Theorem 4.1]{lsyy23}.
The result for Besov spaces with $s=0$ can be derived by real interpolation between
 $B^s_{p,q}(\mathbb{R}^n)$, $s>0$, and $B^{\tilde{s}}_{p,q}(\mathbb{R}^n)$, $\tilde{s}<0$.
 Finally, if $s<0$ and $q\in(0,1)$, we also use real interpolation between  $B^s_{p,1}(\mathbb{R}^n)$, $s>0$, and $B^{\tilde{s}}_{p,1}(\mathbb{R}^n)$, $\tilde{s}<0$.
Part (iv) follows from ${\bf 1}_E \not \in B^{1, \tau}_{1,1,{\rm unif}} (\mathbb{R}^n)$, see Theorem \ref{wichtig10}.
\end{proof}	

\begin{remark}
In case of ${\bf 1}_Q$  Theorem \ref{multi1} is known, also in case of the characteristic function of the half space ${\bf 1}_{\rr^n_+}$, see
Remark \ref{remhalf}. In this case one even knows that the result is in {\em if and only if} form, see \cite[Theorem~4.6.3/1]{RS}.
The  novelty in Theorem \ref{multi1} consists in the generality with respect to $E$.
\end{remark}

%&&&&&&&&&&&&&&&&&&&&&&&&&&&&&&&&&&&&&&&&&&&&&&&&&&&&&&&&&&&&
%&&&&&&&&&&&&&&&&&&&&&&&&&&&&&&&&&&&&&&&&&&&&&&&&&&&&&&&&&&&&

\subsection{Examples of domains with maximal regularity}
\label{Main3}

%&&&&&&&&&&&&&&&&&&&&&&&&&&&&&&&&&&&&&&&&&&&&&&&&&&&&&&&&&&&&
%&&&&&&&&&&&&&&&&&&&&&&&&&&&&&&&&&&&&&&&&&&&&&&&&&&&&&&&&&&&&

Interesting examples are given by  elementary Lipschitz domains.
Concerning these domains we shall make use of the following definition, see Burenkov \cite[Section 4.3]{Bu}.
In this definition  we shall apply the notation $x= (x',x_n),$
$x'=(x_1, \ldots, x_{n-1}) \in \mathbb{R}^{n-1}$, $x_n \in \mathbb{R}$.

\begin{definition}\label{deflip}
	Let integer $n\ge 2$.
	An open bounded set $E\subset \mathbb{R}^n$ is called an \emph{elementary Lipschitz domain}
	if there exist a function $\varphi$ and  numbers $0 < D_1 \le D_2 < \infty$, $a_1, \ldots a_n$, $b_1, \ldots , b_{n-1}$, and
$L$ such that
	\begin{itemize}
		\item[(i)]
	  $\diam (E) \le D_2$;
		\item[(ii)]
	  $W:= \{{x'} \in \mathbb{R}^{n-1}: \: a_i < x_i < b_i, \: i=1, \ldots , n-1\}$;
		\item[(iii)]
			$E = \{ x \in \mathbb{R}^n: ~ a_n< x_n< \varphi(x'), \: x' \in W\}$;
		\item[(iv)]
		 $a_n + D_1 \le \varphi (x'),  \forall\, x' \in W$;
		\item[(v)]
		 $|\varphi (x') - \varphi (y')| \le L \,
		|x' - y'|, \forall\, x',   y' \in W$.
	\end{itemize}
\end{definition}
The number $L$ in Definition \ref{deflip}(v) is called the \emph{Lipschitz constant} of the function $\varphi$.

\begin{theorem}\label{Lipschitz1}
Characteristic functions of elementary Lipschitz domains  have
the maximal regularity in the sense of Definition \ref{MAX}.
\end{theorem}

\begin{proof}
	We will apply wavelet arguments. Let $E$ denote our  elementary Lipschitz domain.
	Based on Lemma \ref{support} it will be enough to deal  with
	\[
	\sup_{\{P\in\mathcal{Q}: |P|\le 1\}}\frac1{|P|^{\tau}}\left\{\sum_{j=\max\{j_P,0\}}^\infty
	2^{j(s+\frac n2-\frac np)q} \sum_{i=1}^{2^n-1}
	\left[\sum_{\{m\in\mathbb{Z}^n:\ Q_{j,m}\subset P\}}
	|\lambda_{i,j,m}  ({\bf 1}_E)|^p\right]^{\frac qp}\right\}^{\frac 1q}.
	\]
Because of the moment conditions of $\psi_{i,j,k}$ in \eqref{moment},
		we conclude that
		$\langle {\bf 1}_E, \psi_{i,j,k}\rangle=0$ if either $\supp \psi_{i,j,k}
		\subset \overline{E}$ or
		$\supp \psi_{i,j,k}  \cap \overline{E}= \emptyset$.
For any $j\in\mathbb{N}_0$, we then define
		\[
		\Omega_j:= \{r\in \mathbb{Z}^n: \
		\exists\ i\in\{1,\ldots,2^n-1\}\ \mbox{such that}\,
		\supp \psi_{i,j,r} \cap  \partial E \neq \emptyset\}.
		\]
Assume now that $P\in\cq$ with $|P|\le 1$ and $P \cap \partial E \neq \emptyset$.
We may write
		$P:= Q_{\ell,m}$ with $\ell \in \mathbb{N}_0$ and $m \in\mathbb{Z}^n$.
We claim that
		\begin{equation}\label{w014}
		\left|\Omega_j \cap \{r \in \mathbb{Z}^n:~ Q_{j,r} \subset P\}\right|\ls  2^{(j-\ell)(n-1)}, \quad \forall\,j \in\{\ell,\ell+1,\ldots\},
		\end{equation}
with the implicit positive constants independent of $P$.
Therefore we consider the description of the elementary Lipschitz domain as given in Definition		
\ref{deflip}. By $L_\varphi$ we denote  the Lipschitz constant of the function $\varphi$, see
Definition \ref{deflip}(v).
As a consequence of this definition we find that $P \cap E$ is an elementary Lipschitz domain as well with a boundary, given by parts of hyperplanes and parts of the graph of $\varphi$, see the following  pictures  for various possibilities.

\smallskip

\begin{minipage}{0.4\textwidth}
	\includegraphics[height=4.5cm]{{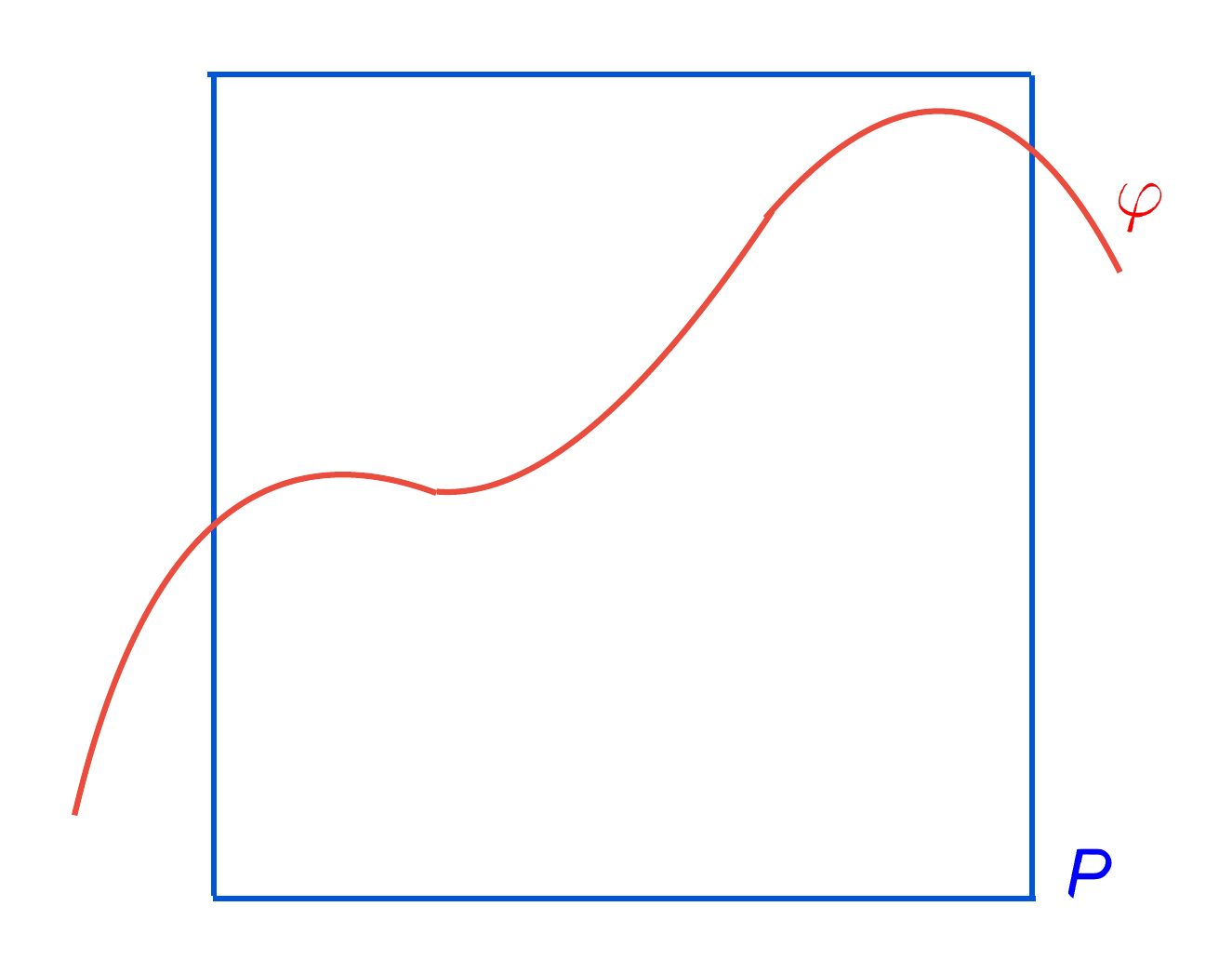}}
\end{minipage}\hfill
\begin{minipage}{0.4\textwidth}
	\includegraphics[height=5.4cm]{{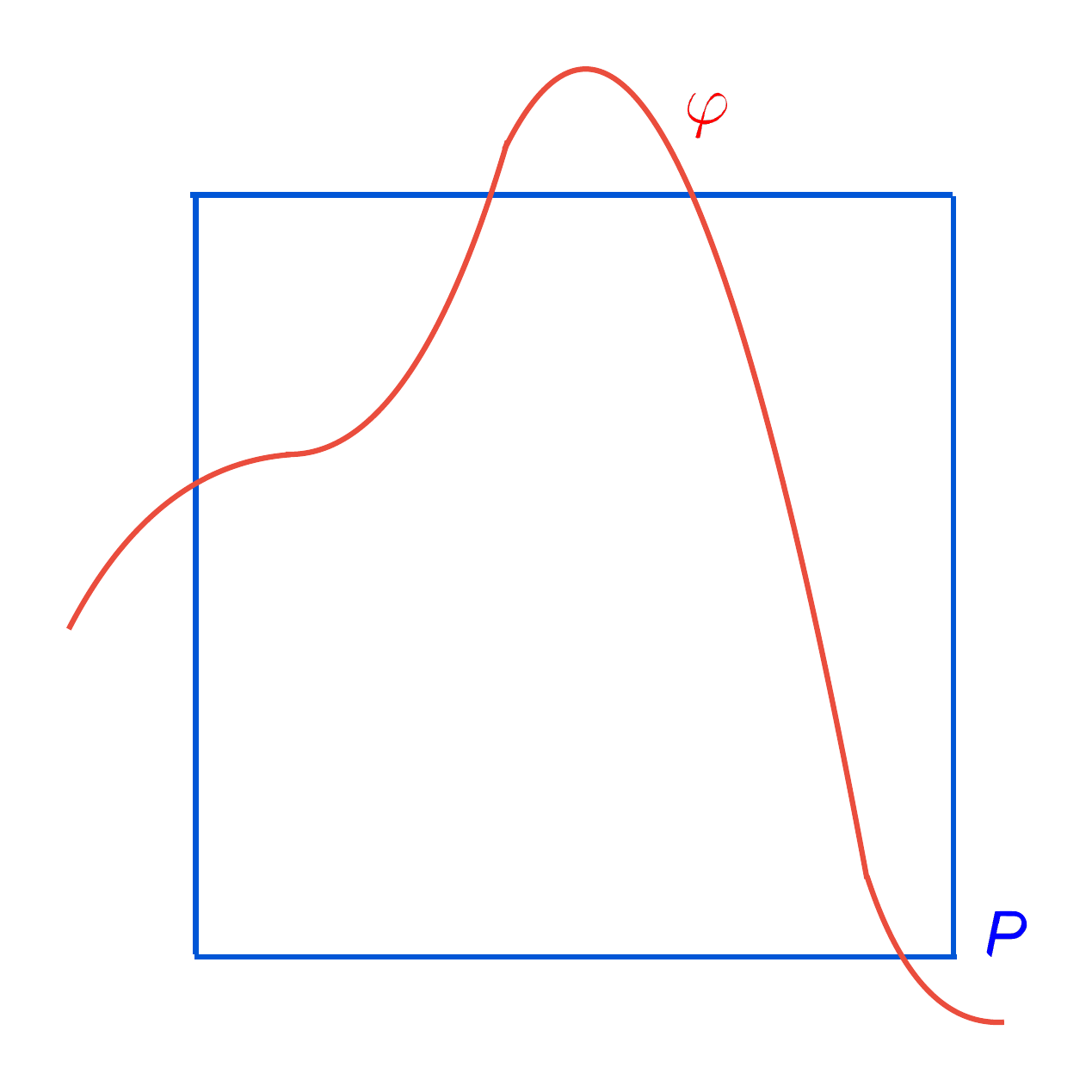}}
\end{minipage}

\smallskip

We only consider some sort of a standard situation for sufficiently small dyadic cubes
(all other situations can be discussed with the same sort of arguments)
\[
P \cap E =\left\{ x \in \mathbb{R}^n: ~    x' \in W_P,\  A_{P,\ell,m} < x_n < \min \{\varphi(x'),\ 2^{-\ell}(m_n+1)\}\right\},
\]
where $A_{P,\ell,m}$ is a constant depending on $P$, $\ell$, $m$, and $a_n$ and
\[
W_P:= \left\{{x'} \in \mathbb{R}^{n-1}:\   2^{-\ell}m_i < x_i < 2^{-\ell} (m_i + 1),  \ i=1, \ldots , n-1\right\}.
\]
We split $W_P$ into $2^{(j-\ell)(n-1)}$ dyadic subcubes of level $j$, denoted by $\widetilde{Q}_{j,t}$ with some  $t\in\mathbb{Z}^{n-1}$. These  cubes    are not all of interest.
We fix one of these cubes, $\widetilde{Q}_{j,t} \subset W_P$. Then we select
points $x^0, x^1$ such that
\[
\varphi (x^0) = \max \left\{\varphi (x'):~ x' \in \widetilde{Q}_{j,t}\right\}\quad \mbox{and}\quad
\varphi (x^1) = \min \left\{\varphi (x'):~ x' \in \widetilde{Q}_{j,t}\right\}.
\]
It follows
\[
|\varphi (x^1)- \varphi (x^0)|\le  \sqrt{n-1} L_\varphi \, 2^{-j}.
\]
To cover the set $\{ x \in \mathbb{R}^n: ~ x' \in \widetilde{Q}_{j,t}, ~ \varphi (x^1)\le x_n \le \varphi (x^0)\}$ by dyadic cubes in $\mathbb{R}^n$ of level $j$, we need at most $\lceil\sqrt{n-1}(L_\varphi + 2)\rceil$ of those cubes.
Doing that for each cube $\widetilde{Q}_{j,t}$ we obtain \eqref{w014}.
Next we observe that
\begin{align}\label{eqxx}
	|\langle {\bf 1}_E, \psi_{i,j,k} \rangle| \le 2^{jn/2} \int_{E} | \psi_{i} (2^j x -k)| \, dx
	\ls  2^{-jn/2}.
\end{align}
Applying \eqref{w014} and \eqref{eqxx}, for $P= Q_{\ell,m}$, we find that

\begin{align*}
&\frac1{|P|^\tau}\left\{\sum_{j=\ell}^\infty 2^{j(s+\frac n2)q}\sum_{i=1}^{2^n-1}
\lf[\sum_{\{k\in\mathbb{Z}^n:~Q_{j,k} \subset P\}}2^{-jn}|\langle {\bf 1}_E, \psi_{i,j,k}\rangle|^p\right]^{\frac qp}\right\}^{\frac 1q}
\\
&\quad \ls  2^{\ell n\tau}\, 	\left\{\sum_{j=\ell}^\infty 2^{j(s-\frac np)q} \,  2^{(j-\ell)(n-1)\frac qp}\right\}^{\frac 1q}
\ls  2^{\ell[ n\tau - \frac{n-1}p]}\, 	
\left\{\sum_{j=\ell}^\infty 2^{j(s - \frac1p)q}\right\}^{\frac 1q}.
\end{align*}
If either  $s=1/p$ and $q= \infty$ or $s<1/p$ and $q$ arbitrary,  we conclude that
\begin{align*}%\label{eq-004}
\frac1{|P|^\tau}\left\{\sum_{j=\max\{j_P,0\}}^\infty2^{j(s+\frac n2)q}\sum_{i=1}^{2^n-1}
\lf[\sum_{\{k\in\mathbb{Z}^n:~Q_{j,k} \subset P\}}2^{-jn}|\langle {\bf 1}_E,
\psi_{i,j,k}\rangle|^p\right]^{\frac qp}\right\}^{\frac 1q}
\ls 2^{\ell (s+ n\tau - \frac np)},
\end{align*}
which is uniformly bounded in $\ell$ for $s+ n\tau - n/p\le 0$. This finishes the proof of Theorem \ref{Lipschitz1}.
\end{proof}

Of  central importance for the proof of Theorem \ref{Lipschitz1} is the inequality \eqref{w014}.
The following corollary is immediate.

\begin{corollary}
Let $E$ be a bounded domain. For any $j \in \mathbb{N}_0$, define	
\[
\Omega_j:= \left\{r\in \mathbb{Z}^n: \
\exists\ i\in\{1,\ldots,2^n-1\}\ \mbox{such that}\,
\supp \psi_{i,j,r} \cap  \partial E \neq \emptyset\right\}.
\]
Suppose that, for any dyadic cube   $P:= Q_{\ell,m}$  with $\ell \in \mathbb{N}_0$ and $m \in\mathbb{Z}^n$,
	\begin{equation}\label{w014c}
		\left|\Omega_j \cap \{r \in \mathbb{Z}^n:~ Q_{j,r} \subset P\}\right|\ls  2^{(j-\ell)(n-1)}, \quad \forall\ j \in\{\ell,\ell+1,\ldots\},
	\end{equation}
	with the implicit positive constant  independent of $P$. Then ${\bf 1}_E$ has the maximal regularity.
\end{corollary}

Besov-type spaces  and Triebel--Lizorkin-type spaces are invariant under rotations, translations, and reflections. In connection with Theorem \ref{Lipschitz1} this leads to the following.

\begin{corollary}\label{Lipschitz3}
	Let $E$ be a set such that $\overline{E} $  can be written as the union of the closures of
	a finite number  of pairwise disjoint domains $E_1, \ldots , E_N$ satisfying that, for any $j\in\{1, \ldots , N\}$,
	  $E_j$ is the image of an elementary Lipschitz domain
under a finite number of rotations, translations, and reflections.
	Then ${\bf 1}_E$ has the maximal regularity.
\end{corollary}

\begin{remark}
 Observe that for a set $E$ as in Corollary \ref{Lipschitz3} we have $|\partial E|=0$.
 The function spaces under consideration are distribution spaces, that is, for a regular
 distribution we do not distinguish functions, which coincide almost everywhere. Hence,
 for any measurable set  $E$ satisfying $|\partial E| =0$ we have
 ${\bf 1}_E = {\bf 1}_{\overline{E}} = {\bf 1}_{\mathring{E}}$.
\end{remark}

Corollary \ref{Lipschitz3}  makes clear that a set  $E$, not having maximal regularity,  must have a rather complicated boundary.
Before we   turn  to those cases we supplement Theorem \ref{Lipschitz1}
by a further sufficient condition.

\begin{theorem}\label{Lipschitz2}
Let $E$ be a bounded set and $p\in[1,\infty)$. Suppose that there exists a positive
constant $c$ such that  for all dyadic cubes $P$ satisfying
$|P|\le 1$ it holds
\begin{equation}\label{geo1}
\sup_{0 <t<1} \sup_{\{h\in\mathbb{R}^n:\, t/2<|h|<t\}} \left\| \Delta_h^1 {\bf 1}_E \right\|_{L^1 (P)}\le c\, t \, [\ell (P)]^{n-1}.
\end{equation}
\begin{itemize}
\item[{\rm(i)}]  If $\tau \in(0,\frac{n-1}{np}]$, then ${\bf 1}_E$ belongs to $B^{1/p,\tau}_{p,\infty} (\mathbb{R}^n)$.

\item[\rm{(ii)}] If $  \tau\in(\frac{n-1}{np},\frac 1p)$, then ${\bf 1}_E$ belongs to $B^{s,\tau}_{p,q} (\mathbb{R}^n)$ with $s:= n(\frac 1p - \tau)$ and  all $q\in [1,\infty]$.
\end{itemize}
\end{theorem}

\begin{proof}
We use the characterization by differences  given in Proposition \ref{t4.7}. For that reason we limit ourselves to $s>0$.

{\em Step 1.} 	Let $p\in(1, \infty)$  and  $\tau \in(0, \frac{n-1}{np}]$.
In this case, since $E$ is bounded,  by
 Proposition \ref{t4.7}  with $0< s=1/p< 1$, $q=\infty$, and $M=1$,  we only need to prove
\begin{align*}
\|f\|^\spadesuit_{B^{\frac1p,\tau}_{p,\infty} (\mathbb{R}^n) }&:=\sup_{P\in\mathcal{
				Q}} \, \frac1{|P|^\tau} \, \sup_{0 < t< 2\min\{\ell(P),1\}}
		t^{-s}\sup_{\frac t2\le|h|<t}\left[\int_P
		|{\bf 1}_E(x+h)-{\bf 1}_E (x) |^p\,dx\right]^{\frac 1p} < \infty.
\end{align*}
	Clearly, we only need to consider those dyadic   cubes $P$ such that $\inf_{x \in E} {\rm dist}\, (x, P) < 1$.
 Essentially it will be enough to consider cubes such that $\ell (P) \le 1$.
	For those cubes we define
$$
		E(P,h):= \left\{x \in E \cap P: ~ x+h \not\in E \right\}
$$
and
$$F(P,h) :=  \left\{x \in P: ~x \not\in E, ~x+h \in E \right\}.
$$
	Because of
	\[
	\dint_P
	|{\bf 1}_E(x+h)-{\bf 1}_E (x) |^p\,dx = |	E(P,h)|+  |	F(P,h)|
	\]
	it follows
	\[
	\|f\|^\spadesuit_{B^{\frac 1p,\tau}_{p,\infty} (\mathbb{R}^n)}  \ls \sup_{P \in \mathcal{Q},\, |P| \le 1} \frac{1}{\vert P \vert^{\tau}}    \sup_{0 < t < 2 \ell (P)} \,  t^{-1/p}\,   \sup_{\frac t2< |h| <t}\left[|E(P,h)|+  |F(P,h)| \right]^{\frac{1}{p}}.
	\]
	Now applying \eqref{geo1} we derive
	\[
	\|f\|^\spadesuit_{B^{1/p,\tau}_{p,\infty} (\mathbb{R}^n)}  \ls \sup_{\{P \in \mathcal{Q},\, |P|\le 1\}} \frac{1}{[\ell (P)]^{n\tau}}      [\ell(P)]^{(n-1)/p}   < \infty
	\]
	because of  $0 < n\tau \le \frac{n-1}{p}$.

If $p=1$, $s=1$, and $\tau\in(0,\frac{n-1}n]$,  then it is enough to observe that the second-order difference can be estimated from above by two first-order differences: for any $x,h\in\mathbb{R}^n$,
\[
\left|\Delta_h^2{\bf 1}_E (x)\right| \le  \left|\Delta_h^1{\bf 1}_E (x+h)\right| + \left|\Delta_h^1{\bf 1}_E (x)\right|.
\]
	This allows to reduce the estimate again to first-order differences and we can proceed as above.
	No further modifications are needed. Hence, (i) is proved.	
	
{\em Step 2.} Let $p\in[1, \infty)$ and $\tau\in(\frac{n-1}{np}, \frac 1p)$. Clearly,
in this case, $0 < s= n(\frac 1p - \tau) < 1$
	for all admissible $p$. Arguing as above we find
	\begin{align*}
		\|f\|^\spadesuit_{B^{s,\tau}_{p,q} (\mathbb{R}^n)} & \ls
		\sup_{\{P \in \mathcal{Q}, |P| \le 1\} } \frac{1}{\vert P \vert^{\tau}}
		\left\{\int_0^{2\ell(P)} t^{-sq}\sup_{\frac t2\le|h|<t}\left[|	E(P,h)|+  |	F(P,h)| \right]^{\frac qp}\,\frac{dt}{t}\right\}^{\frac 1q}\\
		& \ls
		\sup_{\{P \in \mathcal{Q}, | P| \le 1\} } \frac{1}{\vert P \vert^{\tau}}
		\left\{\int_0^{2\ell(P)} t^{(\frac 1p-s)q}\, [\ell(P)]^{(n-1)\frac qp}\,\frac{dt}{t}\right\}^{\frac 1q}\\
		& \ls
		\sup_{\{P \in \mathcal{Q}, |P| \le 1\} } \frac{1}{[\ell (P)]^{n\tau}} [\ell (P)]^{\frac{n-1}p}\,
		\left\{\int_0^{2\ell(P)} t^{(n\tau- \frac{n-1}{p})q}\,\frac{dt}{t}\right\}^{\frac 1q}.
	\end{align*}
	Since $n\tau > \frac{n-1}{p}$ we conclude
$$
		\|f\|^\spadesuit_{B^{s,\tau}_{p,q} (\mathbb{R}^n)} \ls
		\sup_{\{P \in \mathcal{Q}, |P | \le 1\}} \frac{[\ell(P)]^{\frac{n-1}p} \,[\ell (P)]^{n\tau- \frac{n-1}{p}} }{\vert \ell (P) \vert^{n\tau}} \ls 1.
$$
This proves (ii) and hence Theorem \ref{Lipschitz2}.
\end{proof}

In some sense the condition \eqref{geo1} is the counterpart of \eqref{w014c}.
It is an open problem whether these conditions are also necessary for the maximal regularity.
 That would pave the way for a full extension of Proposition \ref{prima0}.

Below we list some concrete examples of domains whose characteristic functions have the maximal regularity.

%&&&&&&&&&&&&&&&&&&&&&&&&&&&&&&&&&&&&&&&&&&&&&&&&&&&&&&&&&&&&&&&&&&&&&&&&&&&&&&&&&&&&&&&&&&&&&&&&
%%%%%%%%%%%%%%%%%%%%%%%%%%%%%%%%%%%%%%%%%%%%%%%%%%%%%%%%%%%%%%%%%%%%%%%%%%%%%%%%%%%%%%%%%%%%%%%%%%

\smallskip

\centerline{\bf Concrete examples}

\smallskip

\begin{minipage}{0.46\textwidth}
	\begin{center}\includegraphics[height=4.5cm]{{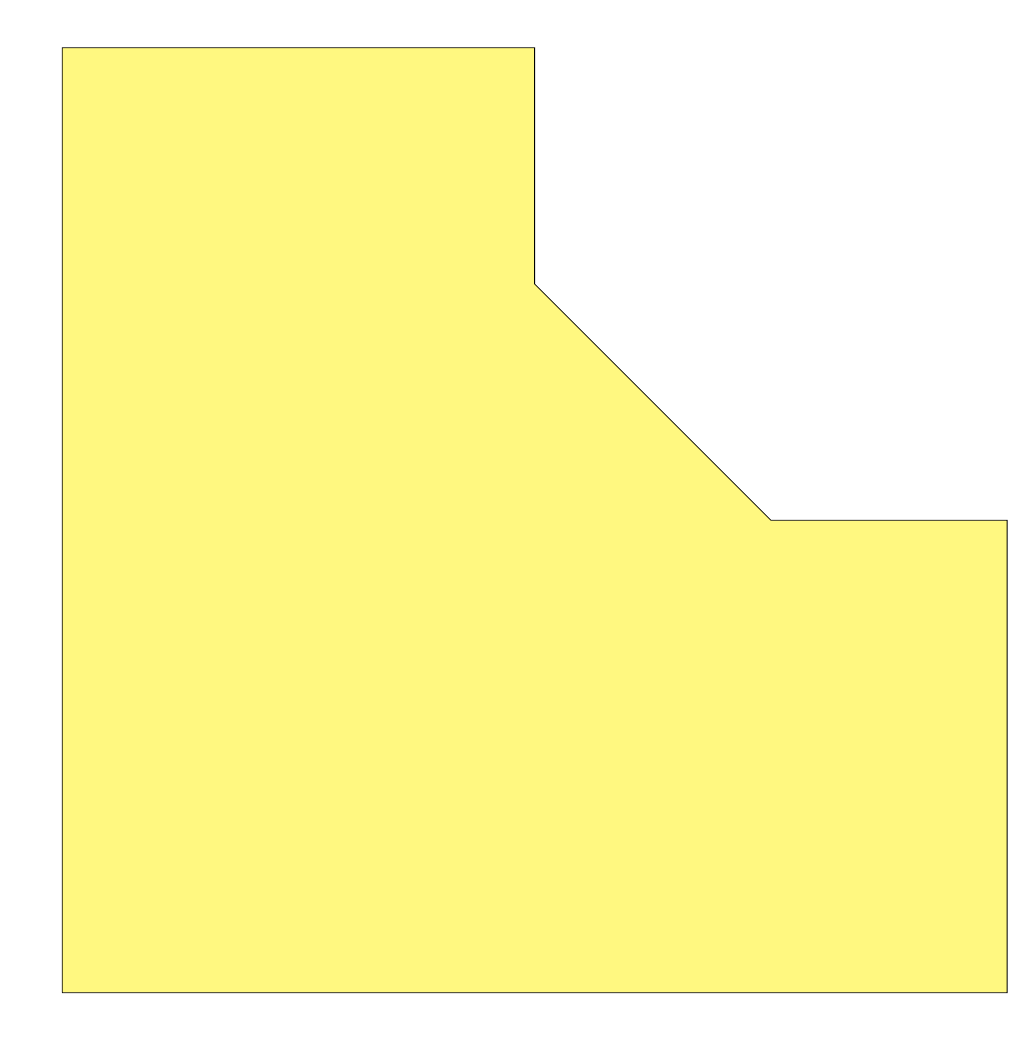}}\end{center}
\end{minipage}\hfill
\begin{minipage}{0.46\textwidth}
	\includegraphics[height=4.5cm]{{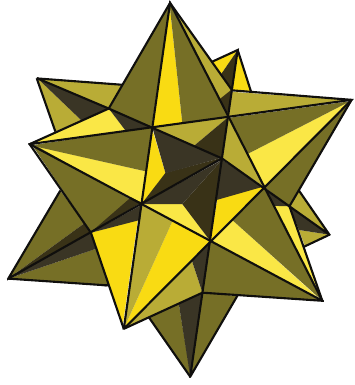}}
\end{minipage}

\smallskip

In the above two pictures, the domain with the polygonal boundary  on the left  is an elementary Lipschitz domain,  the icosahedron on the right   not, but it is a domain for which  Corollary \ref{Lipschitz3} applies.

The Lipschitz regularity of the boundary is not necessary for the maximal regularity of the associated  characteristic function. Here are two  examples.
First we  take  the domain $A \subset \mathbb{R}^2$ with boundary $\partial A$ given by the astroid.
Recall  that the  determining functional equation of the astroid is given by
\begin{equation}\label{rot1}
	x^{\frac 23} + y^{\frac 23}=1,\quad \forall\, x,y\in \mathbb{R} .
\end{equation}
Obviously, by \eqref{rot1}, the boundary $\partial A$ has the H\"older regularity $\alpha =2/3 $  and is
therefore not Lipschitz (in four isolated points).

\begin{minipage}{0.4\textwidth}
	\includegraphics[height=5.7cm]{{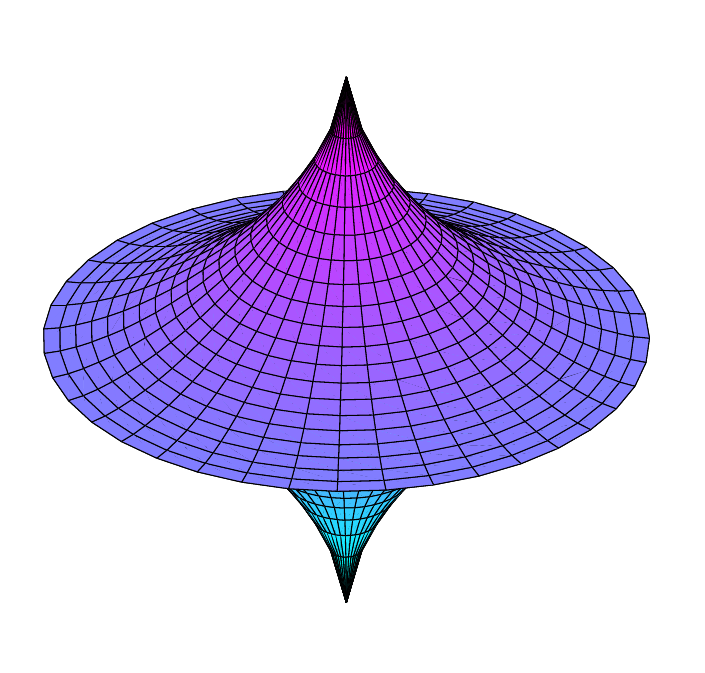}}
\end{minipage}\hfill
\begin{minipage}{0.4\textwidth}
	Now we consider the rotation of this curve around the $y$-axis resulting in the domain  $A_{{\rm rot}} \subset \mathbb{R}^3$ (see the left picture).
	Obviously the boundary $\partial A_{{\rm rot}}$ is  not Lipschitz in the north and the south poles
	and on the equator, so in infinitely many points. However, we can argue as follows.
Since we have in some sense a simple boundary it is not difficult to apply the same arguments as in the proof of Theorem \ref{Lipschitz1}.
\end{minipage}

Indeed, for any $j\in\mathbb{N}_0$, define
\[
\Omega_j:= \left\{r\in \mathbb{Z}^n: \
\exists\ i\in\{1,\ldots,2^n-1\}\ \mbox{such that}\,
\supp \psi_{i,j,r} \cap  \partial A_{{\rm rot}} \neq \emptyset\right\}.
\]
Assume now that $P\in\mathcal{Q}$ with $|P|\le 1$ and $P \cap \partial A_{{\rm rot}} \neq \emptyset$. We may write
$P:= Q_{\ell,m}$ with $\ell \in \mathbb{N}_0$ and $m \in\mathbb{Z}^n$.
Then we claim that
\begin{equation*}
	\left|\Omega_j \cap \{r \in \mathbb{Z}^n:~ Q_{j,r} \subset P\}\right|\ls  2^{(j-\ell)(n-1)}, \qquad \forall\,j \in\{\ell,\ell+1,\ldots\},
\end{equation*}
with the implicit  positive  constant independent of $P$, comparing with \eqref{w014}.
A proof can be given similar to that of \eqref{w014} in case of  elementary Lipschitz domains.
This further implies that ${\bf 1}_{ A_{{\rm rot}}}$ has the maximal regularity.

The next example is even simpler. Let $\varepsilon \in (0,1)$.
We define
\[
E_\varepsilon:= \left\{(x,y)\in \mathbb{R}^2: ~ -1 < x < 1, ~~ |x|^\varepsilon < y < 1 \right\}.
\]
See the following picture.

\smallskip

\begin{minipage}{0.56\textwidth}
	\includegraphics[width=6cm]{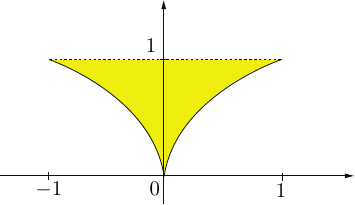}
\end{minipage}\hfill\begin{minipage}{0.4\textwidth}
The domain $E_\varepsilon$ has a boundary with
the H\"older regularity $\alpha = \varepsilon$.
So the H\"older regularity can be arbitrarily small.
However, $E_\varepsilon$ has the
maximal regularity. This can be shown by using the same
arguments as in case of $ A_{{\rm rot}}$.
\end{minipage}

\smallskip

%&&&&&&&&&&&&&&&&&&&&&&&&&&&&&&&&&&&&&&&&&&&&&&&&&&&&&&&&&&&&&&&&&&&&&&&
%&&&&&&&&&&&&&&&&&&&&&&&&&&&&&&&&&&&&&&&&&&&&&&&&&&&&&&&&&&&&&&&&&&&&&&&

\subsection{The characteristic function of a spiral type domain}
\label{Main10}

%&&&&&&&&&&&&&&&&&&&&&&&&&&&&&&&&&&&&&&&&&&&&&&&&&&&&&&&&&&&&&&&&&&&&&&&&&
%&&&&&&&&&&&&&&&&&&&&&&&&&&&&&&&&&&&&&&&&&&&&&&&&&&&&&&&&&&&&&&&&&&&&&&&&&

When searching for extremal functions in $A^{s,\tau}_{p,q}(\mathbb{R}^n)$ the characterization of $A^{s,\tau}_{p,q}(\mathbb{R}^n)$ given in
\eqref{new4} might be helpful. Since we are interested in the local situation two cases remain:
\begin{itemize}
	\item[(i)] $\|   f  \|^*_{A^{s,\tau}_{p,q} (\mathbb{R}^n)} = \, |P_0|^{-\tau} \, \|   f  \|_{A^s_{p,q} (P_0)}$ for one dyadic cube $P_0$;
		\item[(ii)] $\|   f  \|^*_{A^{s,\tau}_{p,q} (\mathbb{R}^n)} = \lim_{\ell \to \infty}\, |P_\ell|^{-\tau} \, \|   f  \|_{A^s_{p,q} (P_\ell)}$ for an appropriate sequence $\{P_\ell\}_{\ell\in\mathbb{N}}$ of  dyadic cubes such that $\lim_{\ell \to \infty} |P_\ell| =0$.
\end{itemize}
In  case (i) the decision whether  $\|   f  \|^*_{A^{s,\tau}_{p,q} (\mathbb{R}^n)} < \infty$ or not does not depend on $\tau$.
So we are back in the classical situation of $A^{s}_{p,q} (\mathbb{R}^n)$.
Clearly, in (ii) the decision whether  $\|   f  \|^*_{A^{s,\tau}_{p,q} (\mathbb{R}^n)} < \infty$ or not will depend on $\tau$.
This is the interesting situation for us.
Next, some calculations show that in \eqref{new4} we may replace the collection of dyadic cubes $\mathcal{Q}$
by all cubes having edges parallel to the axes. Then we will end up with an equivalent quantity.
By working with this quantity,
to simplify the construction of our extremal functions, we may think on a situation where we additionally have $\bigcap_{\ell} P_\ell = \{x^0\}$ for some $x^0 \in \mathbb{R}^n$.
This means, we will look for functions with exactly one singular point.
 Since our function of interest should be a characteristic function
 we look for a domain such that the boundary is $C^\infty$ except one point.
For us this motivates to deal with the so-called spiral type domains. We will pick  up this idea  here in this subsection
and later again in Subsection \ref{Main9}.

More exactly, we consider a domain lying between two spirals (red and blue) with endpoints in the origin.
 The domain $E$,  defined in polar coordinates $(r,\varphi)$, is given by
\[
E:= \left\{(r,\varphi):\  \frac{1}{(\varphi + \pi) \, \log^2 (\varphi+\pi)} \le r <
\frac 1{\varphi\, \log^2 \varphi}, \  2\pi \le \varphi < \infty\right\}.
\]

\smallskip

\begin{minipage}{0.4\textwidth}
	\includegraphics[width=6cm]{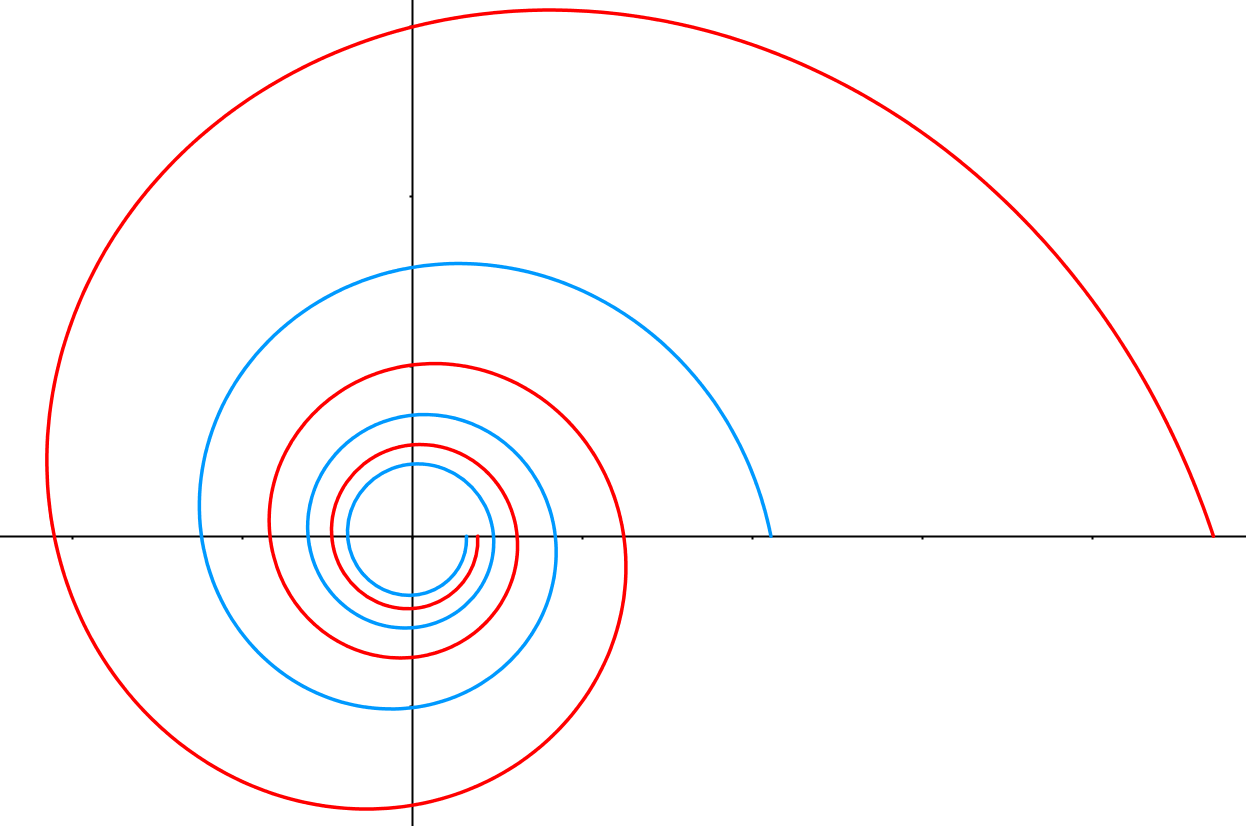}
\end{minipage}\hfill
\begin{minipage}{0.4\textwidth}
In the left picture only
a part of the domain $E$ is shown. Using the notation
of the proof of Theorem \ref{spirale0}, in particular formula \eqref{w044},
it represents the set $\bigcup_{i=1}^3E_i$.
The power of the Logarithm is chosen to be $2$
because this guarantees that the length of the boundary $\partial E$ is finite.
\end{minipage}

\smallskip

We will calculate the smoothness of the characteristic function ${\bf 1}_E$.
First we will deal with the classical situation, namely $\tau =0$.
Here we will work with the $\delta$-neighborhood of the boundary $\partial E$, namely
\[
 (\partial E)^\delta =  \left\{x\in \mathbb{R}^2: ~ \dist (x, \partial E)< \delta \right\}, \quad \forall\,\delta\in(0,\infty).
\]

\begin{theorem}\label{spirale0}
The characteristic function ${\bf 1}_E$ of this spiral-type domain $E$ belongs to
$ B^{1/p}_{p, \infty} (\mathbb{R}^2)$ for all $p \in [1,\infty]$.
\end{theorem}

\begin{proof}
For any $k \in \mathbb{N}$, let
\begin{equation}\label{w044}
E_k := \left\{(r,\varphi):~ \frac{1}{(\varphi + \pi) \, \log^2 (\varphi+\pi)} \le r <
\frac 1{\varphi\, \log^2 \varphi} , \  2\pi k \le \varphi < 2\pi (k+1) \right\}.
\end{equation}
Observe that, for any $k\in\mathbb{N}$, the width of $E_k$ satisfies $\sim [k\, \log (k+1)]^{-2}$ and therefore
\begin{equation}\label{vol}
	|E_k| \sim \frac{1}{ k^3\, \log^4 (k+1)}.
\end{equation}
Now we estimate the volume of the sets $E(h)$ and $F(h)$ defined in
\eqref{w040} and \eqref{w041}, respectively.
To simplify the argument we concentrate on the shifts
$|h|= 2^{-j}$ for any $j\in \mathbb{N}$.
If $|h|$ is smaller than the width of $E_k$ it is obvious
that $(\partial E_k)^{|h|}$ satisfies
\[
  \left|(\partial E_k)^{|h|}\right|\ls |h|\, \frac{1}{k\, \log^2 (k+1)}.
\]
For any $j \in \mathbb{N}$, define
\[
 \mathcal{E}_j := \bigcup_{\gfz{k \in \mathbb{N}}{k \, \log^2 (k+1) < 2^j}} E_k.
\]
Then we find that
\[
 \left|( \partial \mathcal{E}_j)^{2^{-j}}\right|\le 2^{-j}\,  \sum_{\gfz{k \in \mathbb{N}}{k \, \log^2 (k+1) < 2^j}} \, \frac{1}{k\, \log^2 (k+1)} .
\]
Next we observe that there exists some natural number $j_0$ such that, for any $j \ge j_0$,
\[
\frac{2^j}{j^2} \, \log^2 \left(\frac{2^j}{j^2} \right) \sim 2^j.
\]
Hence
\[
 \sum_{\gfz{k \in \mathbb{N}}{k \, \log^2 (k+1) < 2^j}} \, \frac{1}{k\, \log^2 (k+1)}
 \ls \sum_{k=j_0}^{c2^j/j^2} \, \frac{1}{k\, \log^2(k+1)}
\]
for an appropriate positive constant $c$. But
\[
\sum_{k=j_0}^{c2^j/j^2} \, \frac{1}{k\, \log^2 (k+1)} \ls 1.
\]
Taking into account that $\sum_{k=1}^{j_0} \, \frac{1}{k\, \log^2 (k+1)}= C < \infty$ this further  implies
\begin{equation}\label{w042}
\left|(\partial \mathcal{E}_j)^{2^{-j}}\right|\ls 2^{-j}, \quad\forall j \in \mathbb{N}.
\end{equation}
If $k\in \mathbb{N}$ is chosen such that $k \, \log^2 (k+1) \ge 2^j$ we observe that
\[
E_k \subset B\left({\bf 0}, [2\pi k \log^2 (2\pi k)]^{-1}\right) \subset B({\bf 0}, 2^{-j}).
\]
Obviously this leads to the estimate
\[
 \left|\bigcup_{\gfz{k \in \mathbb{N}}{k \, \log^2 (k+1) \ge  2^j}} E_k\right| \ls |B(\textbf{0}, 2^{-j})|\ls 2^{-2j}.
\]
Together with \eqref{w042} this yields
\[
\left|( \partial E)^{2^{-j}}\right|\ls 2^{-j}, \quad\forall j \in \mathbb{N}.
\]
By the monotonicity of $|(\partial E)^{\delta}|$ on $\delta$, we obtain
\begin{equation*}%\label{w043}
\left|( \partial E)^{\delta}\right|\ls \delta, \quad\forall\, \delta\in(0,\infty).
\end{equation*}
Since $E(h) \cup F(h) \subset ( \partial E)^{\delta}$ whenever $|h|= \delta$,
an application of Proposition \ref{prima0} yields the desired conclusion of Theorem \ref{spirale0}.
 \end{proof}

\begin{remark}
 We will give an  alternative proof of Theorem \ref{spirale0}.
For any $s\in(0,\infty)$, let ${\mathcal{H}}^{s}$ denote the $s$-dimensional Hausdorff measure.
 We will use   the identity
  \[
  \per E = \inf \left[ \liminf_{j\to \infty}\, {\mathcal{H}}^{n-1} (\partial M_j) \right],
  \]
  where the infimum  is taken  over all sequences $\{M_j\}_j$ of sets with a smooth boundary (or polyhedra) such that
  \[
  \lim_{j\to \infty} \left\| {\bf 1}_E - {\bf 1}_{M_j}\right\|_{L^1 (K)}=0
  \]
 for any compact set $K\subset\mathbb{R}^n$,  see, for instance, \cite[Sections 14.1 and 14.6]{BZ}.
Observe that the length $\ell (\partial E)$ of the boundary of $E$ and the measure of $E$ are finite.
	Now it is easy to construct a sequence $\{\Omega_j\}_{j\in \mathbb{N}}$ of domains with
polygonal boundary such that
\[
\lim_{j\to \infty} \, \| {\bf 1}_E - {\bf 1}_{\Omega_j}\|_{L^1(\mathbb{R}^2)} = 0
\]	
and
\[
\sup_{j \in \mathbb{N}} \, \ell (\partial \Omega_j) \ls 1.
\]
This implies $\per E < \infty$. Then one can  apply Proposition \ref{prima0}
to obtain the desired conclusion of Theorem \ref{spirale0}.
\end{remark}

Theorem \ref{spirale0} combined with Proposition  \ref{prima0} shows that ${\bf 1}_E$ has
 the maximal regularity within the classical Besov and Triebel--Lizorkin spaces.
Now we show that this is no longer true within type spaces, namely in the sense of Definition \ref{MAX}.

\begin{theorem}\label{spirale1}
Let $p \in [1,\infty)$, $\tau \in [0,1/p]$, and $s\in(0,1/p]$.
Then ${\bf 1}_E \in B^{s,\tau}_{p, \infty} (\mathbb{R}^2)$ if and only if $\tau + s \le 1/p$.
\end{theorem}

\begin{proof}
We proceed as in the proof of Theorem \ref{Lipschitz2} and use the characterization by differences  given in Proposition \ref{t4.7}. Hence we can work with $s>0$  only. The proof is divided into three steps.
Step 1 will be devoted to the sufficiency part of the theorem in case $p=1$.
In Step 2 we will investigate the necessity part in case $p=1$.
Finally, Step 3 will be devoted to the case  $p\in(1,\infty)$.
\\
{\em Step 1.} In this step, we want to show the finiteness of
\begin{eqnarray}\label{XXX}
	{\rm I}:=	\sup_{P\in\mathcal{Q}} \, \frac1{|P|^\tau} \, \sup_{0 < t< 2\min\{\ell(P),1\}}
	t^{-s}\sup_{\frac t2\le|h|<t} \int_P 		|{\bf 1}_E(x+h)-{\bf 1}_E (x)|\,dx
\end{eqnarray}
when $\tau + s \le 1$. To this end, we use the notation of the proof of Theorem \ref{spirale0}, in particular, the sets $E_k$, $k \in \mathbb{N}$, are defined  in \eqref{w044}.
In addition we notice that there exists a positive constant $c$ such that, for any $\ell \in\mathbb{N}$,
\[
|E_k \cap Q_{\ell,\mathbf{0}}| =0 \quad \mbox{if}\quad k \le c\,\frac{ 2^{\ell}}{\ell ^2}.
\]
Essentially this remains true if we replace $Q_{\ell,\mathbf{0}}$ by $\{x \in \mathbb{R}^2:~ {\rm dist}\,(x,Q_{\ell,\textbf{0}}) <2^{-\ell} \}$, that is,
\begin{equation}\label{w030b}
	\left|E_k \cap \left\{x \in \mathbb{R}^2:~ {\rm dist}\,(x,Q_{\ell,\mathbf{0}}) <2^{-\ell} \right\}\right| = 0 \quad \mbox{if}\quad k \le C\, \frac{2^{\ell}}{\ell^2}
\end{equation}
for some positive constant $C$.
Clearly, to estimate ${\rm I}$, we only need to consider   cubes $P$ such that $\inf_{x \in E} \dist (x, P) < 1$.
In addition we observe that, whenever $\textbf{0} \not\in \overline {P}$, then $P\cap E$ is a set with a Lipschitz boundary. These observations, together with the definition of $E$,
further imply that we only need to deal with the dyadic cubes $Q_{\ell,{\bf0}}$ with integer $\ell\in\mathbb{N}$.	
As above we shall use the notation
$$E(P,h): = \left\{x \in E \cap P: ~ x+h \not\in E \right\}$$
and
$$F(P,h):= \left\{x \in P: ~x \not\in E, ~x+h \in E \right\}.$$
It follows
\begin{align*}
	& \dint_P
	|{\bf 1_E}(x+h)-{\bf 1}_E (x) |\,dx \\
&\quad= |	E(P,h)|+  |	F(P,h)|\nonumber
	\\
	\nonumber
	&\quad = \sum_{k=1}^\infty \left( |\{x \in P: ~ x \in E_k \, , x+h \not\in E\}| +
	|\{x \in P: ~ x+h \in E_k \, , x \not\in E\}|\right)
	\\
	& \quad= \sum_{k=1}^\infty \left(|E_k (P,h) | + |F_k (P,h)| \right),\nonumber
\end{align*}
where, for any $k \in \mathbb{N}$,
\begin{align*}
	{E}_k(P,h) & : = \left\{x \in E_k \cap P: ~ x+h \not\in E \right\}
\end{align*}
and
\begin{align*}
	F_k(P,h) & : =  \left\{x \in P: ~x \not\in E, ~x+h \in E_k \right\}.
\end{align*}	
Now we fix $\ell\in\mathbb{N}$ and put $P = Q_{\ell,{\bf0}}.$ Let $h\in\mathbb{R}^n$ with
$|h| \in [2^{-j}, 2^{-j+1}]$ for some $j \ge \ell$.
Clearly, this and \eqref{w030b} yield
\begin{equation}\label{w034}
	\left\| \Delta_h^1 \mathbf{1}_E \right\|_{L^1(P)} = \sum_{k=C\, 2^{\ell}/\ell^2}^\infty \left(  |	E_k(P,h)|+  |	 F_k(P,h)|\right),
\end{equation}	
where $C$ is the same as in \eqref{w030b}.
We have to consider two cases: (a) the width of $E_k$ is less than $|h|$;
(b) the width of $E_k$ is larger than or equal to $|h|$.\\
{\bf Case (a)}. Observe that the width of $E_k$ is comparable to $[k \log (k+1)]^{-2}$.
Hence, in this case, we only need to consider
$k\gtrsim \frac{2^{j/2}}{j}.$
We will apply the trivial estimate
\begin{equation}\label{w031}
	|	E_k(P,h)|  \ls |E_k| \ls  \frac{1}{ k^3\, \log^4 (k+1)},
\end{equation}
see \eqref{vol}.
Similarly we obtain
\begin{equation}\label{w031b}
	|	F_k(P,h)|  \ls |E_k|\ls  \frac{1}{ k^3\, \log^4 (k+1)}\, .
\end{equation}
\\
{\bf Case (b)}. In this case,  we only need to consider $C \, \frac{2^{\ell}}{\ell^2} \le k \ls \frac{2^{j/2}}{j}.$
Observe that between $E_k$ and $E_{k+1}$ there is a strip of almost the same size as $E_k$  but belonging to $E^\complement$.
Hence, in this situation we have	
\begin{equation}\label{w032}
	|	E_k(P,h)| \sim |h| \, \frac 1{k \log^2 (k+1)}\, .
\end{equation}	
Clearly, by the same argument we obtain	
\begin{equation}\label{w033}
	|	F_k(P,h)| \sim |h| \, \frac 1{k \log^2 (k+1)}\, .
\end{equation}

In case $\ell \le j < 2\ell$	we insert \eqref{w031} and \eqref{w031b} into \eqref{w034} and conclude that
\begin{equation}\label{w035}
	\left\| \, \Delta_h^1 {\bf 1}_E \,\right\|_{L^1(P)} \ls \sum_{k=C\, 2^{\ell}/\ell^2}^\infty \frac{1}{ k^3\, \log^4 (k+1)}
\ls  2^{-2\ell}.
\end{equation}	
Here we have used
\[
\int_{a}^\infty 	\frac{1}{x^3 \log^4(x+1)}\, dx\ls \frac{1}{\log^4 (a+1)}\, \int_{a}^\infty 	
\frac{1}{x^3}\, dx  \ls \frac{1}{a^2\, \log^4 (a+1)}\, .
\]
In case of $j\ge 2\ell$ we insert \eqref{w031}  through \eqref{w033} into \eqref{w034}  and find that
\begin{align}\label{w036a}
	\left\| \, \Delta_h^1 {\bf 1}_E \,\right\|_{L^1(P)} & \ls \sum_{k=C\, 2^{\ell}/\ell^2}^{2^{j/2}/j}
	\frac{|h|}{k \, \log^2 (k+1)}  + \sum_{k=2^{j/2}/j}^\infty \frac{1}{k^3 \log^4(k+1)}\nonumber
	\\
	& \ls 2^{-j}   + \frac{1}{j^2  \,  2^{j}}  \ls  2^{-j}\sim |h|.
\end{align}	
From \eqref{w035} we deduce that
\begin{equation}\label{w038}
	\sup_{2^{-2\ell} < t< 2^{-\ell + 1}}\,
	t^{-s}\sup_{t/2 <|h|<t} \, \dint_P
	|{\bf 1}_E(x+h)-{\bf 1}_E (x) |\,dx \ls  2^{2\ell (s-1)}
\end{equation}
because of $s >0$.
Furthermore, \eqref{w036a} yields
\begin{equation}\label{w037b}
	\sup_{0 < t\le  2^{-2\ell}}\, 	t^{-s} \sup_{t/2 <|h|<t} \, \int_P |{\bf 1}_E(x+h)-{\bf 1}_E (x) |\,dx
	\ls  2^{2\ell(s-1)}
\end{equation}
because of $s\le 1$.
Combining \eqref{w038} and \eqref{w037b}, we obtain
\[
{\rm I}\ls \sup_{\ell\in \mathbb{N}} \, 2^{2 \ell \tau}  2^{2\ell(s-1)} .
\]
Consequently  $\tau + s-1 \le 0$ is sufficient for ${\bf 1}_E \in B^{s,\tau}_{1,\infty}(\mathbb{R}^2)$.
\\
{\em Step 2.} We show necessity of  $\tau + s-1 \le 0$ when $p=1$.
Therefore we assume that  $\tau + s-1 >0$.
Let $|h|= 2^{-j}$ be such that $C \, \frac{2^{\ell}}{\ell^2} \le k \ls \frac{2^{j/2}}{j}.$
Based on \eqref{w034}, \eqref{w032}, and \eqref{w033}, we conclude that
\[
\left\| \, \Delta_h^1 {\bf 1}_E \,\right\|_{L^1(P)} \ge
\sum_{k=C\, 2^{\ell}/\ell^2}^{2^{j/2}/j} [|	E_k(P,h)|+|F_k(P,h)|]\gtrsim 2^{-j} \, \sum_{k=C\, 2^{\ell}/\ell^2}^{2^{j/2}/j} \frac{1}{ k\, \log^2 (k+1)}.
\]
Taking $j= 2\ell$ we obtain
\[
\sum_{k=C\, 2^{\ell}/\ell^2}^{2^{\ell}/(2\ell)} \frac{1}{ k\, \log^2 (k+1)} \gtrsim \frac{1}{\ell^2}\,
\sum_{k=C\, 2^{\ell}/\ell^2}^{2^{\ell}/(2\ell)} \frac{1}{k} \gtrsim \frac{1}{\ell^2}\, \log \ell\, .
\]
For that reason we find that
\begin{align*}
	{\rm I}
	&\gtrsim  	\sup_{\ell\in\mathbb{N}} \, 2^{2\ell\tau} 2^{2\ell (s-1)} \frac{\log \ell}{\ell^2}.
\end{align*}
If $\tau +s -1 >0$ it follows that the right-hand side of the above inequality is infinity.
\\
{\em Step 3.} Let $p\in (1,\infty)$. Then Corollary \ref{t4.8} yields that ${\bf 1}_E \in B^{sp,\tau p}_{1, \infty} (\mathbb{R}^2) $
if and only if ${\bf 1}_E \in B^{s,\tau }_{p, \infty} (\mathbb{R}^2) $.
Observing that $\tau p +sp -1 \le 0$ if and only if $ \tau   + s - \frac 1p \le 0$,
we then complete the proof of Theorem \ref{spirale1}.
\end{proof}
	
\begin{remark}
 Notice that $s=1/p$ is possible only in case $\tau =0$. For $\tau >0$ there is a serious restriction.
 This shows that this domain is not of maximal regularity.
 \end{remark}

Finally, let us have a look onto the rather interesting pointwise multiplier properties of ${\bf 1}_E$.

\begin{theorem}\label{multi2}
Let $p\in[1,\infty)$, $q\in(0,\infty]$ and $s\in \rr\setminus\{0\}$. Then  the function ${\bf 1}_E$ does not belong to $M(B^s_{p,q}(\mathbb{R}^2))$.
\end{theorem}

\begin{proof}
	{\em Step 1.}
Let $p \in [1,\infty)$, $\tau \in [0,1/p]$, and $s\in(0,1/p]$. By Theorem \ref{spirale1} we know that
 ${\bf 1}_E \in B^{s,\tau}_{p, \infty} (\mathbb{R}^2)$ implies $\tau + s \le 1/p$.	
 Hence ${\bf 1}_E \not\in B^{s,\frac 1p - \frac s2}_{p, \infty} (\mathbb{R}^2)$ and furthermore
 ${\bf 1}_E \not\in B^{s,\frac 1p - \frac s2}_{p, \infty,{\rm unif}} (\mathbb{R}^2)$.
This  shows
\[
{\bf 1}_E \not \in M\left(B^s_{1,1}(\mathbb{R}^2)\right) \quad \mbox{and}\quad {\bf 1}_E \not \in M\left(B^s_{p,1}(\mathbb{R}^2), B^s_{p,\infty}(\mathbb{R}^2)\right)\, , \quad s\in(0,1/p]
\]
see Proposition \ref{FAchar} and Theorem \ref{Guli3}. If we assume that ${\bf 1}_E  \in M(B^s_{p,q}(\mathbb{R}^2))$ for some $s>0$, $p \in [1,\infty)$ and $q \in (0,\infty]$, then by real interpolation with $L^p(\mathbb{R}^2)$ we arrive at a contradiction.
\\
	{\em Step 2.} Let $s<0.$ If we assume that   ${\bf 1}_E  \in M(B^s_{p,q}(\mathbb{R}^n))$, then \eqref{dual1} yields
	a contradiction to Step 1. This finishes the proof of Theorem \ref{multi2}.
\end{proof}

%&&&&&&&&&&&&&&&&&&&&&&&&&&&&&&&&&&&&&&&&&&&&&&&&&&&&&&&&&&&&&&&&&&&&&&&
%&&&&&&&&&&&&&&&&&&&&&&&&&&&&&&&&&&&&&&&&&&&&&&&&&&&&&&&&&&&&&&&&&&&&&&&

\section{Characteristic functions with low regularity}
\label{Main4}

%&&&&&&&&&&&&&&&&&&&&&&&&&&&&&&&&&&&&&&&&&&&&&&&&&&&&&&&&&&&&&&&&&&&&&&&&&
%&&&&&&&&&&&&&&&&&&&&&&&&&&&&&&&&&&&&&&&&&&&&&&&&&&&&&&&&&&&&&&&&&&&&&&&&&

Here in this section we shall deal with smoothness $s\in (0,1/p)$.
Just for  technical reasons we have to consider Besov-type
and Triebel-Lizorkin-type spaces separately.

%&&&&&&&&&&&&&&&&&&&&&&&&&&&&&&&&&&&&&&&&&&&&&&&&&&&&&&&&&&&&&&&&&&&
%&&&&&&&&&&&&&&&&&&&&&&&&&&&&&&&&&&&&&&&&&&&&&&&&&&&&&&&&&&&&&&&&&&&

\subsection{Sufficient conditions in case $0 < s<1/p$. Triebel--Lizorkin-type spaces}\label{suff}

%&&&&&&&&&&&&&&&&&&&&&&&&&&&&&&&&&&&&&&&&&&&&&&&&&&&&&&&&&&&&&&&&&&&
%&&&&&&&&&&&&&&&&&&&&&&&&&&&&&&&&&&&&&&&&&&&&&&&&&&&&&&&&&&&&&&&&&&&

Taking inspiration from the known results in case $\tau =0$ we try to find sufficient conditions in terms of
the measure of the $\delta$-neighborhoods of the boundary of $E$.
We shall employ the characterization of $F^{s,\tau}_{p,q}(\mathbb{R}^n)$ by Haar wavelets in Proposition \ref{haarprop}.
Recall, if $p\in[1,\infty)$,  $q=1$, and $ s\in(0, 1/p)$, then
$\|\delta (f) \|_{f^{s,\tau}_{p,q}(\mathbb{R}^n)}\sim  \|f \|_{F^{s,\tau}_{p,q}(\mathbb{R}^n)}$ if
\[
0 < s < \min \left\{ \frac 1p, n \left(\frac 1p -\tau\right)\right\},
\]
where $\delta(f)$ is the same as  in \eqref{koeffh}.

\begin{theorem}\label{ok1}
Let $p\in[1,\infty)$, $\tau\in(0,1/p)$,   and   $0 < s < \min \{ \frac 1p, n (\frac 1p -\tau)\}$.
Let $E \subset \mathbb{R}^n$ be a bounded domain. Then
	\[
	\sup_{\{P \in \mathcal{Q},\, |P|\le 1\}} \, \frac{1}{|P|^\tau} \left[\int_{0}^{\sqrt{n} \ell (P)} \delta^{-sp}\,  |(\partial E)^{\delta} \cap P|\,
	\frac{d\delta}{\delta} \right]^{\frac1p}< \infty
	\]	
	implies ${\bf 1}_E \in F^{s,\tau}_{p,1}(\mathbb{R}^n)$.
\end{theorem}

\begin{proof}
	By Lemma \ref{support} it will be enough to deal with small dyadic cubes $P$, say $|P|\le 1$.
	Let $P = Q_{\ell,m}$ for some $\ell \in \mathbb{N}_0$ and some $m \in \mathbb{Z}^n$.
	We suppose that $ \min\{|P\cap E|, |P \cap E^\complement|\} >0$.
	Next we observe that
	\[
	|\langle {\bf 1}_E , h_{i,j,m}  \rangle | \le \left\{
	\begin{array}{lll}
		0 &\quad & \mbox{if}\quad |Q_{j,m} \cap E|= 2^{-jn},
		\\
		0 &\quad & \mbox{if}\quad |Q_{j,m} \cap E|=0 ,
		\\
		2^{\frac{jn}2}\, |E \cap Q_{j,m}| &\quad & \mbox{if}\quad 0 < |Q_{j,m} \cap E| <  2^{-jn} .
	\end{array}
	\right.
	\]
	Hence, we only need to consider those Haar coefficients of ${\bf 1}_E$ which are related to the boundary of $E$. This yields
	\begin{align*}
		{\rm I}_P:=& \int_{P}
		\left[\sum_{i=1}^{2^n-1} \sum_{j=\ell}^\infty\,  \sum_{\{m\in\mathbb{Z}^n:~ Q_{j,m} \subset P\}} 2^{j(s+\frac n2)} |\langle {\bf 1}_E , h_{i,j,m} \rangle \cx(2^jx-m)|\right]^{p}\, dx
		\\
		\le  &
		\int_{P}
		\left[\sum_{i=1}^{2^n-1} \sum_{j=\ell}^\infty \sum_{m\in K_j} 2^{\frac{jn}2}\,  2^{j(s+\frac n2)} |E \cap Q_{j,m}|  \, \cx(2^jx-m)
		\right]^{p}\, dx
		\\
		  \ls&   \int_{P}
		\left[ \sum_{j=\ell}^\infty \sum_{m\in K_j} 2^{js} \cx (2^jx-m) \right]^{p}\, dx,
	\end{align*}
where, for any $j \in \mathbb{N}_0$,
\begin{equation}\label{kjse}
 K_j := \left\{m \in \mathbb{Z}^n:~  Q_{j,m} \subset P, ~0 <
 |E \cap Q_{j,m}|< 2^{-jn} \right\}.
\end{equation}
	Now we start to work with parts of the $\delta $-neighborhood of $\partial E$.
For any $\alpha \in \mathbb{Z}$, let
	\begin{equation}\label{oalphaP}
		O^\alpha_P:=  \left\{ x \in P:~
		\sqrt{n}  2^{-\alpha -1}\le {\rm dist}\, (x, \partial E)< \sqrt{n}  2^{-\alpha}\right\}.
	\end{equation}
	It follows
	\begin{align*}
		{\rm I}_P & \le  \sum_{\alpha=\ell}^\infty \int_{O^\alpha_P}
		\left[\sum_{j=\ell}^\infty \sum_{m\in K_j} 2^{js} \mathcal{X} (2^jx-m)  \right]^{p}\, dx.
	\end{align*}
Notice that, for any fixed $x$, the sum $\sum_{k\in K_j}
	\cdots$ consists of at most one nonzero term.
	If $x\in O^\alpha_P $, by $s>0$, we   conclude that
	\[
	\sum_{j=\ell}^\infty \sum_{m\in K_j}  	 2^{js} \mathcal{X}(2^jx-m)
	\le    \sum_{j=\ell}^{\alpha -1} \sum_{m\in K_j}
	2^{js} \mathcal{X} (2^jx-m)
	\ls  2^{(\alpha-1) s}.
	\]
	Clearly, $O^\alpha_P \subset (\partial E)^{\sqrt{n}2^{-\alpha}} \cap P$.
	This yields
	\[
	2^{\alpha s p}\, |O^\alpha_P| \ls
	\int_{\sqrt{n}2^{-\alpha-1}}^{\sqrt{n}2^{-\alpha}} \delta^{-sp}\, \left|(\partial E)^{\delta} \cap P\right|\,
	\frac{d\delta}{\delta}.
	\]
	Summing up we have found
	\[
	{\rm I}_P \ls
	\int_{0}^{\sqrt{n} 2^{-\ell}} \delta^{-sp}\,  |(\partial E)^{\delta} \cap P|\,
	\frac{d\delta}{\delta}.
	\]
	This finishes the proof of Theorem \ref{ok1}.
\end{proof}

From the classical case $\tau =0$, it is well known that Theorem \ref{ok1} is not sharp in general, we refer to \cite{Si23} for a concrete example.
However, under extra conditions, for instance, if $E$ is a weakly exterior thick domain, it becomes necessary and sufficient.
Therefore, after a short supplement using Besov-type spaces,   we will work with some restrictions concerning the regularity of the boundary.

%&&&&&&&&&&&&&&&&&&&&&&&&&&&&&&&&&&&&&&&&&&&&&&&&&&&&&&&&&&&&&&&&&&&
%&&&&&&&&&&&&&&&&&&&&&&&&&&&&&&&&&&&&&&&&&&&&&&&&&&&&&&&&&&&&&&&&&&&

\subsection{Sufficient conditions in case $0 < s<1/p$. Besov-type spaces}\label{suff2}

%&&&&&&&&&&&&&&&&&&&&&&&&&&&&&&&&&&&&&&&&&&&&&&&&&&&&&&&&&&&&&&&&&&&
%&&&&&&&&&&&&&&&&&&&&&&&&&&&&&&&&&&&&&&&&&&&&&&&&&&&&&&&&&&&&&&&&&&&

We shall derive a counterpart of Theorem \ref{ok1} for Besov-type spaces.
But this time we will work with the characterization by differences. This will allow us to deal with the limiting case $s=n(\frac 1p -\tau)$ and this will be important for us  in connection with the multiplier problem.
We concentrate on $q=\infty$.

\begin{theorem}\label{ok2}
Let   $\tau\in(0,1)$ and   $s:= n - n \tau < 1 $.	
Let $E \subset \mathbb{R}^n$ be a bounded open set.
Assume that for some positive constant $C$,
	\begin{equation}\label{w017}
	\sup_{\{P \in \mathcal{Q}, ~ |P|\le 1\}} \, \frac{1}{|P|^\tau} \, \sup_{0 < \delta < C \ell (P)} \, \delta^{-s}\,  \left|(\partial E)^{\delta} \cap P\right|< \infty.
	\end{equation}
Then  ${\bf 1}_E \in B^{s,\tau}_{1,\infty}(\mathbb{R}^n)$.
\end{theorem}

\begin{proof}
For simplicity we assume that $C= 1$.	Otherwise the needed modifications are straightforward.
We shall work with Lemma \ref{Guli2}.
Hence we only need to show that
\[
B ({\bf 1}_E):= \sup_{0 < |h|<1} \,|h|^{-s} \sup_{r\in(0, \infty)} r^{s-n}
\sup_{x\in \mathbb{R}^n} \, \int_{B(x,r)} \, |{\bf 1}_E(y+h)-{\bf 1}_E(y)|\, dy < \infty\, .
\]	
Since $E$ is bounded it will be enough to consider  values $r< 1$.
Next we need to distinguish the cases $0<|h|< r$ and $|h|\ge r$.
\\
\noindent
{\em Step 1.} Let $0<|h|< r$. We  choose $\delta$ such that $|h|< \delta \le \min\{r,2|h|\}$.
Clearly, $ |{\bf 1}_E(y+h)-{\bf 1}_E(y)| >0$ requires  $y, y+h  \in {(\partial E)^\delta}$.
Hence we may restrict the integration to $B(x,r)\cap {(\partial E)^\delta}$.
Observe that  our assumption \eqref{w017} can be equivalently formulated by using balls instead of dyadic cubes.
Employing this modified version of  \eqref{w017} we obtain
 \begin{align}\label{pm1}
 B_1 ({\bf 1}_E)  : =& \sup_{r\in(0, 1)} r^{s-n} \sup_{0 <|h|<r} \,|h|^{-s}  \sup_{x\in \mathbb{R}^n} \,
 \int_{B(x,r)\cap {(\partial E)^\delta}} \, |{\bf 1}_E(y+h)-{\bf 1}_E(y)| \,dy
 \nonumber
 \\
  \ls &
  \sup_{r\in(0, 1)} r^{s-n} \sup_{0 <|h|<r} |h|^{-s} \,  \sup_{x\in \mathbb{R}^n} \,
 \delta^s\, |B(x,r)|^\tau
  \nonumber
 \\
  \ls &
\sup_{r\in(0, 1)} r^{s-n + n\tau} \sim 1 \,,
 \end{align}
 where we used the choice of $\delta$ in the second inequality.
\\
{\em Step 2.} Let $0 < r \le |h|<1$. By using a trivial estimate of the integral we conclude
 \begin{align}\label{pm2}
B_2({\bf 1}_E) := & \sup_{0 < |h|<1} \,|h|^{-s} \sup_{r\in(0, |h|)} r^{s-n} \sup_{x\in \mathbb{R}^n} \, \int_{B(x,r)} \, |{\bf 1}_E(y+h)-{\bf 1}_E(y)| dy
\nonumber
\\
 \ls & \sup_{0 < |h|<1} \,|h|^{-s} \sup_{r\in(0, |h|)} r^{s} \ls 1\, .
 \end{align}
The inequalities \eqref{pm1} and \eqref{pm2} together then yield the claim.
\end{proof}

\begin{corollary}\label{allp}
	Let   $\tau_0\in(0,1)$   and   $s_0:= n - n \tau_0 < 1 $.	
	Let $E \subset \mathbb{R}^n$ be a bounded open set such that \eqref{w017} is satisfied.
\begin{itemize}
\item[{\rm (i)}] Then  ${\bf 1}_E \in B^{s,\tau}_{1,q}(\mathbb{R}^n)$ for all $s< s_0$, all $\tau \le \tau_0$ and all $q \in (0,\infty]$.
\item[{\rm (ii)}] Let $p\in(1,\infty)$. Then  ${\bf 1}_E \in B^{s_0/p,\tau_0/p}_{p,\infty}(\mathbb{R}^n)$.
\end{itemize}
\end{corollary}

\begin{proof}
	Part (i) is an immediate consequence of Theorem \ref{ok2}, the elementary embeddings in Remark \ref{grund},
and Lemma \ref{support2}.
	Part (ii) follows from Theorem \ref{ok2} and Corollary \ref{t4.8}.
\end{proof}

\begin{remark}
Let 	 $\tau =0$.
In many examples the regularity of a domain with a fractal boundary
 has a best possible regularity given by a Besov space  $B^s_{p,\infty} (\mathbb{R}^n)$,
 for instance, domains with a boundary being a $d$-set. We refer to \cite{Si21}.
 But here, in this more general framework of type spaces, we do not have a corresponding lower bound.
\end{remark}

%&&&&&&&&&&&&&&&&&&&&&&&&&&&&&&&&&&&&&&&&&&&&&&&&&&&&&&&&&&&&&&&&&&&&&&&
%&&&&&&&&&&&&&&&&&&&&&&&&&&&&&&&&&&&&&&&&&&&&&&&&&&&&&&&&&&&&&&&&&&&&&&&

\subsection{Thick domains and weakly exterior thick domains}
\label{wets}

%&&&&&&&&&&&&&&&&&&&&&&&&&&&&&&&&&&&&&&&&&&&&&&&&&&&&&&&&&&&&&&&&&&&&&&&&&
%&&&&&&&&&&&&&&&&&&&&&&&&&&&&&&&&&&&&&&&&&&&&&&&&&&&&&&&&&&&&&&&&&&&&&&&&&

For the necessity, we need  to consider more special
thick domains and weakly exterior thick domains.
To recall their definitions, we need the well-known Whitney decomposition.
 In what follows,  we  say that two cubes \emph{touch} if their boundaries
 have a common point. For any cube $Q$,
$\ell(Q)$ denotes its the edge length.

%&&&&&&&&&&&&&&&&&&&&&&&&&&&&&&&&&&&&&&&&&&&&&&&&&&&&&&&&&&&&&&&&&&&&&&
%&&&&&&&&&&&&&&&&&&&&&&&&&&&&&&&&&&&&&&&&&&&&&&&&&&&&&&&&&&&&&&&&&&&&&&&&

We now recall the \textbf{Whitney decomposition}. Let  $E$ be a non-trivial open set in $\mathbb{R}^n$. We define $F:= \mathbb{R}^n \setminus \overline{E}$ and $G:= E^\complement:=\mathbb{R}^n \setminus E$.
Then there exists a collection $\mathcal W$ of dyadic cubes having the following properties:
\begin{itemize}
	\item[(i)] $E = \bigcup_{Q \in \mathcal{W}} Q$;
	\item[(ii)] The interiors of cubes in $\mathcal{W}$   are pairwise disjoint;
	\item[(iii)] For all cubes $Q \in \mathcal{W}$ we have
	${\rm diam}\, Q \le {\rm dist}\, (Q,G) \le 4\,  {\rm diam}\, Q$;
	\item[(iv)] If $Q_1, Q_2 \in \mathcal{W}$ touch then
	$\frac{{\rm diam}\, Q_2}4 \le {\rm diam}\, Q_1 \le 4\,  {\rm diam}\, Q_2$ follows;
	\item[(v)] For any $Q \in \mathcal W$ there exist at most $N= (12)^n$ cubes in
	$\mathcal W$ which touch $Q$.
\end{itemize}
For the construction of such a Whitney decomposition, we refer, for instance, to \cite[Chapter~VI]{St70}.

Now we turn to various types of domains.
First we recall the notion of a thick domain.
Therefore we follow Triebel \cite[Section 3.1]{t08}.

\begin{definition}\label{thick}
\begin{itemize}
\item[{\rm(i)}]  	A non-trivial domain $E \subset \mathbb{R}^n$ is said to be \emph{E-thick}
(exterior thick) if one finds for any interior cube $Q_i \subset E$ (with edges parallel to the coordinate axes)
	with $\ell(Q_i)  \sim 2^{-j}$ and
$ {\rm dist}\,(Q_i , \partial E)  \sim 2^{-j}$, $ j \ge j_0 \in \mathbb{N}$,
	a complementing exterior cube $Q_e \subset \mathbb{R}^n \setminus E$  with
	\[
	\ell(Q_e)  \sim 2^{-j} \ \ \mbox{and}\ \
	{\rm dist}\,(Q_e , \partial E) \sim {\rm dist}\,(Q_i , Q_e)  \sim 2^{-j} ,
	\ \forall\,j \ge  j_0 \in \mathbb{N}.
	\]
	\item[{\rm (ii)}]
	A non-trivial  domain $E \subset \mathbb{R}^n$ is said to be \emph{I-thick} (interior thick) if one finds for any exterior cube $Q^e \subset E^\complement$ (with edges parallel to the coordinate axes) with
	$\ell(Q^e)  \sim 2^{-j}$ and ${\rm dist}\,(Q^e, \partial E)  \sim 2^{-j}$, $j \ge j_0 \in \mathbb{N}$,
	a reflected interior cube $Q_e \subset E$  with
	\[
	\ell\left(Q^i\right)  \sim 2^{-j}  \ \ \mbox{and}\ \
	{\rm dist}\,\left(Q^i, \partial E\right)  \sim {\rm dist}\,\left(Q^i, Q^e\right) \sim  2^{-j},
	\ \forall\, j \ge  j_0 \in \mathbb{N}.
	\]
	\item[{\rm (iii)}] A non-trivial  domain $E \subset \mathbb{R}^n$ is said to be \emph{thick}
	if it is both $E$-thick and $I$-thick.
\end{itemize}
\end{definition}

Most important for us is the following notion which represents a moderate modification of an exterior thick domain.
By ${\mathcal W}$ we denote the dyadic cubes of the Whitney decomposition of $E$.

\begin{definition}\label{wthick}
\begin{itemize}
\item[{\rm(i)}]
	A non-trivial domain $E \subset \mathbb{R}^n$ is said to be  \emph{weakly exterior thick}
	if there exist positive constants $A$  and $B$ such that  one finds for any interior cube $Q \in  {\mathcal W}$
	a complementing exterior set  $\Omega_Q \subset E^\complement$  and a point $x^0 \in Q$
	with
	\begin{equation*}
		\Omega_Q \subset B\left(x^0, A\, \ell(Q)\right)  \qquad \mbox{and}\qquad
		|\Omega_Q|  \ge B \, |Q|.
	\end{equation*}
	\item[{\rm (ii)}] A non-trivial domain $E \subset \mathbb{R}^n$ is said to be  \emph{weakly interior thick}
	if $F$  is  weakly exterior thick.
		\item[{\rm (iii)}] A non-trivial domain $E \subset \mathbb{R}^n$ is said to be \emph{weakly thick}  if it is  weakly interior and  weakly exterior thick.
\end{itemize}
\end{definition}

\smallskip

\begin{minipage}{0.5\textwidth}
	\includegraphics[height=5.2cm]{{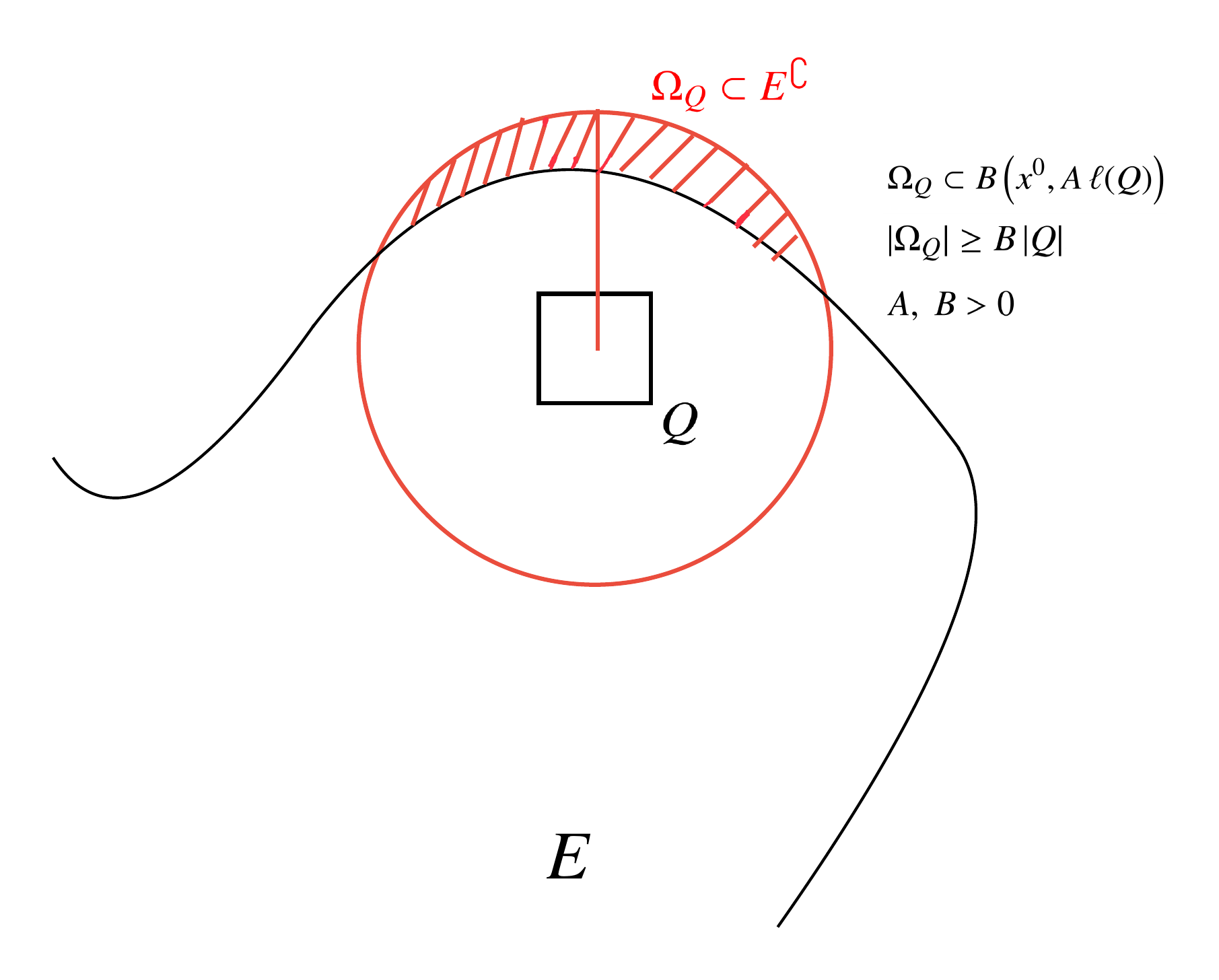}}
\end{minipage}\hfill
\begin{minipage}{0.38\textwidth}
In the left picture we plotted a typical situation. The dashed
region represents a possible choice of the set $\Omega_Q$.
Clearly, exterior thick domains are weakly exterior thick.
Below we shall use the following conventions. Let $(A,B)$ be
any admissible pair for the weakly exterior thick domain $E$.
 We fix $B_E:=B$, and define
\end{minipage}

\smallskip

\begin{equation*}
		A_Q := \inf \left\{A\in(0,\infty):~ \Omega_Q \subset B(x, A\, \ell(Q)) ~~ \forall x \in Q\right\}
\end{equation*}
as well as
\begin{equation*}
	A_E := \inf \left\{A_Q:~ Q \in \mathcal{W}\right\}.
\end{equation*}
Observe that in any case  $A_E$ is positive.
Next we would like to recall a first consequence for the boundaries of $E$ and $F$, see \cite{Si23}.

\begin{lemma}\label{dE=dF}
	Let $E$ be a non-trivial weakly exterior thick  domain in $\mathbb{R}^n$.
	Then $\partial E = \partial F$ follows.
\end{lemma}

Here is a simple example which shows what is going on.
Let
\begin{equation*}
	E_1:= B(\textbf{0},1) \quad  \mbox{and} \quad E_2:= B(\textbf{0},1)\setminus \left\{(x,0)\in\mathbb{R}^2:~ 0<x<1\right \}	.
\end{equation*}
Then $F_1 = F_2 = \{(x,y)\in\mathbb{R}^2:~ x^2+ y ^2 > 1 \}$. Whereas $E_1$ is obviously
weakly exterior thick, $E_2$ is not.

It is essentially obvious that we can replace $\mathcal{W}$ in Definition
\ref{wthick} by all the cubes
satisfying property (iii) of the Whitney cubes.

\begin{lemma}\label{allcubes}
Let $E$ be weakly exterior thick.
	Then there exist positive
	constants $A'_E$ and $B'_E$ such that, for any interior cube $Q$ satisfying ${\rm diam}\,Q \le
{\rm dist}\, (Q,E^\complement)\le 4 \, {\rm diam}\,Q $,
	there exist  a complementing exterior set  $\Omega_Q \subset E^\complement$ and a point $x^0 \in Q$
	with
	\begin{equation*}
		\Omega_Q \subset B\left(x^0, A'_E\, \ell(Q)\right)  \quad \mbox{and}\quad
		|\Omega_Q|  \ge B'_E \, |Q|.
	\end{equation*}
	\end{lemma}

Indeed, it is easy to check that  $x^0$ in Lemma \ref{allcubes}
can   be any point in $Q$.

%&&&&&&&&&&&&&&&&&&&&&&&&&&&&&&&&&&&&&&&&&&&&&&&&&&&&&&&&&&&&&&&&&&&
%&&&&&&&&&&&&&&&&&&&&&&&&&&&&&&&&&&&&&&&&&&&&&&&&&&&&&&&&&&&&&&&&&&&

\subsection{Necessary conditions in case $0 < s<1/p$}\label{with}

%&&&&&&&&&&&&&&&&&&&&&&&&&&&&&&&&&&&&&&&&&&&&&&&&&&&&&&&&&&&&&&&&&&&
%&&&&&&&&&&&&&&&&&&&&&&&&&&&&&&&&&&&&&&&&&&&&&&&&&&&&&&&&&&&&&&&&&&&

Our approach to derive necessary conditions is in some sense standard.
Also Jaffard and Meyer \cite{JM}, Faraco and Rogers \cite{FR}, and Sickel \cite{s99b,s99,Si21}
have applied the following strategy.
Let $Q \subset E$ be a cube satisfying $[{\rm dist}\, (Q,E^\complement)]^n\sim |Q|$. We then suppose
the existence of a measurable set $\Omega_Q \subset E^\complement $
satisfying $$[{\rm dist}\, (Q, E^\complement)]^n \sim [{\rm diam}\, \Omega_Q]^n \sim |Q|\sim |\Omega_Q|.$$

Let $\mathcal{W}_1 $ be the collection of all dyadic cubes forming the Whitney decomposition of $E$.
By  $\mathcal{W}_P $ we denote   the collection of all dyadic cubes forming the Whitney decomposition of $E\cap P$.

\begin{lemma}\label{nec1}
Let $p\in[1,\infty)$, $\tau \in(0, \frac 1p)$, and $s \in(0,\min\{\frac 1p, n(\frac 1p - \tau)\})$.	Let $E$ be a  weakly exterior thick domain.
	Then there exists a positive constant $C$ such that
	\[
	\|{\bf 1}_E \|_{{F}^{s,\tau}_{p,\infty}(\mathbb{R}^n)}
	\ge C  \sup_{\{P\in\mathcal{Q}, |P|\le 1\}}\, \frac{1}{|P|^{\tau}}
	\,   \left\{	\sum_{Q \in \mathcal{W}_P}
	[l(Q)]^{n-sp}\right\}^{\frac1p} .
	\]
\end{lemma}

\begin{proof}
We restrict ourselves to dyadic cubes $P$ with  $|P|\le 1$ satisfying  that $|P \cap E|>0$ and $|P \cap E^\complement|>0$.
Let $Q \in \mathcal{W}_P$. Hence, $Q \subset E$, but not necessarily $Q \in \mathcal{W}$.
	By $\Omega_Q$ we denote the associated subset of $E^\complement$ according to Lemma \ref{allcubes}.
	Then, by Lemma \ref{allcubes}, one has  $\Omega_Q \subset B(x,t_Q)$ for all $x\in Q$, where $t_Q := A'_E \ell(Q)$.
Applying the characterization by differences in Proposition \ref{H1} with a dyadic cube $P$ satisfying
$P \cap E \neq \emptyset$, we find that
	\begin{align*}
	\|  {\bf 1}_E \|_{{F}^{s,\tau}_{p,\infty}(\mathbb{R}^n)}^p
& \ge
	   \left\| \sup_{0<t<\infty}  \frac1{t^{s}} \left[\frac1{t^{n}} \int_{B(\textbf{0},t)} |\Delta^{1}_{h}{\bf 1}_E(x) |^p \,dh \right]^{\frac1p} 	 \right\|_{\cm^{\frac{p}{1-\tau p}}_p (\mathbb{R}^n)}^p
	\\
	& \ge   \frac{1}{|P|^{\tau p}} \int_P
		 \sup_{0<t<\infty}  \frac1{t^{sp+n}}\, \int_{B(\textbf{0},t)} |{\bf 1}_E(x+h)- {\bf 1}_E (x)|^p \,dh \, dx 	
	\\
	& \ge   \frac{1}{|P|^{\tau p}} \sum_{{Q \in \mathcal{W}_P}}
	\int_{Q}
\sup_{t_Q<t<\infty}  \frac1{t^{sp+n}}\, \int_{B(\textbf{0},t)} |{\bf 1}_E(x+h)- {\bf 1}_E (x)|^p \,dh\, dx 	
\\	
	& \ge  \frac{1}{|P|^{\tau p}} \sum_{{Q \in \mathcal{W}_P}}
	\int_{Q}
\sup_{t_Q<t<\infty}   \frac1{t^{sp+n}}\, \int_{\Omega_Q} |{\bf 1}_E(y)- {\bf 1}_E (x)|^p \,dy \, dx.
\end{align*}
From this and  Lemma \ref{allcubes}, it follows that
\begin{align*}
		\| {\bf 1}_E  \|^p_{{F}^{s,\tau}_{p,\infty}(\mathbb{R}^n)}
& \gs  \frac{1}{|P|^{\tau p}} \sum_{{Q \in \mathcal{W}_P}} 	t^{-sp-n}_Q \, |\Omega_Q|\,  |Q|	
\\
& \gs   \frac{1}{|P|^{\tau p}} \sum_{{Q \in \mathcal{W}_P}} 	
[A'_E  \, \ell(Q)]^{-sp-n} \, B'_E\,  [\ell (Q)]^{n} \, |Q|	
\\
& \gs  \frac{1}{|P|^{\tau p}} \sum_{{Q \in \mathcal{W}_P}} 	[l(Q)]^{-sp+n} .
\end{align*}
This finishes the proof of Lemma \ref{nec1}.
\end{proof}

Let $\mathcal{W}^*_{P}$ denote the Whitney decomposition of $P \cap E^\complement$.
If we suppose that $E$ is a bounded and weakly thick domain, then we can interchange the role of $E$ and $E^\complement$.
Since we only deal with dyadic cubes $P$ satisfying $|P|\le 1$ with a nontrivial intersection with $\partial E$,
this yields the following lemma.

\begin{lemma}\label{nec11}
	Let $E$ be a bounded  weakly thick domain. Let $p,s$, and $\tau$ be the same as in  Lemma \ref{nec1}.
Then there exists a positive constant $C$ such that
		\[
	\| {\bf 1}_E \|_{{F}^{s,\tau}_{p,\infty}(\mathbb{R}^n)}
	\ge C \sup_{\{P\in\mathcal{Q},\, |P|\le 1\}}\, \frac{1}{|P|^{\tau}}
	\,   \left\{	\sum_{Q \in \mathcal{W}^*_P}
	[l(Q)]^{n-sp}\right\}^{\frac1p} .
	\]
\end{lemma}

For later use we shall deal with a reformulation of this estimate.
For any integer $j \ge -\log_2 \ell (P)$, let
\[
\mathcal{W}_{P,j}:= \left\{Q \in  \mathcal{W}_{P}:\ |Q|= 2^{-jn}\right\}.
\]
It follows
\begin{equation}\label{w016}
\sum_{Q \in \mathcal{W}_P} [l(Q)]^{n-sp}
= \sum_{j=-\log_2 \ell (P)}^\infty \sum_{Q \in \mathcal{W}_{P,j}}
2^{-j(n-sp)} =  \sum_{j=-\log_2 \ell (P)}^\infty |\mathcal{W}_{P,j}|\,
2^{-j(n-sp)}.
\end{equation}
Let $Q=Q_{j,m} \in \mathcal{W}_{P,j}$.
Then $|Q \cap {O}^\alpha_P |>0$ implies $\alpha \le j$, where $O^\alpha_P$ is the same as in \eqref{oalphaP}.
Consequently, we obtain
\[
O^\alpha_P \cap E  \subset \bigcup_{j=\alpha}^\infty \left(\bigcup_{Q \in \mathcal{W}_{P,j}} Q\right)
\]
and therefore
\[
|O^\alpha_P \cap E| \le
\sum_{j= \alpha}^\infty  |\mathcal{W}_{P,j}| \, 2^{-jn}.
\]
From these and $s>0$, it follows that
\begin{align*}
\sum_{\alpha=\ell}^\infty 2^{\alpha sp} |O^\alpha_P \cap E|
& \le   \sum_{\alpha=\ell}^\infty
2^{\alpha sp}  \sum_{j= \alpha}^\infty |\mathcal{W}_{P,j}| \, 2^{-jn}
\\
&=
\sum_{j=\ell}^\infty\, |\mathcal{W}_{P,j}| \, 2^{-jn}
\sum_{\alpha = \ell}^j 2^{\alpha sp}
 \ls
\sum_{j=\ell}^\infty \, |\mathcal{W}_{P,j}| \, 2^{-jn} 2^{jsp}
\end{align*}
with the implicit positive constant independent of $E$ and $P$.
Now we combine this estimate with \eqref{w016} and find
\[
\sum_{\alpha=\ell}^\infty 2^{\alpha sp} |O^\alpha_P \cap E|\ls \sum_{Q \in \mathcal{W}_P} [l(Q)]^{n-sp},
\]
which further implies
\begin{align*}
\sup_{\{P\in\mathcal{Q}, |P|\le 1\}}\, \frac{1}{|P|^\tau} \left( \sum_{\alpha=\ell}^\infty 2^{\alpha sp}
|O^\alpha_P \cap E|\right)^{\frac1p}
&\ls \sup_{\{P\in\mathcal{Q}, |P|\le 1\}}\, \frac{1}{|P|^\tau}\left(\sum_{Q \in \mathcal{W}_P} [l(Q)]^{n-sp}
\right)^{\frac1p} \\
&\ls 	\|{\bf 1}_E \|_{{F}^{s,\tau}_{p,\infty}(\mathbb{R}^n)},
\end{align*}
where we used Lemma \ref{nec1}.
For this reason we have proved the following lemma.

\begin{lemma}\label{nec2}
Under the same assumptions as in Lemma \ref{nec1}, one has
\[
\sup_{\{P\in\mathcal{Q}, |P|\le 1\}}\, \frac{1}{|P|^\tau} \left( \sum_{\alpha=\ell}^\infty 2^{\alpha sp}
|O^\alpha_P \cap E|\right)^{\frac1p} \ls \|{\bf 1}_E \|_{{F}^{s,\tau}_{p,\infty}(\mathbb{R}^n)}
\]
with the implicit positive constant independent of $E$.
\end{lemma}

Again there is a supplement for bounded weakly thick domains $E$.

\begin{lemma}\label{nec21}
Under the same assumptions as in Lemma \ref{nec11}, one has
\[
\sup_{\{P\in\mathcal{Q}, |P|\le 1\}}\, \frac{1}{|P|^\tau} \left( \sum_{\alpha=\ell}^\infty 2^{\alpha sp}
\left|O^\alpha_P \cap E^\complement\right|\right)^{\frac1p} \ls \|  {\bf 1}_E  \|_{{F}^{s,\tau}_{p,\infty}(\mathbb{R}^n)}
\]
with the implicit positive constant independent of $E$.
\end{lemma}

Now we would like to switch from $|O^\alpha_P \cap E|$ to
 $|(\partial E)^{\sqrt{n} 2^{-j} }\cap E|$.
Clearly
\[
(\partial E)^{\sqrt{n} 2^{-j}} \cap P = \bigcup_{\alpha=j}^\infty O^\alpha_P \,.
\]
Because of $s>0$ we find
\begin{align*}
\sum_{j=\ell}^\infty 2^{jsp}\, \left|(\partial E)^{\sqrt{n} 2^{-j}} \cap E \cap P\right|
& \le \sum_{j=\ell}^\infty 2^{jsp} \sum_{\alpha=j}^\infty |O^\alpha_P \cap E|
\\
& =  \sum_{\alpha = \ell}^\infty
\left|O^\alpha_P \cap E\right|  \sum_{j=0}^\alpha 2^{jsp}
\ls  \sum_{\alpha = \ell}^\infty  |O^\alpha_P \cap E| \, 2^{\alpha sp} .
\end{align*}
For aesthetic reasons we turn sums into  integrals. Observe
\begin{eqnarray*}
\sum_{j=\ell}^\infty 2^{jsp}\,\left|(\partial E)^{\sqrt{n} 2^{-j}} \cap E \cap P\right|
\gs
\int_{0}^{\sqrt{n} 2^{-\ell}} \delta^{-sp}\, \left|(\partial E)^{\delta} \cap E \cap P\right| \, \frac{d\delta}{\delta}.
\end{eqnarray*}
Hence, we have proved the following assertion.

\begin{theorem}\label{nec3}
Let $p\in[1,\infty)$, $\tau \in(0, \frac 1p)$, and $ s \in(0,\min\{\frac 1p, n(\frac 1p - \tau)\})$.	Let $E$ be a bounded  weakly  exterior thick domain.
	Then ${\bf 1}_E \in {F}^{s,\tau}_{p,\infty}(\mathbb{R}^n)$ implies
	\[
\sup_{\{P\in\mathcal{Q},\, |P|\le 1\}}\, \frac{1}{|P|^\tau}\, 	\left[ \int_{0}^{\sqrt{n} \ell (P)} \delta^{-sp}\,
	\left|(\partial E)^{\delta} \cap E \cap P\right| \, \frac{d\delta}{\delta}\right]^{\frac1p} < \infty .
	\]
\end{theorem}

Similarly one can prove  the following result.

\begin{theorem}\label{nec31}
Let $p\in[1,\infty)$, $\tau \in(0, \frac 1p)$, and $s\in(0, \min\{\frac 1p, n(\frac 1p - \tau)\})$.	Let $E$ be a bounded weakly  thick domain.
	Then ${\bf 1}_E \in {F}^{s,\tau}_{p,\infty}(\mathbb{R}^n)$ implies
	\[
\sup_{\{P\in\mathcal{Q},\, |P|\le 1\}}\, \frac{1}{|P|^\tau}\, 	\left[ \int_{0}^{\sqrt{n} \ell (P)} \delta^{-sp}\,
	\left|(\partial E)^{\delta}  \cap P\right| \, \frac{d\delta}{\delta}\right]^{\frac1p} < \infty .
	\]
\end{theorem}

Applying Theorems \ref{ok1} and \ref{nec31} and  the elementary embeddings
\[ {F}^{s,\tau}_{p,1}(\mathbb{R}^n) \hookrightarrow  {F}^{s,\tau}_{p,q}(\mathbb{R}^n) \hookrightarrow  {F}^{s,\tau}_{p,\infty}(\mathbb{R}^n),
\]
we obtain  the following quite satisfactory conclusion.

\begin{theorem}\label{top}
	Let  $p\in[1,\infty)$, $\tau \in(0, \frac 1p)$,  $q\in[1,\infty]$, and $s\in(0, \min\{\frac 1p, n(\frac 1p - \tau)\})$. Let $E$ be a bounded weakly  thick domain.
	Then ${\bf 1}_E \in {F}^{s,\tau}_{p,q}(\mathbb{R}^n)$ holds if and only if
	\begin{equation*}%\label{fs-con}
	\sup_{\{P\in\mathcal{Q}, |P|\le 1\}}\, \frac{1}{|P|^\tau}\, 	\left[ \int_{0}^{\sqrt{n} \, \ell (P)} \delta^{-sp}\,
	\left|(\partial E)^{\delta} \cap P\right| \, \frac{d\delta}{\delta}\right]^{\frac1p} < \infty .
	\end{equation*}
\end{theorem}

\begin{remark}
This theorem represents a surprising extension of the known result in case $\tau =0$, see Sickel \cite{Si23}. We do not know whether the given restrictions on the parameters are necessary.
We also do not know  whether the property of being a bounded  weakly thick domain is indispensable. However, we know that for some more general domains Theorem \ref{top} is no longer true. A corresponding example has been treated in detail in \cite{Si23}.
\end{remark}

%&&&&&&&&&&&&&&&&&&&&&&&&&&&&&&&&&&&&&&&&&&&&&&&&&&&&&&&&&&&&&&&&&&&
%&&&&&&&&&&&&&&&&&&&&&&&&&&&&&&&&&&&&&&&&&&&&&&&&&&&&&&&&&&&&&&&&&&&

\subsection{The snowflake domain}
\label{Main6}

%&&&&&&&&&&&&&&&&&&&&&&&&&&&&&&&&&&&&&&&&&&&&&&&&&&&&&&&&&&&&&&&&&&&
%&&&&&&&&&&&&&&&&&&&&&&&&&&&&&&&&&&&&&&&&&&&&&&&&&&&&&&&&&&&&&&&&&&&

For later use we recall the construction of the boundary of the snowflake domain,  the von Koch curve.

\smallskip

\begin{center}
\includegraphics[height=3.6cm]{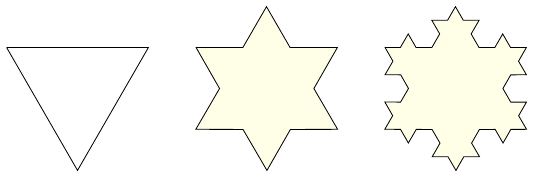}
\end{center}

\smallskip

The standard construction of the von Koch curve is as follows.
We start with an equilateral triangle. Then we subdivide each side into three equal parts and remove the middle one.
This middle part is replaced by an equilateral triangle again.

\smallskip

\begin{minipage}{0.5\textwidth}
\includegraphics[height=6cm]{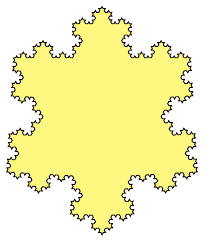}
 \end{minipage}\hfill\begin{minipage}{0.4\textwidth}
The edge length is now $1/3$ of the original one.
This procedure is iterated.
After a few further iterations one obtains the figure on the left which might be seen as a reasonable
approximation of the von Koch curve. As mentioned above the domain $\Omega$ with the {\em von Koch curve} as its boundary
is called the {\em snowflake domain}.
The von Koch curve  is a selfsimilar set satisfying
$\dim_H \partial \Omega = \dim_M \partial \Omega = \frac{\log 4}{\log 3}$,
see,  e.g.,   Falconer \cite[Example~9.5]{Fa90}.
\end{minipage}

\smallskip

To be able to apply  Theorems \ref{ok1}, \ref{ok2},  \ref{nec3}, and \ref{top},
we need to know more about the boundary.

\begin{lemma} \label{hrandsnow}
The von Koch curve is a $d$-set
 with $d:= \frac{\log 4}{\log 3}$.
 Moreover, there exist positive constants $A$ and $B$  such that
\begin{equation}\label{w060}
A\, r^{2-d}\le
\left|(\partial \Omega)^r\right|\le B\, r^{2-d}, \quad\forall\ r\in(0, 1).
\end{equation}
\end{lemma}

\begin{proof}
		For the fact that the von Koch curve is a $d$-set, we refer to Falconer \cite[Theorem 9.3]{Fa90}.
On the other hand, observe that $d$-sets are special $h$-sets. In this generality Bricchi \cite{Bri02,Bri04} has proved \eqref{w060}.
\end{proof}

\begin{lemma}\label{prep}
	The snowflake domain is a thick domain.
\end{lemma}

A proof of Lemma \ref{prep} can be found in Caetano et al. \cite{CHM}, but see also \cite[Proposition~3.8]{t08}.

We need to go one step further.
Instead of $|(\partial \Omega)^\delta|$ we have to investigate $|(\partial \Omega)^\delta \cap P|$,
where $P$ is an arbitrary dyadic cube satisfying $|P|\le 1$.

\begin{lemma}\label{schneerand}
Let $d:= \frac{\log 4}{\log 3}$.
Then there exists a	positive constant $A$ such that
 \begin{equation}\label{w061}
\sup_{\{P \in \mathcal{Q},~|P|\le 1\}}  [\ell (P)]^{-d}\,  \sup_{0 <\delta \le A\, \ell (P)}\, \delta^{d-2}\,  \left|(\partial \Omega)^\delta \cap P\right|
< \infty \, .
\end{equation}
\end{lemma}

\begin{proof}
{\em Step 1.}
To derive \eqref{w061} we take a closer look at the construction of $\partial \Omega$.
Without loss of generality we may assume that the original triangle
on the left-hand side in the above figure has the edge length $1$.
Below we will concentrate on one edge of the triangle identified with
$\Gamma_0 := \{(x,0):\ 0 \le x < 1\}.$
Then we need the following four contractions $S_1, S_2, S_3, S_4$, defined on $\mathbb{R}^2$,  and given by
\begin{align*}
S_1 (\Gamma_0) & := \{(x,0):\ 0 \le x < 1/3\}, \\
S_2 (\Gamma_0) & :=\left\{(x,y):\ 1/3 \le x < 1/2, ~ y = \sqrt{3}\, (x-1/3)\right\}, \\
S_3 (\Gamma_0) & :=\left\{(x,y):\ 1/2 \le x < 2/3, ~ y= -\sqrt{3}x + 2/\sqrt{3}\right\}, \quad\mbox{and}\\
S_4 (\Gamma_0) & := \{(x,0):\ 2/3 \le x < 1\}.
\end{align*}	
These contractions map $\Gamma_0$ onto the following polygon.

\smallskip

\begin{center}
\hspace*{2cm} \includegraphics[height=3cm]{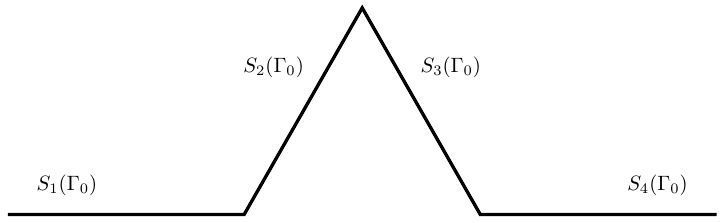}
\end{center}

\smallskip

We define
\[
\Gamma_1 :=\bigcup_{i=1}^4 S_i (\Gamma_0)  ,
\]
which consists of $4$ straight pieces of length $1/3$.
Now we iterate this procedure. After $j$ iterations we arrive at $\Gamma_j$ given by
\[
\Gamma_j := \bigcup_{\gfz{i_\ell \in \{ 1, \ldots, 4\}}{\ell \in\{ 1, \ldots , j\}}} ( S_{i_1} \circ \ldots \circ S_{i_j}) (\Gamma_0) .
\]
Clearly, $\Gamma_j$ is a polygon and  consists of $4^j$ straight pieces of length $3^{-j}$. This yields that the length $\ell (\Gamma_j)$ of  $\Gamma_j$ is given by $(4/3)^j$.
Now we take two integers $L$ and $M$ such that $L \in \mathbb{N}_0$, $M \in \mathbb{N}$, and $L < M$.
By $\Lambda_1^L, \ldots, \Lambda^L_{4^L}$ we denote the straight pieces of length $3^{-L}$ which form the set $\Gamma_L$. Then we can construct the set $\Gamma_M$ in two ways.
The first one is described above. The second one is as follows.
This time our point of departure is not $\Gamma_0$. We take any of the straight pieces $\Lambda^L_m$ and start the iteration on it. Let
\[
\mathcal{S}^{L,M}_m :=\, \bigcup_{\gfz{i_\ell \in \{ 1, \ldots, 4\}}{\ell \in\{1, \ldots , M-L\}}} ( S_{i_1} \circ \ldots \circ S_{i_{M-L}}) (\Lambda^L_m)\, .
\]
It follows
\[
\Gamma_M = \bigcup_{m=1}^{4^L} \, \mathcal{S}^{L,M}_m\, .
\]
By construction, for any $m$ the set $\mathcal{S}^{L,M}_m$ is the image of $\mathcal{S}^{L,M}_1$
under a rotation and a translation.
This means we can decompose the polygon $\Gamma_M$ into $4^L$ disjoint parts of length $4^{-L} (4/3)^M$.
The same argument applies for the $\delta$-neighborhoods.
It follows $|(\mathcal{S}^{L,M}_m )^{\delta}|= |(\mathcal{S}^{L,M}_1 )^{\delta}|$.
Clearly, one has $\lim_{M \to \infty} \, \Gamma_M = \partial \Omega^*$,
where $\partial \Omega^*$ denotes ``one third'' of the von Koch curve related to the fact that we started our iteration only on  $\Gamma_0$, not on the whole triangle.
Letting $M\to\infty$  we obtain a pairwise disjoint decomposition of $\partial \Omega^*$
into $4^L$ pieces which differ only by a rotation and a translation,

Let us denote by $S(\Lambda^L_m)$ the limit of the sets $\mathcal{S}^{L,M}_m$ for $M \to \infty$, namely  the part of $\partial \Omega^*$
which is the result of the iteration with the starting set $\Lambda^L_m$.
Hence
\[
\partial \Omega^* = \bigcup_{m=1}^{4^L} S(\Lambda^L_m).
\]
Similarly we may argue for $(\partial \Omega^*)^\delta$. Then there exists a cover of $(\partial \Omega^*)^\delta$
given by $4^L$  identical pieces (up to a rotation and a translation).
This cover is not pairwise disjoint because the neighborhoods are overlapping.
We then have
\[
(\partial \Omega^*)^\delta  \subset \bigcup_{m=1}^{4^L}\left (S(\Lambda^L_m)\right)^\delta
\]
and therefore
\begin{equation}\label{w064}
\left|(\partial \Omega^*)^\delta \right| \le {4^L}\, \left|\left (S(\Lambda^L_1)\right)^\delta\right|
\end{equation}
for all $\delta \in(0,1)$.
To control the overlap   we proceed as follows. Instead of considering all straight pieces of $\Gamma_L$ we consider only any third element.
Define
\[
 \Gamma_L^* := \bigcup_{m \in \nn:  ~ 3m-2 \le 4^L}\, \Lambda^L_{3m-2}.
\]
Here we suppose that the enumeration is done in a natural way, that is,
$\Lambda_m$ touches $\Lambda_{m+1}$ in exactly one point.
We observe that
\[
 {\rm dist}\, ((x,y),\partial \Omega^*)\le \sum_{\ell =1}^\infty \frac{\sqrt{3}}{6}\, 3^{-\ell} = \frac{\sqrt{3}}{12}, \quad\forall\ (x,y) \in \Gamma_0.
\]
From the same argument it follows
\begin{equation}\label{w062}
 {\rm dist}\, ((x,y),\partial \Omega^*)\le \sum_{\ell =L+1}^\infty \frac{\sqrt{3}}{6}3^{-\ell} = \frac{\sqrt{3}}{12}\, 3^{-L}, \quad
 \forall\ (x,y) \in \Gamma_L.
\end{equation}
Our next observation is also elementary.
Just looking at the pictures of $\Gamma_1$ and $\Gamma_2$ (the upper parts of the second and of the third iteration) we conclude that
\[
\max_{\ell}\,  \inf \Big\{|(x_0,y_0)-(x_1,y_1)|:\ (x_0,y_0) \in \Lambda^L_{3\ell-2},\ (x_1,y_1) \in \Lambda^L_{3\ell + 1}\Big\} =   3^{-L}\, , \quad L=1,2\, .
\]
By the construction of $\Gamma_L$, this situation is repeated on a lower scale in all $\Gamma_L$, $L \ge 3$.
Length of the straight pieces is $3 ^{-L}$ and angles between those pieces are either $60$ or $120$ degrees.
It follows
\begin{equation}\label{w062b}
\max_{\ell}\, \inf \Big\{|(x_0,y_0)-(x_1,y_1)|:\ (x_0,y_0) \in \Lambda^L_{3\ell-2},\ (x_1,y_1) \in \Lambda^L_{3\ell + 1}\Big\} = 3^{-L}, \qquad L \ge 1\, .
\end{equation}
Such estimates remain true when  $\Lambda^L_{3\ell +1}$ is replaced by $\Lambda^L_{3\ell + m}$ with
$m \ge 2$. Now we look on the situation on $\partial \Omega^*$. Based on our previous arguments, see \eqref{w062} and \eqref{w062b}, we obtain
\begin{align}\label{w067}
&\inf \left\{|(x_0,y_0)-(x_1,y_1)|:\ (x_0,y_0) \in S(\Lambda^L_{3\ell-2}),\ (x_1,y_1)
\in S(\Lambda^L_{3\ell + 1})\right\}\nonumber
\\
&\quad \ge    3^{-L} - 2\, \frac{\sqrt{3}}{12}\, 3^{-L} = \Big( 1- \frac{\sqrt{3}}{6}\Big) \, 3^{-L}.
\end{align}
Now we turn to the consequences for the $\delta$-neighborhoods.
The previous formula, together with \eqref{w062}, further implies
\begin{align*}
&\inf \left\{|(x_0,y_0)-(x_1,y_1)|:\ (x_0,y_0)
\in\left( S(\Lambda^L_{3\ell-2})\right)^\delta,\ (x_1,y_1) \in \left(S(\Lambda^L_{3\ell + 1})\right)^\delta\right\}
\\
 &\quad \ge  3^{-L} \left(1 - \frac{\sqrt{3}}{6}\right) - 2 \delta
 \ge  3^{-L} \, \frac{1}{6}
\end{align*}
if $0<\delta \le 3^{-L} (\frac{5-\sqrt{3}}{12}) $.
This shows that
\[
\left| \bigcup_{m \in \nn:~ 3m-1\le 4^L}\left (S(\Lambda^L_{3m-1})\right)^\delta\right| \ge \frac 14 \, 4^{L} \,
 \left|\left(S(\Lambda^L_{1})\right)^\delta\right|  \quad \mbox{if}\ \ 0<\delta \le 3^{-L} \left(\frac{5-\sqrt{3}}{24}\right), \quad L\ge 2\, ,
\]
which, together with \eqref{w064}, further implies
\begin{equation*}
	\left|(\partial \Omega^*)^\delta \right| \sim {4^L}\, \left|\left (S(\Lambda^L_1)\right)^\delta\right|  \quad \mbox{if}\ \
	0 < \delta \le 3^{-L} \left(\frac{5-\sqrt{3}}{12}\right)\, .
\end{equation*}
Using Lemma \ref{hrandsnow} we conclude
\begin{equation}\label{w066}
4^{-L}	\, \delta^{2-d} \sim \, \left|\left (S(\Lambda^L_1)\right)^\delta\right|  \quad \mbox{if}\
\	0 < \delta \le 3^{-L} \left(\frac{5-\sqrt{3}}{12}\right).
\end{equation}
Here the hidden constants do neither depend on $L \ge 2$ nor on  $\delta \in \Big(0,  \big(\frac{5-\sqrt{3}}{12}\big)\Big)$.\\
{\em Step 2.}
Now we turn to the proof of the claimed inequality  \eqref{w061}.
Let $P=Q_{\ell,k}$ be such that $P \cap (\partial \Omega)^\delta \neq \emptyset$.
For any given integer $\ell \ge 2$, we choose $L \in \mathbb{N}$ such that
\begin{equation}\label{w069}
 \ell \,  \frac{\log 2}{\log 3} \le L <  \ell \,  \frac{\log 2}{\log 3} + 1.
\end{equation}
It follows that $3^L \sim 2^\ell$.
Now we again work with $\Lambda^L_1, \ldots , \Lambda^L_{4^L}$.
There is only a finite number of images $S(\Lambda^L_m)$ which have a nontrivial intersection with $P$, more exactly
\begin{equation}\label{w070}
\sup_{k \in \mathbb{Z}^n}\,  \left|\left\{m\in\{1,\ldots, 4^L\}:~ S(\Lambda^L_m)\cap Q_{\ell,k} \neq \emptyset\right\} \right|\le C <\infty\, ,
\end{equation}
due to \eqref{w067}. Notice that $C$ is independent of $\ell$.
Hence, by Step 1 and \eqref{w069}, one has
\[
\left|(\partial \Omega)^\delta \cap P \right| \ls  \left|\left(S(\Lambda^L_1)\right)^\delta \right|  \sim 4^{-L}\, \delta^{2-d}  \sim 2^{-\ell d}\, \delta^{2-d}
\]
if $0 <\delta \ls 2^{-\ell}$. This finishes the proof of Lemma \ref{schneerand}.
\end{proof}

For the convenience of the reader we  collect what we know about the regularity of ${\bf 1}_\Omega$ in case $\tau =0$,
see Faraco and Rogers \cite{FR} and Sickel \cite{Si21}.

\begin{proposition}\label{schnee}
Let $p\in[1,\infty)$, $q\in[1,\infty]$, and  $s_0:= 2- \log 4/ \log 3$.
\begin{itemize}
\item[{\rm (i)}] The characteristic function  ${\bf 1}_\Omega$ of the snowflake domain $\Omega$ belongs to $B^{s}_{p,\infty} (\mathbb{R}^2)$ if and only if
$s \le   s_0/p$.
\item[{\rm (ii)}] ${\bf 1}_\Omega$  belongs to $F^{s}_{p,q} (\mathbb{R}^2)$ if and only if
$s <   s_0/p$.
\end{itemize}
\end{proposition}

Now we turn to consequences of Theorem \ref{ok2}.

\begin{theorem}\label{schneeflocke}
Let $p\in[1,\infty)$ and   $d:= \log 4/ \log 3$.
Let $\Omega$ be  the snowflake domain.
\begin{itemize}
\item[{\rm (i)}]
The characteristic function  ${\bf 1}_\Omega$
belongs to $B^{(2-d)/p,d/(2p)}_{p,\infty} (\mathbb{R}^2)$.
\item[{\rm (ii)}] Let  $q\in[1,\infty]$, $\tau\in(0, \frac{d}{2p}]$, and  $0 < s < \min
\{\frac 1p , ~ 2 (\frac 1p - \tau)\}$.
Then the  characteristic function  ${\bf 1}_\Omega$  belongs to $F^{s, \tau}_{p,q} (\mathbb{R}^2)$ if and only if $ s< (2-d) /p$.
\end{itemize}
\end{theorem}

\begin{proof}
{\em Step 1.} Proof of (i).
Because of Lemma \ref{schneerand} we can apply Theorem \ref{ok2} and Corollary \ref{allp}.
\\
{\em Step 2.} Proof of (ii).
Sufficiency can be derived from  the elementary embeddings stetd in Remark \ref{grund} and Lemma \ref{support2}.
More complicated to prove is the necessity.
\\
We consider all the dyadic cubes such that
$|P|= 2^{-2 \ell}$ with $\ell\in\mathbb{N}_0$.
We shall use the notation from the proof of Lemma \ref{schneerand}.
Let $L$ be the same as in \eqref{w069}. Then, for each set $S(\Lambda^L_m)$
there exists a dyadic cube covering $\{Q_i\}_{i=1}^N$
such that, for any $i\in\{1,\ldots,N\}$, $\ell(Q_i)=2^{-\ell}$,
$ Q_i\cap  S(\Lambda^L_m) \neq \emptyset$, and
\[
 S(\Lambda^L_m) \subset  \bigcup_{i=1}^N Q_i.
\]
Here $N$ is a natural number satisfying $N \le C$, see \eqref{w070}.
Hence, if $\delta \in(0, 2^{-\ell}]$, one has
\[
 \left(S(\Lambda^L_m)\right)^\delta \subset \bigcup_{i=1}^N (Q_i)^\delta
\subset O :=  \bigcup_{k=1}^{9N} \widetilde Q_{k}
 \]
because each cube $Q_i$ has $8$ dyadic cube neighbors of the same size, and here we use $\{\widetilde Q_k\}_{k=1}^{9N}$
to denote the set of all these neighbors together
 with $\{Q_i\}_{i=1}^N$ themselves.
 Next we choose $k_\delta \in \{1,  \ldots  , 9N\}$ such that
 \[
  \left|\left(S(\Lambda^L_m)\right)^\delta  \cap  \widetilde Q_{k_\delta}\right|
  = \max_{k\in\{1, \ldots , 9N\}} \left|\left(S(\Lambda^L_m)\right)^\delta  \cap \widetilde Q_{k}\right|.
 \]
 This further implies
 \begin{align*}
  \left|\left(S(\Lambda^L_m)\right)^\delta  \cap \widetilde Q_{k_\delta}\right|& \le  \left|\left(S(\Lambda^L_m)\right)^\delta \right|
  \le \sum_{k=1}^{9N} \left|\left(S(\Lambda^L_m)\right)^\delta \cap \widetilde Q_{k}\right|\\
  &\le 9N \, \left|\left(S(\Lambda^L_m)\right)^\delta  \cap \widetilde Q_{k_\delta}\right|.
\end{align*}
Therefore,
 \[
   \left|\left(S(\Lambda^L_m)\right)^\delta  \cap \widetilde Q_{k_\delta}\right| \,
   \sim \,  \left|\left(S(\Lambda^L_m)\right)^\delta\right|  \quad \mbox{if}\  \delta \in(0, 2^{-\ell}],
 \]
where the positive equivalence constants are independent of $m$ and $\ell$.
 We need to remove the dependence of the chosen cube on $\delta$.
 Therefore we argue as follows.
Observe that there exists a smallest  dyadic cube $Q$ such that $O \subset Q$.
 If necessary we apply a shift of the aforementioned snowflake domain, e.g., into the first quadrant such that
 \[
 \inf_{x,y\in \mathbb{R}}\, \min \left\{\dist (\Omega, (x,0)), \dist (\Omega, (0,y))\right\} > 1.  \]
Hence $|Q| \le (9N)^2 \, 2^{-2\ell}$.
 With respect to this cube we know
\[
\left|\left(S(\Lambda^L_m)\right)^\delta  \cap  Q\right| \le  \left|\left(S(\Lambda^L_m)\right)^\delta \right|
\le 9N \, \left|\left(S(\Lambda^L_m)\right)^\delta  \cap \widetilde Q_{k_\delta}\right|
  		\le 9N \, \left|\left(S(\Lambda^L_m)\right)^\delta  \cap  Q\right|.
 \]
Now we can apply \eqref{w066} and obtain
\begin{equation*}
4^{-L}	\, \delta^{2-d} \sim \, \left|\left (S(\Lambda^L_1)\right)^\delta\right|\sim
\left| \left(S(\Lambda^L_1)\right)^\delta \cap Q\right|.
\end{equation*}
After these preparations  Theorem \ref{top} yields the claim.
\end{proof}

It remains to deal with the pointwise multiplier properties
of ${\bf 1}_\Omega$.

\begin{theorem}
Let $\Omega$ be  the snowflake domain and let  $d:= \log 4/ \log 3$.
\begin{itemize}
\item[{\rm (i)}] We have ${\bf 1}_\Omega \in M(B^{s}_{1,1}(\rr^2))$
for all $s \in (0,2-d)$.
\item[{\rm (ii)}] Let $p\in(1,\infty)$. Then we have ${\bf 1}_\Omega \in
M(B^{(2-d)/p}_{p,1}(\rr^2),B^{(2-d)/p}_{p,\infty}(\rr^2)) $.
\item[{\rm (iii)}] Let $p\in(1,\infty)$. Then we have ${\bf 1}_\Omega \in
M(B^{s}_{p,p}(\rr^2)) $ for all $s \in (0,(2-d)/p)$.
\end{itemize}
\end{theorem}

\begin{proof}
Part (i) follows from Theorem \ref{schneeflocke}, Proposition
\ref{FAchar} and Corollary \ref{Gulidual}.
Also part (iii) can be derived from
Theorem \ref{schneeflocke}, Proposition
\ref{FAchar} and Corollary \ref{Gulidual}.
Finally, part (ii) is a consequence of Theorem \ref{Guli3}
and Theorem \ref{schneeflocke}.
\end{proof}

%&&&&&&&&&&&&&&&&&&&&&&&&&&&&&&&&&&&&&&&&&&&&&&&&&&&&&&&&&&&&&&&&&&&&&&&
%&&&&&&&&&&&&&&&&&&&&&&&&&&&&&&&&&&&&&&&&&&&&&&&&&&&&&&&&&&&&&&&&&&&&&&&

\subsection{Characteristic functions of some spiral type domains}
\label{Main9}

%&&&&&&&&&&&&&&&&&&&&&&&&&&&&&&&&&&&&&&&&&&&&&&&&&&&&&&&&&&&&&&&&&&&&&&&&&
%&&&&&&&&&&&&&&&&&&&&&&&&&&&&&&&&&&&&&&&&&&&&&&&&&&&&&&&&&&&&&&&&&&&&&&&&&

As a second example we consider a family of domains lying between two spirals (red and blue) with endpoints in the origin.
Let $\alpha\in(0,\infty)$.
Then $E_\alpha$,  defined in polar coordinates $(r,\varphi)$, is given by
\[
E_\alpha := \left\{(r,\varphi): \ \frac{1}{(\varphi + \pi)^\alpha}
\le r < \frac 1{\varphi^\alpha}, \  2\pi \le \varphi < \infty\right\}.
\]

\smallskip

\begin{minipage}{0.5\textwidth}
	\includegraphics[width=4.6cm]{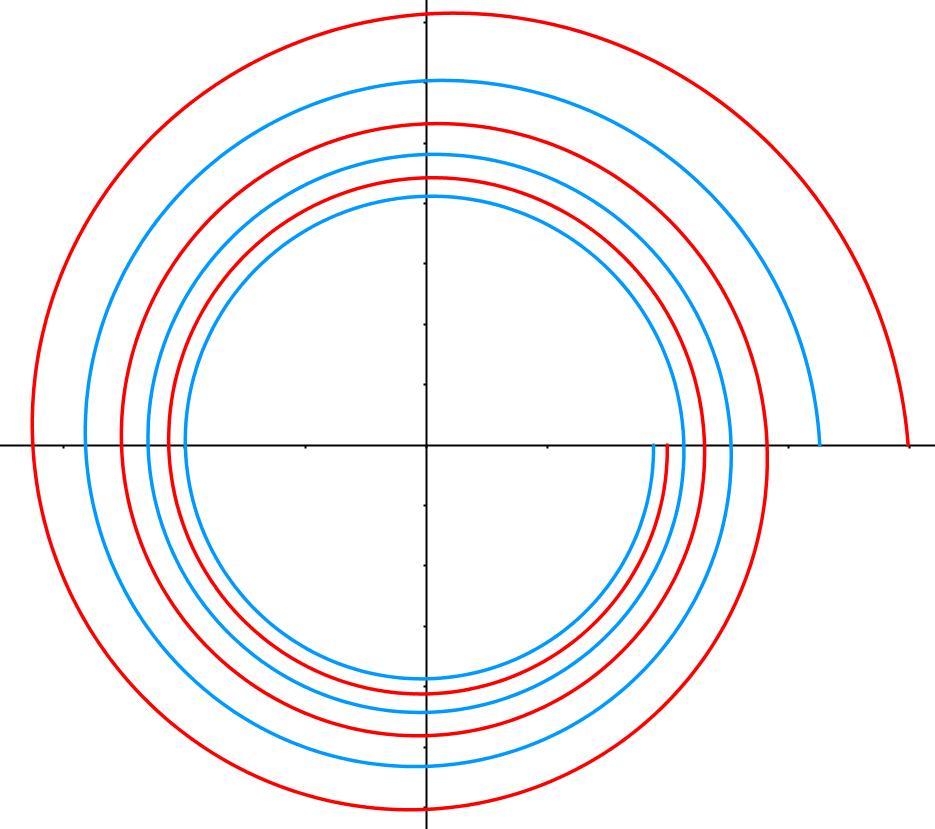}
\end{minipage}\hfill
\begin{minipage}{0.3\textwidth} In the left picture
	a part of the domain $E_{1/2}$ is shown. Elementary arguments show that
	the length of the boundary of $E_\alpha$ is infinite if and only if
$\alpha \in(0,1]$. Below we will calculate the smoothness of the characteristic functions ${\bf 1}_{E_\alpha}$.
\end{minipage}

\smallskip

\begin{theorem}\label{spiral4}
	Let $p \in [1,\infty)$, $\tau \in [0,1/p]$, $ s\in(0,1]$, and $\alpha\in(0,\infty)$.
\begin{itemize}
\item[{\rm (i)}]   Let $ \alpha\in(0, 1)$. Then
	${\bf 1}_{E_\alpha} \in B^{s,\tau}_{p, \infty} (\mathbb{R}^2)$  if and only if
	$s\le \frac{2\alpha}{\alpha+1} \Big(\frac 1p - \tau\Big)$.
\item[{\rm (ii)}] Let $\alpha =1$. Then ${\bf 1}_{E_\alpha} \in B^{s,\tau}_{p, \infty} (\mathbb{R}^2)$ if and only if $ s \le 1/p - \tau$ and $s<1/p$.
\item[{\rm (iii)}]   Let $\alpha\in(1,\infty)$. Then
	${\bf 1}_{E_\alpha} \in B^{s,\tau}_{p, \infty} (\mathbb{R}^2)$  if and only if
		$s\le \frac{2\alpha}{\alpha+1} \Big(\frac 1p - \tau\Big)$
 and	$s\le \frac 1p $.
\end{itemize}
	\end{theorem}

\begin{proof}
 We proceed as in the proof of Theorem \ref{spirale1} and use the characterization by differences
given in Proposition \ref{t4.7}. Thus, we divide the proof into four steps, in which Steps 2 and 3 focus on
the case $p=1$  and Step 4 deals with the case $p\in(1,\fz)$.

{\em Step 1.} Preparations.
For any $k \in \nn$, we put	
	\[
	E_{\alpha,k} := \left\{(r,\varphi):\ \frac{1}{(\varphi + \pi)^\alpha} \le r < \frac {1}{\varphi^\alpha}, \  2\pi k \le \varphi < 2\pi (k+1) \right\}.
	\]
	Observe that the width of $E_{\alpha,k}$ is comparable to $k^{-(1+\alpha)}$ and therefore
	\begin{equation}\label{vol1}
		|	E_{\alpha,k}| \sim \frac{1}{k^{2\alpha +1}}, \ \forall\,  k \in \nn.
	\end{equation}
Furthermore,  there exists a positive constant $c$ such that
\[
|	E_{\alpha,k} \cap Q_{\ell,\textbf{0}}| = 0 \quad \mbox{if}\quad k \le c\, 2^{\ell/\alpha}.
\]

{\em Step 2.} 	Sufficiency in case $p=1$.
 As in the proof of Theorem \ref{spirale1}
 we only consider the dyadic  cubes $P =Q_{\ell,\textbf{0}}$ with $\ell \in \nn_0$.	
As there we shall use the notation
	\begin{align*}
		E_\alpha(P,h) & : =   \left\{x \in E_\alpha \cap P: ~ x+h \not\in E_\alpha \right\},
		\\
		F_\alpha(P,h) & : =  \left\{x \in P: ~x \not\in E_\alpha, ~x+h \in E_\alpha \right\},
	\end{align*}
and, in addition, for any $k \in \nn$,
	\begin{align*}
	E_{\alpha,k} (P,h) & : =  \left\{x \in E_{\alpha,k} \cap P: ~ x+h \not\in E_\alpha \right\},
	\\
	F_{\alpha,k}(P,h)  & : =   \left\{x \in P: ~x \not\in E_{\alpha}, ~x+h \in E_{\alpha,k} \right\}.
\end{align*}	
It follows
\[
\dint_P
|{\bf 1}_{E_\alpha}(x+h)-{\bf 1}_{E_\alpha} (x) |\,dx = |	E_\alpha(P,h)|+  |	F_\alpha(P,h)|
 = \sum_{k=1}^\infty \left[|E_{\alpha,k}(P,h)|+  |	F_{\alpha,k}(P,h)|\right].
\]
Hence
\begin{equation}\label{w054}
\left\| \, \Delta_h^1 {\bf 1}_{E_\alpha} \,\right\|_{L^1(P)}= \sum_{k=c\, 2^{\ell/\alpha}}^\infty \left[ |E_{\alpha,k}(P,h)|+  |F_{\alpha,k}(P,h)|\right].
\end{equation}
Let $|h| \in [2^{-j}, 2^{-j+1}]$ with some $j \ge \ell$.
As in the proof of Theorem \ref{spirale1} above, we have to consider two cases:
\begin{itemize}
\item[(a)] the width of $E_{\alpha,k}$ is less then $|h|$;
\item[(b)] the width of $E_{\alpha,k}$ is larger than or equal to $|h|$.
\end{itemize}

{\bf Case (a)}. Since the width of $E_{\alpha,k}$ is comparable to $k^{-(\alpha+1)}$,  using our restriction concerning $h$, in this case, we only need to
consider $k\gtrsim 2^{j/(\alpha +1)}.$ We will apply the trivial estimate
\begin{equation}\label{w051}
 |	E_{\alpha,k}(P,h)| + |	F_{\alpha,k}(P,h)| \ls |E_{\alpha,k}| \ls  \frac{1}{k^{2\alpha +1}},
\end{equation}
see \eqref{vol1}.

{\bf Case (b)}. By the assumption of this case, we only
consider $c \, 2^{\ell/\alpha} \le k \le C\,  2^{j/(\alpha+1)}.$
This case shows up if and only if $j \ge j_\alpha:= (1+ \frac 1\alpha)\, \ell$
(where we ignored the constants $c,C$ to reduce the technicalities).
Observe that between $E_{\alpha,k}$ and $E_{\alpha,k+1}$ there is a strip of almost the same
size as $E_{\alpha,k}$  but contained in $E^\complement$.
Hence, in this situation we have	
\begin{equation}\label{w052}
	|	E_{\alpha,k} (P,h)| + |F_{\alpha,k} (P,h)| \sim |h| \, k^{-\alpha}.
\end{equation}

Now we are in position to turn to the estimate of $\| \, \Delta_h^1 {\bf 1}_{E_\alpha} \,\|_{L^1(P)}$.
For  $\ell \le j \le j_\alpha$,	we insert \eqref{w051} into \eqref{w054} and obtain
\begin{equation}\label{w055}
\left\| \, \Delta_h^1 {\bf 1}_{E_\alpha} \,\right\|_{L^1(P)} \ls \sum_{k=c\, 2^{\ell/\alpha}}^\infty
\frac{1}{k^{2\alpha +1}} \ls  2^{-2\ell}.
\end{equation}	
In case of $j > j_\alpha $, we apply \eqref{w054} and \eqref{w052} and find
\begin{equation}\label{w056a}
	\left\| \, \Delta_h^1 {\bf 1}_{E_\alpha} \,\right\|_{L^1(P)}  \ls
	\sum_{k=c\, 2^{\ell/\alpha}}^{2^{j/(\alpha+1)}}
	 \frac{|h|}{k^\alpha}  + \sum_{k=2^{j/(\alpha + 1)}}^\infty
	 \frac{1}{k^{2\alpha + 1}} .
\end{equation}	
At this point
we need to distinguish among the three cases  $\alpha \in(0,1)$, $\alpha =1$, and $\alpha\in(1,\fz)$.

\noindent{\em Substep 2.1.} Let $\alpha\in(0,1)$. For any $j > j_\alpha $ the inequality
\eqref{w056a} yields
\[
 \left\| \, \Delta_h^1 {\bf 1}_{E_\alpha} \,\right\|_{L^1(P)}  \ls
	|h| \, 2^{\frac{j(1-\alpha)}{\alpha+1}}
+ 2^{\frac{-2 j \alpha}{\alpha+1}} \sim    2^{\frac{-2 j \alpha}{\alpha+1}}.
\]
This further implies
\begin{equation}\label{w057b}
\sup_{0 < t< 2^{-j_\alpha}}\, 	t^{-s} \dsup_{\frac t2 <|h|<t} \, \dint_P |{\bf 1}_{E_\alpha}(x+h)-{\bf 1}_{E_\alpha} (x) |\,dx
 \ls  \left\{ \begin{array}{lll}
 \infty & \mbox{if}\quad s> \frac{2\alpha}{\alpha +1},
 \\
 1 & \mbox{if}\quad s = \frac{2\alpha}{\alpha +1},
 \\
 2^{j_\alpha (s-\frac{2\alpha}{\alpha + 1})} &
 \mbox{if}\quad s < \frac{2\alpha}{\alpha +1}.
\end{array} \right.
\end{equation}
In case $\ell \le j \le j_\alpha$, from \eqref{w055},  we derive
\begin{equation}\label{w058}
\sup_{2^{-j_\alpha} \le t< 2^{-\ell}}\,
t^{-s}\dsup_{\frac t2 <|h|<t} \, \dint_P
|{\bf 1}_{E_\alpha}(x+h)-{\bf 1}_{E_\alpha} (x) |\,dx \ls  2^{\ell [s (1+\frac1\alpha)-2]}
\end{equation}
because $s >0$.
Both estimates \eqref{w057b} and \eqref{w058} together lead to
\[
{\rm I}\ls \sup_{\ell \in\nn_0} \, 2^{2 \ell \tau} 2^{\ell [s (1+\frac1\alpha)-2]} <\infty
\]
if $2 \tau + s (1+ \frac 1\alpha) -2 \le 0 $ and $s\le \frac{2\alpha}{\alpha + 1}$, where ${\rm I}$ is
the same as in \eqref{XXX}.

\noindent{\em Substep 2.2.} Let $\alpha = 1$. For any $j > j_1=2\ell$ the inequality
\eqref{w056a} yields
\[
 \left\|  \Delta_h^1 {\bf 1}_{E_1} \right\|_{L^1(P)}  \ls
	2^{-j}\, \left(\frac{j}{2} - \ell \right)
+ 2^{-j} \sim    2^{-j}\, \left(\frac{j}{2} - \ell \right).
\]
This further implies that
\begin{equation}\label{w057bb}
\sup_{0 < t< 2^{-2\ell}}\, 	t^{-s} \dsup_{\frac t2 <|h|<t} \, \dint_P
|{\bf 1}_{E_1}(x+h)-{\bf 1}_{E_1} (x) |\,dx
 \ls  \left\{ \begin{array}{lll}
 \infty   & \mbox{if}\quad s\ge 1,
 \\
 2^{2\ell (s-1)}  &
 \mbox{if}\quad s < 1.
\end{array} \right.
\end{equation}
In case $\ell \le j \le j_1= 2\ell$,
from \eqref{w055}, we derive
\begin{equation}\label{w058b}
\sup_{2^{-2\ell} \le t< 2^{-\ell}}\,
t^{-s}\dsup_{\frac t2 <|h|<t} \, \dint_P
|{\bf 1}_{E_{1}}(x+h)-{\bf 1}_{E_1} (x) |\,dx \ls  2^{2\ell (s -1)}
\end{equation}
because $s >0$.
Combining \eqref{w057bb} and \eqref{w058b}, we obtain
\[
{\rm I}\ls \sup_{\ell \in\nn_0} \, 2^{2 \ell \tau} 2^{2\ell (s -1)} <\infty
\]
if $\tau + s  -1 \le 0 $ and $s< 1$.

\noindent{\em Substep 2.3.} Let $\alpha\in(1,\fz)$. For  any $j > j_\alpha$ the inequality
\eqref{w056a} yields
\[
 \left\| \, \Delta_h^1 {\bf 1}_{E_\alpha} \,\right\|_{L^1(P)}  \ls
	2^{-j}\, 2^{-\ell \frac{\alpha-1}{\alpha}}
+ 2^{-j\frac{2\alpha}{\alpha + 1}} \,  .
\]
Observe that $j > j_\alpha$ yields
\[
2^{-j\frac{2\alpha}{\alpha + 1}} < 2^{-j}\, 2^{-\ell \frac{\alpha-1}{\alpha}}\,.
\]
Collecting these estimates we obtain
\begin{equation*}%\label{w057c}
\sup_{0 < t< 2^{-j_\alpha}}\, 	t^{-s} \dsup_{\frac t2 <|h|<t} \, \dint_P
|{\bf 1}_{E_\alpha}(x+h)-{\bf 1}_{E_\alpha} (x) |\,dx
 \ls  \left\{ \begin{array}{lll}
 \infty & \mbox{if}\quad s > 1,
 \\
  2^{-\ell \frac{\alpha-1}{\alpha}}  & \mbox{if}\quad s = 1,
 \\
2^{-\ell \frac{\alpha-1}{\alpha}}
  2^{j_\alpha (s-1)}  &
 \mbox{if}\quad s < 1 .
\end{array} \right.
\end{equation*}
In case $\ell \le j \le j_\alpha $,
from \eqref{w055}, we derive
\begin{equation*}%\label{w058c}
\sup_{2^{-j_\alpha} \le t< 2^{-\ell}}\,
t^{-s}\dsup_{\frac t2 <|h|<t} \, \dint_P
|{\bf 1}_{E_\alpha} (x+h)-{\bf 1}_{E_\alpha} (x) |\,dx \ls  2^{\ell[(1+\frac1\alpha)s -2]}
\end{equation*}
because $s >0$.
Hence, we obtain
\[
{\rm I}\ls \sup_{\ell \in\nn_0 } \, 2^{2 \ell \tau} 2^{\ell (1+ \frac1\alpha)s - 2} <\infty
\]
if $2\tau + (1+ \frac1\alpha) s  -2 \le 0 $ and $s\le  1$.

{\em Step 3.} We show the necessity. The critical case is always given by $j=j_\alpha$.

\noindent{\em Substep 3.1}. Let $ \alpha\in(0, 1)$.
We argue by contradiction.
Therefore we assume that  $2 \tau + s (1+ \frac 1\alpha) -2 > 0 $.
Let $|h|= 2^{-j}$ and choose $k$ such that  $c \, 2^{\ell/\alpha} \le k \le C \, 2^{j/(\alpha + 1)}.$
Based on \eqref{w054} and \eqref{w052} we conclude
\[
\left\| \, \Delta_h^1 {\bf 1}_{E_\alpha} \,\right\|_{L^1(P)} \ge
\sum_{k=c\, 2^{\ell/\alpha}}^{C\, 2^{j/(\alpha + 1)}} \int_{E_{\alpha,k}} |\Delta_h^1 {\bf 1}_{E_\alpha} (x)|\, dx \gtrsim  2^{-j} \,
\sum_{k=c\, 2^{\ell/\alpha}}^{C \, 2^{j/(\alpha + 1)}}
\frac{1}{ k^\alpha} .
\]
Choosing $j=j_\alpha = (1+1/\alpha)\ell $ we find
\[
2^{-j_\alpha}\, \sum_{k=c\, 2^{\ell/\alpha}}^{C \, 2^{j_\alpha/(\alpha + 1)}} \frac{1}{ k^\alpha} \gtrsim
2^{\ell \frac{1-\alpha}{\alpha}}  2^{-\ell (1+1/\alpha)}
\gtrsim
2^{\ell \frac{1-\alpha - (\alpha +1)}{\alpha}} \gtrsim 2^{-2\ell}.
\]
For that reason we obtain
\begin{align*}%\label{w057}
&\dsup_{P\in\mathcal{Q}}  \frac1{|P|^\tau} \, \sup_{0 < t< 2\min\{\ell(P),1\}}
	t^{-s}\dsup_{\frac t2\le|h|<t} \, \dint_P 		|{\bf 1}_E(x+h)-{\bf 1}_E (x) |\,dx
\\
&\quad\gtrsim  	\sup_{\ell\in\nn} \, 2^{2\ell\tau} 2^{\ell [(1+1/\alpha)s-2]}=\infty.\nonumber	
\end{align*}
The used argument is independent of $s$. This proves the necessity of $2 \tau + s (1+ \frac 1\alpha) -2 \le0 $.

Now we turn to the necessity of
$s\le \frac{2\alpha}{\alpha + 1}$.
We assume $s > \frac{2\alpha}{\alpha + 1}$. Then it is enough to apply the argument from \eqref{w057b} with $j=j_\alpha$.

\noindent{\em Substep 3.2}. Let $\alpha = 1$.
We proceed as above.
Therefore we assume that  $\tau + s  -1 > 0 $.
Let $|h|= 2^{-j}$ and let $c \, 2^{\ell} \le k \le C \, 2^{j/2}.$
Based on \eqref{w054} and \eqref{w052} we conclude
\[
\left\| \, \Delta_h^1 {\bf 1}_{E_1} \,\right\|_{L^1(P)} \ge
\sum_{k=c\, 2^{\ell}}^{C\, 2^{\frac j2}} \int_{E_{1,k}} |\Delta_h^1 {\bf 1}_{E_1} (x)|\, dx \gtrsim  2^{-j} \,
\sum_{k=c\, 2^{\ell}}^{C \, 2^{\frac j2}}
\frac{1}{ k} .
\]
Choosing $j=j_1 = 2 \ell $ we find
\[
2^{-2\ell}\, \sum_{k=c\, 2^{\ell}}^{C \, 2^{\ell} } \frac{1}{ k} \gtrsim
2^{-2\ell} .
\]
As above we conclude
\begin{align*}%\label{w057e}
&\dsup_{P\in\mathcal{Q}}  \frac1{|P|^\tau} \, \sup_{0 < t< 2\min\{\ell(P),1\}}
	t^{-s}\dsup_{\frac t2\le|h|<t} \, \dint_P 		|{\bf 1}_{E_1}(x+h)-{\bf 1}_{E_1} (x) |\,dx
\\
&\quad\gtrsim 	\sup_{\ell\in\nn} \, 2^{2\ell\tau} 2^{\ell (2s-2)} =\infty.\nonumber	
\end{align*}
This proves the necessity of $\tau + s  -1 \le 0 $ for all $s$.

To show the necessity of $s>1$ we argue again with choosing $t=|h|= 2^{-j_{\alpha}}$.

\noindent{\em Substep 3.3} Let $\alpha \in(1,\fz)$.  Also in this case  we argue by contradiction.
Therefore we assume that  $2\tau + s\frac{\alpha + 1}{\alpha}  -2 > 0 $.
Let $|h|= 2^{-j}$ and  $c \, 2^{\frac {\ell}{\alpha}} \le k \le C \, 2^{\frac j{\alpha + 1}}.$
Based on \eqref{w052} we conclude
\[
\| \, \Delta_h^1 {\bf 1}_{E_\alpha} \,\|_{L^1(P)} \ge
\sum_{k=c\, 2^{\frac{\ell}{\alpha}}}^{C\, 2^{\frac j{\alpha + 1}}} \int_{E_{\alpha,k}} |\Delta_h^1 {\bf 1}_{E_\alpha} (x)|\, dx \gtrsim  2^{-j} \,
\sum_{k=c\, 2^{\frac{\ell}{\alpha}}}^{C \, 2^{\frac j{\alpha + 1}}}
\frac{1}{ k^\alpha} .
\]
Choosing $j=j_\alpha = (1+\frac1\alpha)\ell $ we find
\[
2^{-j_\alpha}\, \sum_{k=c\, 2^{\frac{\ell}{\alpha}}}^{C \, 2^{\frac{j_\alpha}{\alpha + 1}}} \frac{1}{ k^\alpha} \gtrsim
2^{\ell \frac{1-\alpha}{\alpha}}  2^{-\ell (1+\frac1\alpha)}
\gtrsim
2^{\ell \frac{1-\alpha - (\alpha +1)}{\alpha}} \gtrsim 2^{-2\ell}.
\]
Similarly as above we obtain
\begin{align*}%\label{w057d}
&\dsup_{P\in\mathcal{Q}}  \frac1{|P|^\tau} \, \sup_{0 < t< 2\min(\ell(P),1)}
	t^{-s}\dsup_{\frac t2\le|h|<t} \, \dint_P 		|{\bf 1}_{E_1}(x+h)-{\bf 1}_{E_1} (x) |\,dx
\\
&\quad\gtrsim 	\sup_{\ell\in{\mathbb N}} \, 2^{2\ell\tau} 2^{\ell[(1+\frac1\alpha)s-2]}=\infty.\nonumber	
\end{align*}
This proves the necessity of $2\tau + s\frac{\alpha + 1}{\alpha}  -2 \le  0 $
for all $s$.
Concerning the necessity of $s\le \frac{2\alpha}{\alpha +1}$ we argue by working with $j=j_\alpha$.

{\em Step 4.} Let $p\in (1,\infty)$. Then Corollary \ref{t4.8} yields that ${\bf 1}_E \in B^{sp,\tau p}_{1, \infty} (\rr^2) $ implies ${\bf 1}_E \in B^{s,\tau }_{p, \infty} (\rr^2) $.
From  $\tau p +sp -1 \le 0$ and $s p\le \frac{2\alpha}{\alpha+1}$ we derive $ \tau   + s - \frac 1p \le 0$ and
 $s\le \frac{2\alpha}{\alpha+1}\, \frac 1p$.
This then finishes the proof of Theorem \ref{spiral4}.
\end{proof}

\begin{remark}
A detailed inspection of Theorem 	\ref{spiral4} shows that the regularity of ${\bf 1 }_{E_\alpha}$ increases with $\alpha$, whereas the length of the boundary decreases with $\alpha$. In case $\alpha =1$ we are almost in the same situation as in Theorem  \ref{spirale1}.
But the tiny difference in the definition of the domains by a logarithmic factor yields a difference of the regularity in the limiting situation. Whereas  ${\bf 1}_{E_1} \not\in B^{1/p,0}_{p, \infty} (\mathbb{R}^2)$ we know that
 ${\bf 1}_{E} \in B^{1/p,0}_{p, \infty} (\mathbb{R}^2)$, see  Theorem  \ref{spirale1}.
\end{remark}

Let us have a short look onto the  pointwise multiplier properties of ${\bf 1}_E$.

\begin{theorem}\label{multi3}
	Let $1\le p <\infty$, $0<q\le \infty$, $0 < \alpha < \infty$ and $s\in \rr$. For $s\neq 0$  the function ${\bf 1}_{E_\alpha}$ does not belong to $M(B^s_{p,q}(\rr^2))$.
\end{theorem}

\begin{proof}
The proof uses the same arguments as that one of Theorem \ref{multi2} and is therefore skipped.
\end{proof}

Next we turn to the regularity of ${\bf 1}_{E_\alpha} $ in the framework of Triebel--Lizorkin--Morrey spaces.
We need some preparations.

\begin{lemma}\label{volume}
\begin{itemize}
\item[{\rm (i)}] 	For any $\delta\in(0,1)$ one has
	\[
	\left|(\partial E_\alpha)^\delta\right| \, \sim \, \left\{ \begin{array}{lll}
		\delta^{\frac{2\alpha}{\alpha +1}} & \quad & \mbox{if}\quad  \alpha \in(0, 1),\\
		- \delta\, \log \delta & \quad & \mbox{if}\quad \alpha = 1,\\
		\delta & \quad & \mbox{if}\quad   \alpha \in(1, \infty).		
	\end{array} \right.
	\]
	\item[{\rm (ii)}]	Let $P=Q_{\ell,\textbf{0}}$ for some
$\ell \in \mathbb{N}_0$. In case $2^{-\ell(1+1/\alpha)} \le \delta \le 2^{-\ell}$ one has
	\begin{equation*}
		\left|(\partial E_\alpha)^\delta \cap P\right| \, \sim \, |P| \, .
	\end{equation*}
	In case $0<\delta <  2^{-\ell(1+1/\alpha)} $ it holds
	\begin{equation*}
		\left|(\partial E_\alpha)^\delta \cap P\right| \, \sim \, \left\{ \begin{array}{lll}
			\delta^{\frac{2\alpha}{\alpha +1}} & \quad & \mbox{if}\quad \alpha \in(0, 1),\\
			\delta\, |\log \delta +  \ell  | & \quad & \mbox{if}\quad \alpha = 1,\\
			\delta \, 2^{\ell \frac{1-\alpha}{\alpha}} & \quad & \mbox{if}\quad \alpha \in(1, \infty).		
		\end{array} \right.
	\end{equation*}
\end{itemize}
Here, all positive equivalence constants are independent of $\delta$.
\end{lemma}

\begin{proof}
	{\em Step 1.} Proof of (i). Let $\delta\in(0,1)$.
	In case that the width of $E_{\alpha,k}$ is larger than $\delta$ we have the obvious estimate
	\begin{equation*}%\label{w070b}
		\left|(\partial E_{\alpha, k})^\delta\right| \sim \, \delta \, k^{-\alpha} \, .
	\end{equation*}
	In the remaining case, i.e., if $k^{-(1+\alpha)} <\delta $ we observe that
	\[
	(\partial E_{\alpha,k})^\delta \subset B\left(\textbf{0}, C\, \delta^{\frac{\alpha}{\alpha+1}}\right)
	\]	 	
	for some positive constant $C$ independent of $\delta$ and $k$. In addition we have
	\[
	B\left(\textbf{0}, c\, \delta^{\frac{\alpha}{\alpha+1}}\right) \subset \bigcup_{\{k\in\nn:~k^{-(1+\alpha)\}} <\delta\}}
	(\partial E_{\alpha,k})^\delta
	\]
	for some positive constant $c$.
	This further implies
	\begin{align*}
		\left|(\partial E_{\alpha})^\delta\right| & \le   \sum_{k=1}^\infty 	\left|(\partial E_{\alpha, k})^\delta\right|
\le
		\sum_{k=1}^{\delta^{-\frac{1}{\alpha+1}}} 	\left|(\partial E_{\alpha, k})^\delta\right| + \left|B\left(\textbf{0}, C\, \delta^{\frac{\alpha}{\alpha+1}}\right)\right|\\
		&\ls
		\sum_{k=1}^{\delta^{-\frac{1}{\alpha+1}}} 	 \frac{\delta}{k^{\alpha}}  +
		\delta^{\frac{2\alpha}{\alpha+1}}\, .
	\end{align*}	
	For a further discussion we need to distinguish into the cases $ \alpha \in(0,1)$, $\alpha =1$, and
	$\alpha \in(1, \infty)$.\\
	{\em Case  $ \alpha \in(0,1)$.} Since
	\[
	\sum_{k=1}^{\delta^{-\frac{1}{\alpha+1}}} 	 \frac{\delta}{k^{\alpha}} \, \sim \, \delta^{\frac{2\alpha}{\alpha+1}},
	\]
	we conclude that
	\begin{equation*}%\label{w071}
		\left|(\partial E_{\alpha})^\delta\right| \,  \sim \,  \delta^{\frac{2\alpha}{\alpha+1}}\, .
	\end{equation*}
	Here we used \eqref{w070}.
	\\
	{\em Case $\alpha =1$.} The same arguments as used in the previous case yield
	\begin{equation*}%\label{w072}
		|(\partial E_{\alpha})^\delta| \,  \sim \,  - \delta \, \log \delta\, .
	\end{equation*}
	{\em Case $\alpha \in(1, \infty)$.}
	This time we find
	\[
	\sum_{k=1}^{\delta^{-\frac{1}{\alpha+1}}} 	 \frac{\delta}{k^{\alpha}}  +
	\delta^{\frac{2\alpha}{\alpha+1}} \, \sim \delta \, .
	\]
Altogether we complete the proof of (i).

	{\em Step 2.} Proof of (ii).
	Let $\delta = 2^{-j}$ for some $j \ge \ell$.
	If $j\ge j_\alpha$, then
	\[
	2^{-j}\, \sum_{k=2^{\frac\ell\alpha}}^{2^{\frac{j}{\alpha+1}}} 	 \frac{1}{k^{\alpha}}  +
	2^{\frac{-2\alpha j}{\alpha+1}} \, \sim
	\left\{ \begin{array}{lll}
		2^{\frac{-2\alpha j}{\alpha +1}} & \  & \mbox{if}\  \alpha\in(0, 1),\\
		2^{-j}\,\left( \frac{j}{2} -  \ell  \right) & \  & \mbox{if}\  \alpha = 1,\\
		2^{-j}  \, 	2^{\ell \, \frac{1-\alpha}{\alpha}}& \  & \mbox{if}\   \alpha \in(1, \infty).		
	\end{array} \right.
	\]
	If $\ell \le j\le j_\alpha$,  we obtain
	\[	\left|(\partial E_{\alpha})^\delta \cap P\right| \,  \sim \, 	|P|
	\]
because the width of the $E_{\alpha,k} \subset P$, defined in Step 1 of the proof of
Theorem \ref{spiral4},	is less than $\delta$. 	This proves the desired result of (ii)  and hence finishes the proof of
Lemma \ref{volume}.
\end{proof}

Mainly as a consequence of Theorems \ref{spiral4} and \ref{top} we obtain the following assertions.

\begin{theorem}\label{spiral5}
	Let $p \in [1,\infty)$, $q \in [1,\infty]$, $\tau \in [0,1/p)$, $ s\in(0,1]$, and $\alpha\in(0,\infty)$.
Then the following assertions hold.
	\begin{itemize}
\item[{\rm (i)}]  $E_\alpha$ is  a bounded thick domain.		
\item[{\rm (ii)}]
The function ${\bf 1}_{E_\alpha} $ does not belong to any space
		$F^{1/p,\tau}_{p,q} (\mathbb{R}^2)$.
\item[{\rm (ii)}]   Let  $0 <  s < \min\{\frac 1p , ~ 2(\frac 1p - \tau)\}$.
\\
{\rm (a)} Let $ \alpha\in(0, 1)$. Then
${\bf 1}_{E_\alpha} \in F^{s,\tau}_{p,q} (\mathbb{R}^2)$  if and only if
$2\tau + s\, \frac{\alpha + 1}{\alpha} \le \frac 2p$ and
$s < \frac{2\alpha}{\alpha+1} \, \frac 1p$.
\\
{\rm (b)} Let $\alpha =1$. Then ${\bf 1}_{E_\alpha} \in F^{s,\tau}_{p,q} (\mathbb{R}^2)$ if and only if $\tau + s \le 1/p$ and $s < \frac 1p $.
\\
{\rm (c)}   Let $\alpha\in(1,\infty)$. Then
${\bf 1}_{E_\alpha} \in F^{s,\tau}_{p, q} (\mathbb{R}^2)$  if and only if
$2\tau + s\, \frac{\alpha + 1}{\alpha} \le \frac 2p$ and
$s < \frac 1p $.
\end{itemize}
\end{theorem}

\begin{proof}
For spiral-type domains it is obvious that they are bounded and thick.

Next we employ Remark \ref{grund} and 	Theorem \ref{spiral4} and conclude that the following assertions hold:
\begin{itemize}
	\item[{\rm (i)}]   Let $ \alpha\in(0, 1)$. Then
	${\bf 1}_{E_\alpha} \in F^{s,\tau}_{p,q} (\mathbb{R}^2)$  if
	$2\tau + s\, \frac{\alpha + 1}{\alpha} < \frac 2p$ and
	$s < \frac{2\alpha}{\alpha+1} \, \frac 1p$.
	\item[{\rm (ii)}] Let $\alpha =1$. Then ${\bf 1}_{E_\alpha} \in F^{s,\tau}_{p, q} (\mathbb{R}^2)$ if $\tau + s < \frac1p$ and $s<1/p$.
	\item[{\rm (iii)}]   Let $\alpha\in(1,\infty)$. Then
	${\bf 1}_{E_\alpha} \in F^{s,\tau}_{p,q} (\mathbb{R}^2)$  if	$2\tau + s\, \frac{\alpha + 1}{\alpha} < \frac 2p$ and
	$s < \frac 1p $.
\end{itemize}	

Hence, it remains to deal with certain limiting situations.
For spiral-type domains it is obvious that they are bounded and thick.
Next we employ Remark \ref{grund} and 	Theorem \ref{spiral4} and conclude that the following assertions hold:
\begin{itemize}
	\item[{\rm (i)}]   Let $ \alpha\in(0, 1)$. Then
	${\bf 1}_{E_\alpha} \in F^{s,\tau}_{p,q} (\rr^2)$  if
	$2\tau + s\, \frac{\alpha + 1}{\alpha} < \frac 2p$ and
	$s < \frac{2\alpha}{\alpha+1} \, \frac 1p$.
	\item[{\rm (ii)}] Let $\alpha =1$. Then ${\bf 1}_{E_\alpha} \in F^{s,\tau}_{p, q} (\rr^2)$ if $\tau + s < \frac1p$ and $s<1/p$.
	\item[{\rm (iii)}]   Let $\alpha\in(1,\fz)$. Then
	${\bf 1}_{E_\alpha} \in F^{s,\tau}_{p,q} (\rr^2)$  if	$2\tau + s\, \frac{\alpha + 1}{\alpha} < \frac 2p$ and
	$s < \frac 1p $.
\end{itemize}	

Hence, it remains to deal with certain limiting situations.
First, we exclude the case $s=1/p$.
It is immediate that
${\bf 1}_{E_\alpha} \not \in F^{s,\tau}_{p,q} (\rr^2)$.
This follows from ${\bf 1}_{E_\alpha} \not \in F^{1/p,0}_{p,q} (\rr^2)$ (see Proposition \ref{max})  and Lemma \ref{support2}.

For the remaining limiting cases we employ Theorem \ref{top}. Observe that compared to Theorem \ref{spiral4}
this will cause some extra restrictions, namely $\tau < \frac1p$ and $0 <  s < 2(\frac 1p - \tau)$.
We need to check whether the supremum
\[
{\rm I}:= \sup_{\{P\in\mathcal{Q}, |P|\le 1\}}\, \frac{1}{|P|^\tau}\, 	\left[ \int_{0}^{\sqrt{n} \, \ell (P)} \delta^{-sp}\,
\left|(\partial E_\alpha)^{\delta} \cap P\right| \, \frac{d\delta}{\delta}\right]^{\frac1p}
\]
is finite or not. As done several times before we will concentrate on $P= Q_{\ell,\textbf{0}}$, $\ell \in \nn_0$.
\\
{\em Step 1.} Let $\alpha\in(0, 1)$.

\noindent{\em Substep 1.1}. We assume that $s \ge \frac{2\alpha}{\alpha+1} \, \frac 1p$.
Then
Lemma \ref{volume} yields
\[
\int_{0}^{2^{-\ell(1+\frac1\alpha)}} \delta^{-sp}\,
\left|(\partial E_\alpha)^{\delta} \cap Q_{\ell, \textbf{0}}\right| \, \frac{d\delta}{\delta} \sim
\int_{0}^{2^{-\ell(1+\frac1\alpha)}} \delta^{-sp}\,
\delta^{\frac{2\alpha}{\alpha + 1}} \, \frac{d\delta}{\delta} = \infty \, .
\]
Hence, ${\bf 1}_{E_\alpha} \not \in F^{s,\tau}_{p,q} (\rr^2)$  if
$s \ge \frac{2\alpha}{\alpha+1} \, \frac 1p$.

\noindent{\em Substep 1.2}. We assume that $s < \frac{2\alpha}{\alpha+1} \, \frac 1p$ but
	$2\tau + s\, \frac{\alpha + 1}{\alpha} = \frac 2p$.
Applying Lemma \ref{volume} we obtain
\[
\left[\int_{0}^{2^{-\ell(1+\frac1\alpha)}} \delta^{-sp}\,
\left|(\partial E_\alpha)^{\delta} \cap Q_{\ell, \textbf{0}}\right| \, \frac{d\delta}{\delta}\right]^{\frac1p}
\sim 2^{\ell [(1+\frac1\alpha)s -\frac 2p]}
\]
and
\[
\left[\int^{2^{-\ell}}_{2^{-\ell(1+\frac1\alpha)}} \delta^{-sp}\,
\left|(\partial E_\alpha)^{\delta} \cap Q_{\ell, \textbf{0}}\right| \, \frac{d\delta}{\delta}\right]^{\frac1p}
\sim 2^{\ell [(1+\frac1\alpha)s -\frac 2p]}   .
\]
Consequently  ${\bf 1}_{E_\alpha}  \in F^{s,\tau}_{p,q} (\rr^2)$ under these conditions.

\noindent{\em Substep 1.3.} Necessity.  We assume that $s < \frac{2\alpha}{\alpha+1} \, \frac 1p$ and
$2\tau + s\, \frac{\alpha + 1}{\alpha} > \frac 2p$.
Then 	Theorem \ref{spiral4}(i) yields
${\bf 1}_{E_\alpha}  \not\in B^{s,\tau}_{p,\infty} (\rr^2)$
and therefore, by elementary embeddings in Remark \ref{grund},
${\bf 1}_{E_\alpha}  \not\in F^{s,\tau}_{p,q} (\rr^2)$.\\
{\em Step 2.} Let $\alpha = 1$.
We assume that $\tau + s = 1/p$ and $s<1/p$.
By using  Lemma \ref{volume} we find
\[
\left[\int_{0}^{2^{-2\ell}} \delta^{-sp}\,
\left|(\partial E_1)^{\delta} \cap Q_{\ell, \textbf{0}}\right| \, \frac{d\delta}{\delta}\right]^{\frac1p}
\sim \ell^{1/p}\, 2^{2\ell (s -\frac 1p)}
\]
and
\[
\left[\int^{2^{-\ell}}_{2^{-2\ell}} \delta^{-sp}\,
\left|(\partial E_1)^{\delta} \cap Q_{\ell, \textbf{0}}\right| \, \frac{d\delta}{\delta}\right]^{\frac1p}
\sim 2^{\ell (s -\frac 2p)}.
\]
This further implies that ${\bf 1}_{E_\alpha}  \in F^{s,\tau}_{p,q} (\rr^2)$
if and only if $s\le 2( \frac 1p - \tau)$.
But $ 0 < s = \frac 1p - \tau$ implies $s < 2s = 2( \frac 1p - \tau)$.
If we assume  $\tau + s > \frac1p$ we can argue as in Substep 1.3.
\\
{\em Step 3}. Let $\alpha \in(1, \infty)$.
We assume that  $2\tau + s\, (1+\frac1\alpha) = \frac 2p$. Then
Lemma \ref{volume} yields
\[
\left[\int_{0}^{2^{-2\ell}} \delta^{-sp}\,
\left|(\partial E_1)^{\delta} \cap Q_{\ell, \textbf{0}}\right| \, \frac{d\delta}{\delta}\right]^{\frac1p}
\sim  2^{\ell [(1+\frac1\alpha)s -\frac 2p]}
\]
and
\[
\left[\int^{2^{-\ell}}_{2^{-2\ell}} \delta^{-sp}\,
\left|(\partial E_1)^{\delta} \cap Q_{\ell, \textbf{0}}\right| \, \frac{d\delta}{\delta}\right]^{\frac1p}
\sim  2^{\ell[(1+\frac1\alpha)s -\frac 2p]}.
\]
Consequently
${\bf 1}_{E_\alpha}  \in F^{s,\tau}_{p,q} (\rr^2)$ under these restrictions.
If we assume that  $2\tau + s\, (1+\frac1\alpha) > \frac 2p$, then Theorem \ref{spiral4} yields ${\bf 1}_{E_\alpha}  \not\in B^{s,\tau}_{p,\infty} (\rr^2)$
and therefore, by elementary embeddings in Remark \ref{grund},
${\bf 1}_{E_\alpha}  \not\in F^{s,\tau}_{p,q} (\rr^2)$.
This finishes the proof of Theorem \ref{spiral5}.
\end{proof}

\begin{remark}
 There is only a tiny difference between the B-case and the F-case. In the cases $\alpha \neq 1$
 we have the additional inequalities $s < \frac{2\alpha}{\alpha+1} \, \frac 1p$ and $s<\frac1p$, respectively.
\end{remark}

\smallskip

\noindent\textbf{Acknowledgements}\quad
The authors   would   like to thank the anonymous referees for the excellent job they have done. Based on their  valuable remarks and suggestions we have been able to seriously  improve  our manuscript.

\bigskip

\smallskip

\noindent Wen Yuan and  Dachun Yang

\smallskip

\noindent Laboratory of Mathematics and Complex Systems
(Ministry of Education of China),
School of Mathematical Sciences, Beijing Normal University,
Beijing 100875, The People's Republic of China

\smallskip

\noindent {\it E-mails}:   \texttt{dcyang@bnu.edu.cn} (D. Yang)

\noindent\phantom{{\it E-mails:}} \texttt{wenyuan@bnu.edu.cn} (W. Yuan)

\bigskip

\noindent Winfried Sickel (Corresponding author)

\medskip

\noindent Mathematisches Institut, Friedrich-Schiller-Universit\"at Jena,
Jena 07743, Germany

\smallskip

\noindent {\it E-mail}: \texttt{winfried.sickel@uni-jena.de}

\end{document}